%% file: pRH.tex
\documentclass[letterpaper,11pt,final]{article}%
\usepackage{titlesec,titletoc}
\usepackage[overload]{textcase}
\usepackage{amscd}
\usepackage[overload]{textcase}
\usepackage{etex}
\usepackage{amsxtra}
\usepackage{graphicx}
\usepackage[amsmath,hyperref,thmmarks]{ntheorem}
\usepackage{natbib}
\usepackage{amsmath}
\usepackage{amsfonts}
\usepackage{amssymb}
\usepackage{natbib}
\usepackage[colorlinks=true,bookmarksnumbered=true,pdfpagemode=None]{hyperref}%
\providecommand{\U}[1]{\protect\rule{.1in}{.1in}}
\theoremnumbering{arabic}
\theoremheaderfont{\scshape}
\RequirePackage{latexsym}
\theorembodyfont{\slshape}
\theoremseparator{.}
\newtheorem{X}{X}[section]

\newtheorem{conjecture}[X]{Conjecture}
\newtheorem{corollary}[X]{Corollary}
\newtheorem{E}[X]{}
\newtheorem{lemma}[X]{Lemma}
\newtheorem{proposition}[X]{Proposition}
\newtheorem{theorem}[X]{Theorem}
\newtheorem{mlemma}[X]{Main Lemma}
\theorembodyfont{\upshape}
\newtheorem{definition}[X]{Definition}
\newtheorem{example}[X]{Example}

\newtheorem{plain}[X]{}

\newtheorem{aside}[X]{Aside}
\newtheorem{summary}[X]{Summary}
\theorembodyfont{\small}
\newtheorem*{note}{Notes}

\theorembodyfont{\normalsize}
\theoremstyle{nonumberplain}
\theoremsymbol{\ensuremath{_\Box}}
\newtheorem{proof}{Proof}

\qedsymbol{\ensuremath{_\Box}}
\input{defa}

\input{defs}

\makeindex
\hypersetup{pdftitle={The $p$ Riemann Hypothesis},
pdfauthor={James S. Milne},
pdfkeywords={Riemann hypothesis}}
\begin{document}

\title{The Riemann Hypothesis over Finite Fields\\\vspace{0.3\baselineskip} {\Large From Weil to the Present Day}}
\author{James S. Milne}
\date{\today}
\maketitle

\begin{abstract}
The statement of the Riemann hypothesis makes sense for all global fields, not
just the rational numbers. For function fields, it has a natural restatement
in terms of the associated curve. Weil's work on the Riemann hypothesis for
curves over finite fields led him to state his famous \textquotedblleft Weil
conjectures\textquotedblright, which drove much of the progress in algebraic
and arithmetic geometry in the following decades.

In this article, I describe Weil's work and some of the ensuing progress: Weil 
cohomology (\'etale, crystalline); Grothendieck's standard conjectures; motives; 
Deligne's proof; Hasse-Weil zeta functions and Langlands functoriality.

 It is my contribution to the book \textquotedblleft The legacy of
Bernhard Riemann after one hundred and fifty years\textquotedblright,
edited by S.T. Yau et al.

\end{abstract}
\tableofcontents

\bigskip A global field $K$ of characteristic $p$ is a field finitely
generated and of transcendence degree $1$ over $\mathbb{F}{}_{p}$. The field
of constants $k$ of $K$ is the algebraic closure of $\mathbb{F}{}_{p}$ in $K$.
Let $x$ be an element of $K$ transcendental over $k$. Then%
\[
\zeta(K,s)\overset{\textup{{\tiny def}}}{=}\prod\nolimits_{\mathfrak{p}{}%
}\frac{1}{1-N\mathfrak{p}{}^{-s}}%
\]
where $\mathfrak{p}{}$ runs over the prime ideals of the integral closure
$\mathcal{O}{}_{K}$ of $k[x]$ in $K$ and $N\mathfrak{p}{}=\left\vert
\mathcal{O}{}_{K}/\mathfrak{p}{}\right\vert $; we also include factors for
each of the finitely many prime ideals $\mathfrak{p}{}$ in the integral
closure of $k[x^{-1}]$ not corresponding to an ideal of $\mathcal{O}{}_{K}$.
The Riemann hypothesis for $K$ states that the zeros of $\zeta(K,s)$ lie on
the line $\Re(s)=1/2.$

Let $C$ be the (unique) nonsingular projective over $k$ with function field is
$K$. Let $k_{n}$ denote the finite field of degree $n$ over $k$, and let
$N_{n}$ denote the number of points of $C$ rational over $k_{n}$:%
\[
N_{n}\overset{\textup{{\tiny def}}}{=}\left\vert C(k_{n})\right\vert .
\]
The zeta function of $C$ is the power series $Z(C,T)\in\mathbb{Q}{}[[T]]$ such that%

\[
\frac{d\log Z(C,T)}{dT}=\sum\nolimits_{n=1}^{\infty}N_{n}T^{n-1}.
\]
This can also be expressed as a product%
\[
Z(C,T)=\prod\nolimits_{P}\frac{1}{1-T^{\deg(P)}}%
\]
where $P$ runs over the closed points of $C$ (as a scheme). Each $P$
corresponds to a prime ideal $\mathfrak{p}{}$ in the preceding paragraph, and
$N\mathfrak{p}{}=q^{\deg(P)}$ where $q=|k|$. Therefore%
\[
\zeta(K,s)=Z(C,q^{-s})\text{.}%
\]

Let $g$ be the genus of $C$. Using the Riemann-Roch theorem, one finds that%
\[
Z(C,T)=\frac{P(C,T)}{(1-T)(1-qT)}%
\]
where
\[
P(C,T)=1+c_{1}T+\cdots+c_{2g}T^{2g}\in\mathbb{Z}{}[T];
\]
moreover $Z(C,T)$ satisfies the functional equation%
\[
Z(C,1/qT)=q^{1-g}\cdot T^{2-2g}\cdot Z(C,T).
\]
Thus $\zeta(K,s)$ is a rational function in $q^{-s}$ with simple poles at
$s=0$, $1$ and zeros at $a_{1},\ldots,a_{2g}$ where the $a_{i}$ are the
inverse roots of $P(C,T)$, i.e.,%
\[
P(C,T)=\prod\nolimits_{i=1}^{2g}(1-a_{i}T),
\]
and it satisfies a functional equation relating $\zeta(K,s)$ and
$\zeta(K,1-s)$. The Riemann hypothesis now asserts that%
\[
\left\vert a_{i}\right\vert =q^{\frac{1}{2}},\quad i=1,\ldots,2g.
\]
Note that%
\[
\renewcommand{\arraystretch}{1.5}\log Z(C,T)=\left\{
\begin{array}
[c]{l}%
\sum\nolimits_{1}^{\infty}N_{n}\frac{T^{n}}{n}\\
\log\frac{(1-a_{1}T)\cdots(1-a_{2g}T)}{(1-T)(1-qT)},
\end{array}
\right.
\]
and so%
\[
N_{n}=1+q^{n}-(a_{1}^{n}+\cdots+a_{2g}^{n}).
\]
Therefore, the Riemann hypothesis implies that%
\[
\left\vert N_{n}-q^{n}-1\right\vert \leq2g\cdot(q^{n})^{1/2};
\]
conversely, if this inequality holds for all $n$, then the Riemann hypothesis
holds for $C$.\footnote{The condition implies that the power series%
\[
\sum\nolimits_{i=1}^{2g}\frac{1}{1-a_{i}z}=\sum\nolimits_{n=0}^{\infty}\left(
\sum\nolimits_{i=1}^{2g}a_{i}^{n}\right)  z^{n}%
\]
converges for $|z|\leq q^{-1/2}$, and so $|a_{i}|\leq q^{1/2}$ for all $i$.
The functional equation shows that $q/a_{i}$ is also a root of $P(C,T)$, and
so $\left\vert a_{i}\right\vert =q^{1/2}$.}

The above is a brief summary of the results obtained by the German school of
number theorists (Artin, F.K. Schmidt, Deuring, Hasse, ....) by the mid-1930s.
In the 1930s Hasse gave two proofs of the Riemann hypothesis for curves of
genus $1$, the second of which uses the endomorphism ring of the elliptic
curve\footnote{F.K. Schmidt (1931) showed that every curve over a finite field
has a rational point. The choice of such a point on a curve of genus $1$ makes
it an elliptic curve.} in an essential way (see the sketch below). Deuring
recognized that for curves of higher genus it was necessary to replace the
endomorphism ring of the curve with its ring of correspondences in the sense
of Severi 1903, and he wrote two articles reformulating part of Severi's
theory in terms of \textquotedblleft double-fields\textquotedblright\ (in
particular, extending it to all characteristics). Weil was fully aware of
these ideas of Deuring and Hasse (Schappacher 2006). For a full account of
this early work, see the article by Oort and Schappacher in this volume.

\subsubsection{The proof of the Riemann hypothesis for elliptic curves}

For future reference, we sketch the proof of the Riemann hypothesis for
elliptic curves. Let $E$ be such a curve over a field $k$, and let $\alpha$ be
an endomorphism of $E$. For $\ell\neq\mathrm{char}(k)$, the $\mathbb{Z}%
{}_{\ell}$-module $T_{\ell}E\overset{\textup{{\tiny def}}}{=}\varprojlim
_{n}E_{\ell^{n}}(k^{\mathrm{al}})$ is free of rank $2$ and $\alpha$ acts on it
with determinant $\deg\alpha$. Over $\mathbb{C}{}$ this statement can be
proved by writing $E=\mathbb{C}{}/\Lambda$ and noting that $T_{\ell}%
E\simeq\mathbb{Z}{}_{\ell}\otimes_{\mathbb{Z}{}}\Lambda$. Over a field of
nonzero characteristic the proof is more difficult --- it will be proved in a
more general setting later (\ref{r15}, \ref{e9}).

Now let $\alpha$ be an endomorphism of $E$, and let $T^{2}+cT+d$ be its
characteristic polynomial on $T_{\ell}E$:%
\[
T^{2}+cT+d\overset{\textup{{\tiny def}}}{=}\det(T-\alpha|T_{\ell}E).
\]
Then,%
\begin{align*}
d  &  =\det(-\alpha|T_{\ell}E)=\deg(\alpha)\\
1+c+d  &  =\det(\id_{E}-\alpha|T_{\ell}E)=\deg(\id_{E}-\alpha),
\end{align*}
and so $c~$and $d$ are integers independent of $\ell$.

For an integer $n$, let%
\begin{equation}
T^{2}+c^{\prime}T+d^{\prime}=\det(T-n\alpha|T_{\ell}E)\text{.} \label{e2}%
\end{equation}
Then $c^{\prime}=nc$ and $d^{\prime}=n^{2}d$. On substituting $m$ for $T$ in
(\ref{e2}), we find that%
\[
m^{2}+cmn+dn^{2}=\det(m-n\alpha|T_{\ell}E).
\]
The right hand side is the degree of the map $m-n\alpha$, which is always
nonnegative, and so the discriminant $c^{2}-4d\leq0$, i.e., $c^{2}\leq4d$.

We apply these statements to the Frobenius map $\pi\colon(x_{0}\colon
x_{1}\colon\ldots)\mapsto(x_{0}^{q}\colon x_{1}^{q}\colon\ldots)$. The
homomorphism $\id_{E}-\pi\colon E\rightarrow E$ is \'{e}tale and its kernel on
$E(k^{\mathrm{al}})$ is $E(k)$, and so its degree is $|E(k)|$. Let
$f=T^{2}+cT+d$ be the characteristic polynomial of $\pi$. Then $d=\deg(\pi
)=q$, and $c^{2}\leq4d=4q$. From%
\[
\left\vert E(k)\right\vert =\deg(\id_{E}-\pi)=\det(\id_{E}-\pi|T_{\ell
}E)=f(1)=1+c+q,
\]
we see that%
\[
\left\vert |E(k)|-q-1\right\vert =|c|\leq2\,q^{1/2}\text{, }%
\]
as required.

\section{Weil's work in the 1940s and 1950s}

\subsection{The 1940 and 1941 announcements}

Weil announced the proof of the Riemann hypothesis for curves over finite
fields in a brief three-page note (Weil 1940), which begins with the following paragraph:

\begin{quote}
I shall summarize in this Note the solution of the main problems in the theory
of algebraic functions with a finite field of constants; we know that this
theory has been the subject of numerous works, especially, in recent years, by
Hasse and his students; as they have caught a glimpse of, the key to these
problems is the theory of correspondences; but the algebraic theory of
correspondences, due to Severi, is not sufficient, and it is necessary to
extend Hurwitz's transcendental theory to these functions.\footnote{Je vais
r\'{e}sumer dans cette Note la solution des principaux probl\`{e}mes de la
th\'{e}orie des fonctions alg\'{e}briques \`{a} corps de constantes fini; on
sait que celle-ci a fait l'objet de nombreux travaux, et plus
particuli\`{e}rement, dans les derni\'{e}res ann\'{e}es, de ceux de Hasse et
de ses \'{e}\`{l}\`{e}ves; comme ils l'ont entrevu, la th\'{e}orie des
correspondances donne la clef de ces probl\`{e}mes; mais la th\'{e}orie
alg\'{e}brique des correspondances, qui est due \`{a} Severi, n'y suffit
point, et il faut \'{e}tendre \`{a} ces fonctions la th\'{e}orie transcendante
de Hurwitz.}
\end{quote}

\noindent The \textquotedblleft main problems\textquotedblright\ were the
Riemann hypothesis for curves of arbitrary genus, and Artin's conjecture that
the (Artin) $L$-function attached to a nontrivial simple Galois representation
is entire.

\textquotedblleft Hurwitz's transcendental theory\textquotedblright\ refers to
the memoir Hurwitz 1887, which Weil had studied already as a student at the
Ecole Normale Sup\'{e}rieure. There Hurwitz gave the first proof\footnote{As
Oort pointed out to me, Chasles-Cayley-Brill 1864, 1866, 1873, 1874 discussed
this topic before Hurwitz did, but perhaps not \textquotedblleft in terms of
traces\textquotedblright.} of the formula expressing the number of coincident
points of a correspondence $X$ on a complex algebraic curve $C$ in terms of
traces. In modern terms,%
\[
(\text{number of coincidences)}=\Tr(X|H_{0}(C,\mathbb{Q}{}))-\Tr(X|H_{1}%
(C,\mathbb{Q}{}))+\Tr(X|H_{2}(C,\mathbb{Q})).
\]
This can be rewritten%
\begin{equation}
(X\cdot\Delta)=d_{2}(X)-\Tr(X|H_{1}(C))+d_{1}(X) \label{e3}%
\end{equation}
where $d_{1}(X)=(X\cdot C\times\mathrm{pt})$ and $d_{2}(X)=(X\cdot
\mathrm{pt}\times C)$.

Now consider a nonsingular projective curve $C_{0}$ over a field $k_{0}$ with
$q$ elements, and let $C$ be the curve over the algebraic closure $k$ of
$k_{0}$ obtained from $C$ by extension of scalars. For the graph of the
Frobenius endomorphism $\pi$ of $C$, the equality (\ref{e3}) becomes%
\[
(\Gamma_{\pi}\cdot\Delta)=1-``\Tr(\pi|H_{1}(C))\text{\textquotedblright}+q.
\]
As $(\Gamma_{\pi}\cdot\Delta)=|C_{0}(k_{0})|$, we see that everything comes
down to understanding algebraically (and in nonzero characteristic!), the
\textquotedblleft trace of a correspondence on the first homology group\ of
the curve\textquotedblright.

Let $g$ be the genus of $C_{0}$. In his 1940 note,\footnote{In fact, following
the German tradition (Artin, Hasse, Deuring,\ldots), Weil expressed himself in
this note in terms of functions fields instead of curves. Later, he insisted
on geometric language.} Weil considers the group $G$ of divisor classes on $C$
of degree $0$ and order prime to $p$, and assumes that $G$ is isomorphic to
$(\mathbb{Q}/\mathbb{Z}{}(\text{non-}p))^{2g}$. A correspondence\footnote{See
p.\pageref{correspondence}.} $X$ on $C$ defines an endomorphism of $G$ whose
trace Weil calls the trace $\Tr(X)$ of $X$. This is an element of ${}%
\prod\nolimits_{l\neq p}\mathbb{Z}{}_{l}$, which nowadays we prefer to define
as the trace of $X$ on $TG\overset{\textup{{\tiny def}}}{=}\varprojlim
_{(n,p)=1}G_{n}$.\footnote{In terms of the jacobian variety $J$ of $C$, which
wasn't available to Weil at the time, $TG=\prod\nolimits_{l\neq p}T_{l}J$, and
the positivity of $\Tr(X\circ X^{\prime})$ expresses the positivity of the
Rosati involution on the endomorphism algebra of $J$.}

After some preliminaries, including the definition in this special case of
what is now called the Weil pairing, Weil announces his \textquotedblleft
important lemma\textquotedblright: let $X$ be an $(m_{1},m_{2})$
correspondence on $C$, and let $X^{\prime}$ be the $(m_{2},m_{1}%
)$-correspondence obtained by reversing the factors;

\begin{quote}
if $m_{1}=g$, we have in general (i.e., under conditions which it is
unnecessary to make precise here) $2m_{2}=\Tr(X\circ X^{\prime})$.\footnote{si
$m_{1}=g$, on a en g\'{e}n\'{e}ral (c'est-\`{a}-dire \`{a} des conditions
qu'il est inutile de pr\'{e}ciser ici) $2m_{2}=\Tr(X\circ X^{\prime})$.}
\end{quote}

\noindent From this he deduces that, for all correspondences $X$, the trace
$\Tr(X\circ X^{\prime})$ is a rational integer $\geq0$ and that the number of
coincident points of $X$ is $m_{1}+m_{2}-\Tr(X)$. On applying these statements
to the graph of the Frobenius map, he obtains his main results.

At the time Weil wrote his note, he had little access to the mathematical
literature. His confidence in the statements in his note was based on his own
rather ad hoc heuristic calculations (\OE uvres, I, pp.548--550). As he later wrote:

\begin{quote}
In other circumstances, publication would have seemed very premature. But in
April 1940, who could be sure of a tomorrow? It seemed to me that my ideas
contained enough substance to merit not being in danger of being
lost.\footnote{\textquotedblleft En d'autres circonstances, une publication
m'aurait paru bien pr\'{e}matur\'{e}e. Mais, en avril 1940, pouvait-on se
croire assur\'{e} du lendemain? Il me sembla que mes id\'{e}es contenaient
assez de substance pour ne pas m\'{e}riter d'\^{e}tre en danger de se
perdre.\textquotedblright\ This was seven months after Germany had invaded
Poland, precipitating the Second World War. At the time, Weil was confined to
a French military prison as a result of his failure to report for duty. His
five-year prison sentence was suspended when he joined a combat unit. During
the collapse of France, his unit was evacuated to Britain. Later he was
repatriated to Vichy France, from where in January 1941 he managed to exit to
the United States. Because of his family background, he was in particular
danger during this period.} (\OE uvres I, p.550.)
\end{quote}

\noindent

In his note, Weil had replaced the jacobian variety with its points of finite
order. But he soon realized that proving his \textquotedblleft important
lemma\textquotedblright\ depended above all on intersection theory on the
jacobian variety.

\begin{quote}
In 1940 I had seen fit to replace the jacobian with its group of points of
finite order. But I began to realize that a considerable part of Italian
geometry was based entirely on intersection theory\ldots\ In particular, my
\textquotedblleft important lemma\textquotedblright\ of 1940 seemed to depend
primarily on intersection theory on the jacobian; so this is what I
needed.\footnote{En 1940 j'avais jug\'{e} opportun de substituer \`{a} la
jacobienne le groupe de ses points d'ordre fini. Mais je commen\c{c}ais \`{a}
apercevoir qu'une notable partie de la g\'{e}om\'{e}trie italienne reposait
exclusivement sur la th\'{e}orie des intersections. Les travaux de van der
Waerden, bien qu'ils fussent rest\'{e}s bien en de\c{c}\`{a} des besoins,
donnaient lieu d'esp\'{e}rer que le tout pourrait un jour se transposer en
caract\'{e}ristique $p$ sans modification substantielle. En particulier, mon
\textquotedblleft lemme important\textquotedblright\ de 1940 semblait
d\'{e}pendre avant tout de la th\'{e}orie des intersections sur la jacobienne;
c'est donc celle-ci qu'il me fallait.\ } (\OE uvres I, p.556.)
\end{quote}

\noindent In expanding the ideas in his note, he would construct the jacobian
variety of a curve and develop a comprehensive theory of abelian varieties
over arbitrary fields parallel to the transcendental theory over $\mathbb{C}%
{}$. But first he had to rewrite the foundations of algebraic geometry.

In the meantime, he had found a more elementary proof of the Riemann
hypothesis, which involved only geometry on the product of two curves. First
he realized that if he used the fixed point formula (\ref{e3}) to
\textit{define} the trace $\sigma(X)$ of a correspondence $X$ on $C$, i.e.,
\[
\sigma(X)\overset{\textup{{\tiny def}}}{=}d_{1}(X)+d_{2}(X)-(X\cdot\Delta),
\]
then it was possible to prove directly that $\sigma$ had many of the
properties expected of a trace on the ring of correspondence classes. Then, as
he writes (\OE uvres I, p.557):

\begin{quote}
At the same time I returned seriously to the study of the \textit{Trattato
}[Severi 1926]. The trace $\sigma$ does not appear as such; but there is much
talk of a \textquotedblleft difetto di equivalenza\textquotedblright\ whose
positivity is shown on page 254. I soon recognized it as my integer
$\sigma(X\circ X^{\prime})$ \ldots\ I could see that, to ensure the validity
of the Italian methods in characteristic $p$, all the foundations would have
to be redone, but the work of van der Waerden, together with that of the
topologists, allowed me to believe that it would not be beyond my
strength.\footnote{En m\^{e}me temps je m'\'{e}tais remis s\'{e}rieusement
\`{a} l'\'{e}tude du \textit{Trattato}. La trace $\sigma$ a n'y appara\^{\i}t
pas en tant que telle; mais il y est fort question d'un \textquotedblleft
difetto di equivalenza\textquotedblright\ dont la positivit\'{e} est
d\'{e}montr\'{e}e \`{a} la page 254. J'y reconnus bient\^{o}t mon entier
$\sigma(yy^{\prime})$\ldots\ Je voyais bien que, pour s'assurer de la
validit\'{e} des m\'{e}thodes italiennes en caract\'{e}ristique $p$, toutes
les fondations seraient \`{a} reprendre, mais les travaux de van der Waerden,
joints \`{a} ceux des topologues, donnaient \`{a} croire que ce ne serait pas
au dessus de mes forces.}
\end{quote}

\noindent Weil announced this proof in (Weil 1941). Before describing it, I
present a simplified modern version of the proof.

\begin{note}
\label{r37}Whereas the 1940 proof (implicitly) uses the jacobian $J$ of the
curve $C$, the 1941 proof involves only geometry on $C\times C$. Both use the
positivity of \textquotedblleft$\sigma(X\circ X^{\prime})$\textquotedblright%
\ but the first proof realizes this integer as a trace on the torsion group
$\Jac(C)($non-$p)$ whereas the second expresses it in terms of intersection
theory on $C\times C$.
\end{note}

\subsection{The geometric proof of the Riemann hypothesis for curves}

Let $V$ be a nonsingular projective surface over an algebraically closed field
$k$.

\subsubsection{Divisors}

A divisor on $V$ is a formal sum $D=\sum n_{i}C_{i}$ with $n_{i}\in
\mathbb{Z}{}$ and $C_{i}$ an irreducible curve on $V$. We say that $D$ is
positive, denoted $D\geq0$, if all the $n_{i}\geq0$. Every $f\in k(V)^{\times
}$ has an associated divisor $(f)$ of zeros and poles --- these are the
principal divisors. Two divisors $D$ and $D^{\prime}$ are said to be linearly
equivalent if%
\[
D^{\prime}=D+(f)\text{ some }f\in k(V)^{\times}.
\]

For a divisor $D$, let%
\[
L(D)=\{f\in k(V)\mid(f)+D\geq0\}.
\]
Then $L(C)$ is a finite-dimensional vector space over $k$, whose dimension we
denote by $l(D)$. The map $g\mapsto gf$ is an isomorphism $L(D)\rightarrow
L(D-(f))$, and so $l(D)$ depends only on the linear equivalence class of $D$.

\subsubsection{Elementary intersection theory\label{intersection}}

Because $V$ is nonsingular, a curve $C$ on $V$ has a local equation at every
closed point $P$ of $V$, i.e., there exists an $f$ such that%
\[
C=(f)+\text{components not passing through }P.
\]
If $C$ and $C^{\prime}$ are distinct irreducible curves on $V$, then their
intersection number at $P\in C\cap C^{\prime}$ is%
\[
(C\cdot C^{\prime})_{P}\overset{\textup{{\tiny def}}}{=}\dim_{k}(\mathcal{O}%
{}_{V,P}/(f,f^{\prime}))
\]
where $f$ and $f^{\prime}$ are local equations for $C$ and $C^{\prime}$ at
$P$, and their (global) intersection number is%
\[
(C\cdot C^{\prime})=\sum_{P\in C\cap C^{\prime}}(C\cdot C^{\prime})_{P}.
\]
This definition extends by linearity to pairs of divisors $D,D^{\prime}$
without common components. Now observe that $((f)\cdot C)=0$, because it
equals the degree of the divisor of $f|C$ on $C$, and so $(D\cdot D^{\prime})$
depends only on the linear equivalence classes of $D$ and $D^{\prime}$. This
allows us to define $(D\cdot D^{\prime})$ for all pairs $D,D^{\prime}$ by
replacing $D$ with a linearly equivalent divisor that intersects $D^{\prime}$
properly. In particular, $(D^{2})\overset{\textup{{\tiny def}}}{=}(D\cdot D)$
is defined. See Shafarevich 1994, IV, \S 1, for more details.

\subsubsection{The Riemann-Roch theorem}

Recall that the Riemann-Roch theorem for a curve $C$ states that, for all
divisors $D$ on $C$,%
\[
l(D)-l(K_{C}-D)=\deg(D)+1-g
\]
where $g$ is the genus of $C$ and $K_{C}$ is a canonical divisor (so $\deg
K_{C}=2g-2$ and $l(K_{C})=g$). Better, in terms of cohomology,%
\begin{align*}
\chi(\mathcal{O(}{}D))  &  =\deg(D)+\chi(\mathcal{O}{})\\
h^{1}(D)  &  =h^{0}(K_{C}-D)\text{.}%
\end{align*}

The Riemann-Roch theorem for a surface $V$ states that, for all divisors $D$
on $V$,%
\[
l(D)-\mathrm{\mathrm{sup}}(D)+l(K_{V}-D)=p_{a}+1+\frac{1}{2}(D\cdot D-K_{V})
\]
where $K_{V}$ is a canonical divisor and%
\begin{align*}
p_{a}  &  =\chi(\mathcal{O}{})-1\quad\text{(arithmetic genus),}\\
\mathrm{sup}(D)  &  =\text{superabundance of }D\text{ (}\geq0\text{, and
}=0\text{ for some divisors).}%
\end{align*}
Better, in terms of cohomology,%
\begin{align*}
\chi(\mathcal{O(}{}D))  &  =\chi(\mathcal{O}{}_{V})+\frac{1}{2}(D\cdot D-K)\\
h^{2}(D)  &  =h^{0}(K-D),
\end{align*}
and so%
\[
\mathrm{sup}(D)=h^{1}(D).
\]

We shall also need the adjunction formula: let $C$ be a curve on $V$; then%
\[
K_{C}=(K_{V}+C)\cdot C.
\]

\subsubsection{The Hodge index theorem}

Embed $V$ in $\mathbb{P}{}^{n}$. A hyperplane section of $V$ is a divisor of
the form $H=V\cap H^{\prime}$ with $H^{\prime}$ a hyperplane in $\mathbb{P}%
{}^{n}$ not containing $V$. Any two hyperplane sections are linearly
equivalent (obviously).

\begin{lemma}
\label{r0}For a divisor $D$ and hyperplane section $H$,%
\begin{equation}
l(D)>1\implies(D\cdot H)>0. \label{eq1}%
\end{equation}

\end{lemma}

\begin{proof}
The hypothesis implies that there exists a $D_{1}>0$ linearly equivalent to
$D$. If the hyperplane $H^{\prime}$ is chosen not to contain a component of
$D_{1}$, then the hyperplane section $H=V\cap H^{\prime}$ intersects $D_{1}$
properly. Now $D_{1}\cap H=D_{1}\cap H^{\prime}$, which is nonempty by
dimension theory, and so $(D_{1}\cdot H)>0$.
\end{proof}

\begin{theorem}
[Hodge Index Theorem]\label{r1}For a divisor $D$ and hyperplane section $H$,%
\[
(D\cdot H)=0\implies(D\cdot D)\leq0.
\]

\end{theorem}

\begin{proof}
We begin with a remark: suppose that $l(D)>0$, i.e., there exists an $f\neq0$
such that $(f)+D\geq0$; then, for a divisor $D^{\prime}$,%
\begin{equation}
l(D+D^{\prime})=l((D+(f))+D^{\prime})\geq l(D^{\prime}). \label{eq2}%
\end{equation}

We now prove the theorem. To prove the contrapositive, it suffices to show
that
\[
(D\cdot D)>0\implies l(mD)>1\text{ for some integer }m\text{,}%
\]
because then%
\[
(D\cdot H)=\frac{1}{m}(mD\cdot H)\neq0
\]
by (\ref{eq1}) above. Hence, suppose that $(D\cdot D)>0$. By the Riemann-Roch
theorem%
\[
l(mD)+l(K_{V}-mD)\geq\frac{(D\cdot D)}{2}m^{2}+\text{lower powers of }m.
\]
Therefore, for a fixed $m_{0}\geq1$, we can find an $m>0$ such that%
\begin{align*}
l(mD)+l(K_{V}-mD)  &  \geq m_{0}+1\\
l(-mD)+l(K_{V}+mD)  &  \geq m_{0}+1.
\end{align*}
If both $l(mD)\leq1$ and $l(-mD)\leq1$, then both $l(K_{V}-mD)\geq m_{0}$ and
$l(K_{V}+mD)\geq m_{0}$, and so%
\[
l(2K_{V})=l(K_{V}-mD+K_{V}+mD)\overset{\text{(\ref{eq2})}}{\geq}%
l(K_{V}+mD)\geq m_{0}\text{.}%
\]
As $m_{0}$ was arbitrary, this is impossible.
\end{proof}

Let $Q$ be a symmetric bilinear form on a finite-dimensional vector space $W$
over $\mathbb{Q}{}$ (or $\mathbb{R}{})$. There exists a basis for $W$ such
that $Q(x,x)=a_{1}x_{1}^{2}+\cdots+a_{n}x_{n}^{2}$. The number of $a_{i}>0$ is
called the \textit{index} (of positivity) of $Q$ --- it is independent of the
basis. There is the following (obvious) criterion: $Q$ has index $1$ if and
only if there exists an $x\in V$ such that $Q(x,x)>0$ and $Q(y,y)\leq0$ for
all $y\in\langle x\rangle^{\perp}$.

Now consider a surface $V$ as before, and let $\Pic(V)$ denote the group of
divisors on $V$ modulo linear equivalence. We have a symmetric bi-additive
intersection form%
\[
\Pic(V)\times\Pic(V)\rightarrow\mathbb{Z}{}.
\]
On tensoring with $\mathbb{Q}{}$ and quotienting by the kernels, we get a
nondegenerate intersection form%
\[
N(V)\times N(V)\rightarrow\mathbb{Q}{}.
\]

\begin{corollary}
\label{r2}The intersection form on $N(V)$ has index $1$.
\end{corollary}

\begin{proof}
Apply the theorem and the criterion just stated.
\end{proof}

\begin{corollary}
\label{r3}Let $D$ be a divisor on $V$ such that $(D^{2})>0$. If $(D\cdot
D^{\prime})=0$, then $(D^{\prime2})\leq0$.
\end{corollary}

\begin{proof}
The form is negative definite on $\langle D\rangle^{\perp}$.
\end{proof}

\subsubsection{The inequality of Castelnuovo-Severi}

Now take $V$ to be the product of two curves, $V=C_{1}\times C_{2}$. Identify
$C_{1}$ and $C_{2}$ with the curves $C_{1}\times\mathrm{pt}$ and
$\mathrm{pt\times C}_{2}$ on $V$, and note that%
\begin{align*}
C_{1}\cdot C_{1}  &  =0=C_{2}\cdot C_{2}\\
C_{1}\cdot C_{2}  &  =1=C_{2}\cdot C_{1}.
\end{align*}
Let $D$ be a divisor on $C_{1}\times C_{2}$ and set $d_{1}=D\cdot C_{1}$ and
$d_{2}=D\cdot C_{2}.$

\begin{theorem}
[Castelnuovo-Severi Inequality]\label{r4}Let $D$ be a divisor on $V$; then%
\begin{equation}
(D^{2})\leq2d_{1}d_{2}. \label{e23}%
\end{equation}

\end{theorem}

\begin{proof}
We have%
\begin{align*}
(C_{1}+C_{2})^{2}  &  =2>0\\
(D-d_{2}C_{1}-d_{1}C_{2})\cdot(C_{1}+C_{2})  &  =0.
\end{align*}
Therefore, by the Hodge index theorem,%
\[
(D-d_{2}C_{1}-d_{1}C_{2})^{2}\leq0.
\]
On expanding this out, we find that $D^{2}\leq2d_{1}d_{2}.$
\end{proof}

Define the equivalence defect (\textit{difetto di equivalenza}) of a divisor
$D$ by%
\[
\mathrm{def}(D)=2d_{1}d_{2}-(D^{2})\geq0.
\]

\begin{corollary}
\label{r00}Let $D$, $D^{\prime}$ be divisors on $V$; then%
\begin{equation}
\left\vert \left(  D\cdot D^{\prime}\right)  -d_{1}d_{2}^{\prime}-d_{2}%
d_{1}^{\prime}\right\vert \leq\left(  \mathrm{def}(D)\mathrm{def}(D^{\prime
})\right)  ^{1/2}. \label{e6}%
\end{equation}

\end{corollary}

\begin{proof}
Let $m,n\in\mathbb{Z}{}$. On expanding out%
\[
\mathrm{def}(mD+nD^{\prime})\geq0,
\]
we find that%
\[
m^{2}\mathrm{def}(D)-2mn\left(  (D\cdot D^{\prime})-d_{1}d_{2}^{\prime}%
-d_{2}d_{1}^{\prime}\right)  +n^{2}\mathrm{def}(D^{\prime})\geq0.
\]
As this holds for all $m,n$, it implies (\ref{e6}).
\end{proof}

\begin{example}
\label{r01}Let $f$ be a nonconstant morphism $C_{1}\rightarrow C_{2}$, and let
$g_{i}$ denote the genus of $C_{i}$. The graph of $f$ is a divisor $\Gamma
_{f}$ on $C_{1}\times C_{2}$ with $d_{2}=1$ and $d_{1}$ equal to the degree of
$f$. Now%
\[
K_{\Gamma_{f}}=(K_{V}+\Gamma_{f})\cdot\Gamma_{f}\quad\quad\text{(adjunction
formula).}%
\]
On using that $K_{V}=K_{C_{1}}\times C_{2}+C_{1}\times K_{C_{2}}$, and taking
degrees, we find that%
\[
2g_{1}-2=(\Gamma_{f})^{2}+(2g_{1}-2)\cdot1+(2g_{2}-2)\deg(f).
\]
Hence%
\begin{equation}
\mathrm{def}(\Gamma_{f})=2g_{2}\deg(f). \label{e1}%
\end{equation}

\end{example}

\subsubsection{Proof of the Riemann hypothesis for curves}

Let $C_{0}$ be a projective nonsingular curve over a finite field $k_{0}$, and
let $C$ be the curve obtained by extension of scalars to the algebraic closure
$k$ of $k_{0}$. Let $\pi$ be the Frobenius endomorphism of $C$. Then (see
(\ref{e1})), $\mathrm{def}(\Delta)=2g$ and $\mathrm{def}(\Gamma_{\pi})=2gq$,
and so (see (\ref{e6})),
\[
\left\vert (\Delta\cdot\Gamma_{\pi})-q-1\right\vert \leq2gq^{1/2}.
\]
As%
\[
(\Delta\cdot\Gamma_{\pi})=\text{number of points on }C\text{ rational over
}k_{0}\text{,}%
\]
we obtain Riemann hypothesis for $C_{0}$.

\begin{aside}
\label{r02}Note that, except for the last few lines, the proof is purely
geometric and takes place over an algebraically closed field.\footnote{I once
presented this proof in a lecture. At the end, a listener at the back
triumphantly announced that I couldn't have proved the Riemann hypothesis
because I had only ever worked over an algebraically closed field.} This is
typical: study of the Riemann hypothesis over finite fields suggests questions
in algebraic geometry whose resolution proves the hypothesis.
\end{aside}

\subsubsection{Correspondences\label{correspondence}}

A divisor $D$ on a product $C_{1}\times C_{2}$ of curves is said to have
\textit{valence zero} if it is linearly equivalent to a sum of divisors of the
form $C_{1}\times\mathrm{pt}$ and $\mathrm{pt}\times C_{2}$. The group of
correspondences $\mathcal{C}(C_{1},C_{2}){}$ is the quotient of the group of
divisors on $C_{1}\times C_{2}$ by those of valence zero. When $C_{1}=C_{2}%
=C$, the composite of two divisors $D_{1}$ and $D_{2}$ is
\[
D_{1}\circ D_{2}\overset{\textup{{\tiny def}}}{=}p_{13\ast}(p_{12}^{\ast}%
D_{1}\cdot p_{23}^{\ast}D_{2})
\]
where the $p_{ij}$ are the projections $C\times C\times C\rightarrow C\times
C$; in general, it is only defined up to linear equivalence. When $D\circ E$
is defined, we have%
\begin{equation}
d_{1}(D\circ E)=d_{1}(D)d_{1}(E),\quad d_{2}(D\circ E)=d_{2}(D)d_{2}%
(E),\quad(D\cdot E)=(D\circ E^{\prime},\Delta) \label{e29}%
\end{equation}
where, as usual, $E^{\prime}$ is obtained from $E$ by reversing the factors.
Composition makes the group $\mathcal{C}(C,C)$ of correspondences on $C$ into
a ring $\mathcal{R}{}(C)$.

Following Weil (cf. (\ref{e3}), p.\pageref{e3}), we define the
\textquotedblleft trace\textquotedblright\ of a correspondence $D$ on $C$ by%
\[
\sigma(D)=d_{1}(D)+d_{2}(D)-(D\cdot\Delta).
\]
Applying (\ref{e29}), we find that%
\begin{align*}
\sigma(D\circ D^{\prime})  &  \overset{\textup{{\tiny def}}}{=}d_{1}(D\circ
D^{\prime})+d_{2}(D\circ D^{\prime})-((D\circ D^{\prime})\cdot\Delta)\\
&  =d_{1}(D)d_{2}(D)+d_{2}(D)d_{1}(D)-(D^{2})\\
&  =\operatorname*{def}(D).
\end{align*}
Thus Weil's inequality $\sigma(D\circ D^{\prime})\geq0$ is a restatement of
(\ref{e23}).

\subsection{Sur les courbes\ldots\ (Weil 1948a)}

In this short book, Weil provides the details for his 1941 proof. He does not
use the Riemann-Roch theorem for a surface, which was not available in nonzero
characteristic at the time.

Let $C$ be a nonsingular projective curve of genus $g$ over an algebraically
closed field $k$. We assume a theory of intersections on $C\times C$, for
example, that in Weil 1946 or, more simply, that sketched above
(p.\pageref{intersection}). We briefly sketch\label{Weil} Weil's proof that
$\sigma(D\circ D^{\prime})\geq0$. As we have just seen, this suffices to prove
the Riemann hypothesis.

Assume initially that $D$ is positive, that $d_{2}(D)=g\geq2$, and that
\[
D=D_{1}+\cdots+D_{g}%
\]
with the $D_{i}$ distinct. We can regard $D$ as a multivalued map $P\mapsto
D(P)=\{D_{1}(P),\ldots,D_{g}(P)\}$ from $C$ to $C$. Then%
\begin{align*}
D(k)  &  =\{(P,D_{i}(P))\mid P\in C(k),\quad1\leq i\leq g\}\\
D^{\prime}(k)  &  =\{(D_{i}(P),P)\mid P\in C(k),\quad1\leq i\leq g\}\\
\left(  D\circ D^{\prime}\right)  (k)  &  =\{(D_{i}(P),D_{j}(P))\mid P\in
C(k),\quad1\leq i,j\leq g\}\text{.}%
\end{align*}
The points with $i=j$ contribute a component $Y_{1}=d_{1}(D)\Delta$ to $D\circ
D^{\prime}$, and
\[
(Y_{1}\cdot\Delta)=d_{1}(D)(\Delta\cdot\Delta)=d_{1}(D)(2-2g).
\]
It remains to estimate $(Y_{2}\cdot\Delta)$ where $Y_{2}=D-Y_{1}$. Let $K_{C}$
denote a positive canonical divisor on $C$, and let $\left\{  \varphi
_{1},\ldots,\varphi_{g}\right\}  $ denote a basis for $L(K_{C})$. For $P\in
C(k)$, let%
\[
\Phi(P)=\det(\varphi_{i}(D_{j}(P))
\]
(Weil 1948a, II, n$^{\circ}$13, p.52). Then $\Phi$ is a rational function on
$C$, whose divisor we denote $(\Phi)=(\Phi)_{0}-(\Phi)_{\infty}$. The zeros of
$\Phi$ correspond to the points $(D_{i}(P),D_{j}(P))$ with $D_{i}(P)=D_{j}%
(P)$, $i\neq j$, and so%
\[
\deg(\Phi)_{0}=(Y_{2}\cdot\Delta).
\]
On the other hand the poles of $\Phi$ are at the points $P$ for which
$D_{j}(P)$ lies in the support of $K_{C}$, and so%
\[
\deg(\Phi)_{\infty}\leq\deg(K_{C})d_{1}(D)=(2g-2)d_{1}(D).
\]
Therefore,%
\[
(D\circ D^{\prime},\Delta)\leq d_{1}(D)(2-2g)+(2g-2)d_{1}(D)=0
\]
and so%
\[
\sigma(D\circ D^{\prime})\overset{\textup{{\tiny def}}}{=}d_{1}(D\circ
D^{\prime})+d_{2}(D\circ D^{\prime})-(D\circ D^{\prime},\Delta)\geq
2gd_{1}(D)\geq0.
\]

Let $D$ be a divisor on $C\times C$. Then $D$ becomes equivalent to a divisor
of the form considered in the last paragraph after we pass to a finite
generically Galois covering $V\rightarrow C\times C$ (this follows from Weil
1948a, Proposition 3, p.43). Elements of the Galois group of $k(V)$ over
$k(C\times C)$ act on the matrix $(\varphi_{i}(D_{j}(P))$ by permuting the
columns, and so they leave its determinant unchanged except possibly for a
sign. On replacing $\Phi$ with its square, we obtain a rational function on
$C\times C$, to which we can apply the above argument. This completes the sketch.

Weil gives a rigorous presentation of the argument just sketched\textit{ }in
II, pp.42--54, of his book 1948a. In fact, he proves the stronger result:
$\sigma(\xi\circ\xi^{\prime})>0$ for all nonzero $\xi\in\mathcal{R}{}(C)$
(ibid. Thm 10, p.54). In the earlier part of the book, Weil (re)proves the
Riemann-Roch theorem for curves and develops the theory of correspondences on
a curve based on the intersection theory developed in his
\textit{Foundations.} In the later part he applies the inequality to obtain
his results on the zeta function of $C$, but he defers the proof of his
results on Artin $L$-series to his second book.\footnote{Serre writes (email
July 2015):\bquote Weil always insisted that Artin's conjecture on the
holomorphy of non-abelian $L$-functions is on the same level of difficulty as
the Riemann hypothesis. In his book \textquotedblleft Courbes alg\'{e}briques
... \textquotedblright\ he mentions (on p.83) that the $L$-functions are
polynomials, but he relies on the second volume for the proof (based on the
$\ell$-adic representations: the positivity result alone is not enough). I
find interesting that, while there are several \textquotedblleft
elementary\textquotedblright\ proofs of the Riemann hypothesis for curves,
none of them gives that Artin's $L$-functions are polynomial. The only way to
prove it is via $\ell$-adic cohomology, or equivalently, $\ell$-adic
representations. Curiously, the situation is different over number fields,
since we know several non-trivial cases where Artin's conjecture is true
(thanks to Langlands theory), and no case where the Riemann hypothesis
is!\equote}

\subsubsection{Some history}

Let $C_{1}$ and $C_{2}$ be two nonsingular projective curves over an
algebraically closed field $k$. Severi, in his fundamental paper (1903),
defined a bi-additive form%
\[
\sigma(D,E)=d_{1}(D)d_{2}(E)+d_{2}(D)d_{1}(E)-(D\cdot E)
\]
on $\mathcal{C}(C_{1},C_{2})$ and conjectured that it is non-degenerate. Note
that%
\[
\sigma(D,D)=2d_{1}(D)d_{2}(D)-(D^{2})=\mathrm{def}(D)\text{.}%
\]
Castelnuovo (1906) proved the following theorem,

\begin{quote}
let $D$ be a divisor on $C_{1}\times C_{2}$; then $\sigma(D,D)\geq0$, with
equality if and only if $D$ has valence zero.\footnote{In fact, Castelnuovo
proved a more precise result, which Kani (1984) extended to characteristic
$p$, thereby obtaining another proof of the defect inequality in
characteristic $p$.}
\end{quote}

\noindent from which he was able to deduce Severi's conjecture. Of course,
this all takes place in characteristic zero.

As noted earlier, the Riemann Roch theorem for surfaces in characteristic $p$
was not available to Weil. The Italian proof of the complex Riemann-Roch
theorem rests on a certain lemma of Enriques and Severi. Zariski (1952)
extended this lemma to normal varieties of all dimensions in all
characteristics; in particular, he obtained a proof of the Riemann-Roch
theorem for normal surfaces in characteristic $p$.

Mattuck and Tate (1958) used the Riemann-Roch theorem in the case of a product
of two curves to obtain a simple proof of the Castelnuovo-Severi
inequality.\footnote{In their introduction, they refer to Weil's 1940 note and
write: \textquotedblleft the [Castelnuovo-Severi] inequality is really a
statement about the geometry on a very special type of surface --- the product
of two curves --- and it is natural to ask whether it does not follow from the
general theory of surfaces.\textquotedblright\ \noindent Apparently they had
forgotten that Weil had answered this question in 1941!}

In trying to understand the exact scope of the method of Mattuck and Tate,
Grothendieck (1958a) stumbled on the Hodge index
theorem.\footnote{\textquotedblleft En essayant de comprendre la port\'{e}e
exacte de leur m\'{e}thode, je suis tomb\'{e} sur l'\'{e}nonc\'{e} suivant,
connu en fait depuis 1937 (comme me l'a signal\'{e} J.P. Serre), mais
apparemment peu connu et utilis\'{e}.\textquotedblright\ The Hodge index
theorem was first proved by analytic methods in Hodge 1937, and by algebraic
methods in Segre 1937 and in Bronowski 1938.} In particular, he showed that
the Castenuovo-Severi-Weil inequality follows from a general statement, valid
for all surfaces, which itself is a simple consequence of the Riemann-Roch
theorem for surfaces.

The proof presented in the preceding subsection incorporates these simplifications.

\subsubsection{Variants}

Igusa (1949) gave another elaboration of the proof of the Riemann hypothesis
for curves sketched in Weil 1941. Following Castelnuovo, he first proves a
formula of Schubert, thereby giving the first rigorous proof of this formula
valid over an arbitrary field. His proof makes use of the general theory of
intersection multiplicities in Weil's \textit{Foundations}.

Intersection multiplicities in which one of the factors is a hypersurface can
be developed in an elementary fashion, using little more that the theory of
discrete valuations (cf. p.\pageref{intersection}). In his 1953 thesis, Weil's
student Frank Quigley \textquotedblleft arranged\textquotedblright\ Weil's
1941 proof so that it depends only on this elementary intersection theory
(Quigley 1953). As did Igusa, he first proved Schubert's formula.

Finally, in his thesis (Hamburg 1951), Hasse's student Roquette gave a proof
of the Riemann hypothesis for curves based on Deuring's theory of
correspondences for double-fields (published as Roquette 1953).

\subsubsection{Applications to exponential sums}

Davenport and Hasse (1935) showed that certain arithmetic functions can be
realized as the traces of Frobenius maps. Weil (1948c) went much further, and
showed that all exponential sums in one variable can be realized in this way.
From his results on the zeta functions of curves, he obtained new estimates
for these sums. Later developments, especially the construction of $\ell$-adic
and $p$-adic cohomologies, and Deligne's work on the zeta functions of
varieties over finite fields, have made this a fundamental tool in analytic
number theory.

\subsection{Foundations of Algebraic Geometry (Weil 1946)}

When Weil began the task of constructing foundations for his announcements he
was, by his own account, not an expert in algebraic geometry. During a
six-month stay in Rome, 1925-26, he had learnt something of the Italian school
of algebraic geometry, but mainly during this period he had studied linear
functionals with Vito Volterra.

In writing his Foundations, Weil's main inspiration was the work of van der
Waerden,\footnote{In the introduction, Weil writes that he \textquotedblleft
greatly profited from van der Waerden's well-known series of papers (published
in Math. Ann. between 1927 and 1938)\ldots; from Severi's sketchy but
suggestive treatment of the same subject, in his answer to van der Waerden's
criticism of the work of the Italian school; and from the topological theory
of intersections, as developed by Lefschetz and other contemporary
mathematicians.\textquotedblright} which gives a rigorous algebraic treatment
of projective varieties over fields of arbitrary characteristic and develops
intersection theory by global methods. However, Weil was unable to construct
the jacobian variety of a curve as a projective variety. This forced him to
introduce the notion of an abstract variety, defined by an atlas of charts,
and to develop his intersection theory by local methods. Without the Zariski
topology, his approach was clumsy. However, his \textquotedblleft abstract
varieties\textquotedblright\ liberated algebraic geometry from the study of
varieties embedding in an affine or projective space. In this respect, his
work represents a break with the past.

Weil completed his book in 1944. As Zariski wrote in a review (BAMS 1948):

\begin{quotation}
In the words of the author the main purpose of this book is \textquotedblleft
to present a detailed and connected treatment of the properties of
intersection multiplicities, which is to include all that is necessary and
sufficient to legitimize the use made of these multiplicities in classical
algebraic geometry, especially of the Italian school\textquotedblright. There
can be no doubt whatsoever that this purpose has been fully achieved by Weil.
After a long and careful preparation (Chaps. I--IV) he develops in two central
chapters (V and VI) an intersection theory which for completeness and
generality leaves little to be desired. It goes far beyond the previous
treatments of this foundational topic by Severi and van der Waerden and is
presented with that absolute rigor to which we are becoming accustomed in
algebraic geometry. In harmony with its title the book is entirely
self-contained and the subject matter is developed \textit{ab initio}.

It is a remarkable feature of the book that---with one exception (Chap. III)
---no use is made of the higher methods of modern algebra. The author has made
up his mind not to assume or use modern algebra \textquotedblleft beyond the
simplest facts about abstract fields and their extensions and the bare
rudiments of the theory of ideals.\textquotedblright\ \ldots The author
justifies his procedure by an argument of historical continuity, urging a
return \textquotedblleft to the palaces which are ours by
birthright.\textquotedblright\ But it is very unlikely that our predecessors
will recognize in Weil's book their own familiar edifice, however improved and
completed. If the traditional geometer were invited to choose between
\textquotedblleft makeshift constructions full of rings, ideals and
valuations\textquotedblright\footnote{Weil's description of Zariski's approach
to the foundations.} on one hand, and constructions full of fields, linearly
disjoint fields, regular extensions, independent extensions, generic
specializations, finite specializations and specializations of specializations
on the other, he most probably would decline the choice and say:
\textquotedblleft A plague on both your houses!\textquotedblright
\end{quotation}

For fifteen years, Weil's book provided a secure foundation for work in
algebraic geometry, but then was swept away by commutative algebra, sheaves,
cohomology, and schemes, and was largely forgotten.\footnote{When Langlands
decided to learn algebraic geometry in Berkeley in 1964-65, he read Weil's
\textit{Foundations\ldots} and Conforto's \textit{Abelsche Funktionen\ldots} A
year later, as a student at Harvard, I was able to attend Mumford's course on
algebraic geometry which used commutative algebra, sheaves, and schemes.} For
example, his intersection theory plays little role in Fulton 1984. It seems
that the approach of van der Waerden and Weil stayed too close to the Italian
original with its generic points, specializations, and so on; what algebraic
geometry needed was a complete renovation.

\subsection{Vari\'{e}t\'{e}s ab\'{e}liennes et \ldots\ (Weil 1948b)}

In this book and later work, Weil constructs the jacobian variety of a curve,
and develops a comprehensive theory of abelian varieties over arbitrary
fields, parallel to the transcendental theory over $\mathbb{C}{}$. Although
inspired by his work on the Riemann hypothesis, this work goes far beyond what
is needed to justify his 1940 note. Weil's book opened the door to the
arithmetic study of abelian varieties. In the twenty years following its
publication, almost all of the important results on elliptic curves were
generalized to abelian varieties. Before describing two of Weil's most
important accomplishments in his book, I list some of these developments.

\begin{plain}
\label{r40}There were improvements to the theory of abelian varieties by Weil
and others; see (\ref{r35}) below.
\end{plain}

\begin{plain}
\label{r41}Deuring's theory of complex multiplication for elliptic curves was
extended to abelian varieties of arbitrary dimension by Shimura, Taniyama, and
Weil (see their talks at the Symposium on Algebraic Number Theory, Tokyo \&
Nikko, 1955).
\end{plain}

\begin{plain}
\label{r42}For a projective variety $V$, one obtains a height function by
choosing a projective embedding of $V$. In 1958 N\'{e}ron conjectured that for
an abelian variety there is a \textit{canonical} height function characterized
by having a certain quadratic property. The existence of such a height
function was proved independently by N\'{e}ron and Tate in the early 1960s.
\end{plain}

\begin{plain}
\label{r43}Tate studied the Galois cohomology of abelian varieties, extending
earlier results of Cassels for elliptic curves. This made it possible, for
example, to give a simple natural proof of the Mordell-Weil theorem for
abelian varieties over global fields.
\end{plain}

\begin{plain}
\label{r44}N\'{e}ron (1964) developed a theory of minimal models of abelian
varieties over local and global fields, extending Kodaira's theory for
elliptic curves over function fields in one variable over $\mathbb{C}{}$.
\end{plain}

\begin{plain}
\label{r45}These developments made it possible to state the conjecture of
Birch and Swinnerton-Dyer for abelian varieties over global fields (Tate
1966a) . The case of a jacobian variety over a global function field inspired
the conjecture of Artin-Tate concerning the special values of the zeta
function of a surface over a finite field (ibid. Conjecture C).
\end{plain}

\begin{plain}
\label{r46}Tate (1964) conjectured that, for abelian varieties $A$, $B$ over a
field $k$ finitely generated over the prime field, the map%
\[
\mathbb{Z}{}_{\ell}\otimes\Hom(A,B)\rightarrow\Hom(T_{\ell}A,T_{\ell
}B)^{\Gal(k^{\mathrm{sep}}/k)}%
\]
is an isomorphism. Mumford explained to Tate that, for elliptic curves over a
finite field, this follows from the results of Deuring (1941). In one of his
most beautiful results, Tate proved the statement for all abelian varieties
over finite fields (Tate 1966b). At a key point in the proof, he needed to
divide a polarization by $l^{n}$; for this he was able to appeal to
\textquotedblleft the proposition on the last page of Weil
1948b\textquotedblright.
\end{plain}

\begin{plain}
\label{r47}(Weil-Tate-Honda theory) Fix a power $q$ of a prime $p$. An element
$\pi$ algebraic over $\mathbb{Q}{}$ is called a \textit{Weil number} if it is
integral and $\left\vert \rho(\pi)\right\vert =q^{1/2}$ for all embeddings
$\rho\colon\mathbb{Q}{}[\pi]\rightarrow\mathbb{C}{}$. Two Weil numbers $\pi$
and $\pi^{\prime}$ are conjugate if there exists an isomorphism $\mathbb{Q}%
{}[\pi]\rightarrow\mathbb{Q}{}[\pi^{\prime}]$ sending $\pi$ to $\pi^{\prime}$.
Weil attached a Weil number to each simple abelian variety over $\mathbb{F}%
{}_{q}$, whose conjugacy class depends only on the isogeny class of the
variety, and Tate showed that the map from isogeny classes of simple abelian
varieties over $\mathbb{F}{}_{q}$ to conjugacy classes of Weil numbers is
injective. Using the theory of complex multiplication, Honda (1968) shows that
the map is also surjective. In this way, one obtains a classification of the
isogeny classes of simple abelian varieties over $\mathbb{F}{}_{q}$. Tate
determined the endomorphism algebra of simple abelian variety $A$ over
$\mathbb{F}{}_{q}$ in terms of its Weil number $\pi$: it is a central division
algebra over $\mathbb{Q}{}[\pi]$ which splits at no real prime of
$\mathbb{Q}{}[\pi]$, splits at every finite prime not lying over $p$, and at a
prime $v$ above $p$ has invariant%
\begin{equation}
\inv_{v}(\End^{0}(A))\equiv\frac{\ord_{v}(\pi)}{\ord_{v}(q)}[\mathbb{Q}{}%
[\pi]_{v}\colon\mathbb{Q}{}_{p}]\quad(\text{mod }1); \label{e26}%
\end{equation}
moreover,%
\[
2\dim(A)=[\End^{0}(A)\colon\mathbb{Q}{}[\pi]]^{1/2}\cdot\lbrack\mathbb{Q}%
{}[\pi]\colon\mathbb{Q}{}].
\]
Here $\End^{0}(A)=\End(A)\otimes\mathbb{Q}{}$. See Tate 1968.
\end{plain}

\subsubsection{Iwasawa theory foretold}

Weil was very interested in the analogy between number fields and function
fields and, in particular, in \textquotedblleft extending\textquotedblright%
\ results from function fields to number fields. In 1942 he wrote:

\begin{quote}
Our proof for the Riemann hypothesis depended upon the extension of the
function-fields by roots of unity, i.e., by constants; the way in which the
Galois group of such extensions operates on the classes of divisors in the
original field and its extensions gives a linear operator, the characteristic
roots (i.e., the eigenvalues) of which are the roots of the zeta-function. On
a number field, the nearest we can get to this is by the adjunction of $l^{n}%
$-th roots of unity, $l$ being fixed; the Galois group of this infinite
extension is cyclic, and defines a linear operator on the projective limit of
the (absolute) class groups of those successive finite extensions; this should
have something to do with the roots of the zeta-function of the field. (Letter
to Artin, \OE uvres I, p.298, 1942.)
\end{quote}

\subsection{Construction of the jacobian variety of a curve}

Let $C$ be a nonsingular projective curve over a field $k$, which for
simplicity we take to be algebraically closed. The jacobian variety $J$ of $C$
should be such that $J(k)$ is the group $\Jac(C)$ of linear equivalence
classes of divisors on $C$ of degree zero. Thus, the problem Weil faced was
that of realizing the abstract group $\Jac(C)$ as a projective variety in some
natural way --- over $\mathbb{C}{}$ the theta functions provide a projective
embedding of $\Jac(C)$. He was not able to do this, but as he writes
(\OE uvres I, p.556):

\begin{quote}
In the spring of 1941, I was living in Princeton\ldots\ I often worked in
Chevalley's office in Fine Hall; of course he was aware of my attempts to
\textquotedblleft define\textquotedblright\ the jacobian, i.e., to construct
algebraically a projective embedding. One day, coming into his office, I
surprised him by telling him that there was no need; for the jacobian
everything comes down to its local properties, and a piece of the jacobian,
joined to the group property (the addition of divisor classes) suffices amply
for that. The idea came to me on the way to Fine Hall. It was both the concept
of an \textquotedblleft abstract variety\textquotedblright\ which had just
taken shape, and the construction of the jacobian as an algebraic
group.\footnote{En ce printemps de 1941, je vivais \`{a} Princeton\ldots\ Je
travaillais souvent dans le bureau de Chevalley \`{a} Fine Hall; bien entendu
il \'{e}tait au courant de mes tentatives pour \textquotedblleft
d\'{e}finir\textquotedblright\ la jacobienne, c'est-\`{a}-dire pour en
construire alg\'{e}briquement un plongement projectif. Un jour, entrant chez
lui, je le surpris en lui disant qu'il n'en \'{e}tait nul besoin; sur la
jacobienne tout se ram\`{e}ne \`{a} des propri\'{e}t\'{e}s locales, et un
morceau de jacobienne, joint \`{a} la propri\'{e}t\'{e} de groupe (l'addition
des classes de diviseurs) y suffit amplement. L'id\'{e}e m'en \'{e}tait venue
sur le chemin de Fine Hall. C'\'{e}tait \`{a} la fois la notion de
"vari\'{e}t\'{e} abstraite" qui venait de prendre forme, et la construction de
la jacobienne en tant que groupe alg\'{e}brique, telle qu'elle figure dans
[1948b].}
\end{quote}

In order to explain Weil's idea, we need the notion of a \textit{birational
group} over $k$. This is a nonsingular variety $V$ together with a rational
map $m\colon V\times V\da V$ such that

\begin{enumerate}
\item $m$ is associative (that is, $(ab)c=a(bc)$ whenever both terms are defined);

\item the rational maps $(a,b)\mapsto(a,ab)$ and $(a,b)\mapsto(b,ab)$ from
$V\times V$ to $V\times V$ are both birational.
\end{enumerate}

\begin{theorem}
\label{r5}Let $(V,m)$ be a birational group $V$ over $k$. Then there exists a
group variety $G$ over $k$ and a birational map $f\colon V\da G$ such that
$f(ab)=f(a)f(b)$ whenever $ab$ is defined; the pair $(G,f)$ is unique up to a
unique isomorphism. (Weil 1948b, n$^{\circ}$33, Thm 15; Weil 1955).
\end{theorem}

Let $C^{(g)}$ denote the symmetric product of $g$ copies of $C$, i.e., the
quotient of $C^{g}$ by the action of the symmetric group $S_{g}$. It is a
smooth variety of dimension $g$ over $k$. The set $C^{(g)}(k)$ consists of the
unordered $g$-tuples of closed points on $C$, which we can regard as positive
divisors of degree $g$ on $C$.

Fix a $P\in C(k)$. Let $D$ be a positive divisor of degree $g$ on $C$.
According to the Riemann-Roch theorem%
\[
l(D)=1+l(K_{C}-D)\geq1,
\]
and one can show that equality holds on a dense open subset of $C^{(g)}$.
Similarly, if $D$ is a positive divisor of degree $2g$ on $C$, then
$l(D-gP)\geq1$ and equality holds on a nonempty open subset $U^{\prime}$ of
$C^{(2g)}$. Let $U$ be the inverse image of $U^{\prime}$ under the obvious map
$C^{(g)}\times C^{(g)}\rightarrow C^{(2g)}$. Then $U$ is a dense open subset
of $C^{(g)}\times C^{(g)}$ with the property that $l(D+D^{\prime}-gP)=1$ for
all $(D,D^{\prime})\in U(k)$.

Now let $(D,D^{\prime})\in U(k)$. Because $l(D+D^{\prime}-gP)>0$, there exists
a positive divisor $D^{\prime\prime}$ on $C$ linearly equivalent to
$D+D^{\prime}-gP$, and because $l(D+D^{\prime}-gP)=1$, the divisor
$D^{\prime\prime}$ is unique. Therefore, there is a well-defined law of
composition%
\[
(D,D^{\prime})\mapsto D^{\prime\prime}\colon U\times U\rightarrow C^{(g)}(k).
\]

\begin{theorem}
\label{r6}There exists a unique rational map
\[
m\colon C^{(g)}\times C^{(g)}\da C^{(g)}%
\]
whose domain of definition contains the subset $U$ and which is such that, for
all fields $K$ containing $k$ and all $(D,D^{\prime})$ in $U(K)$,
$m(D,D^{\prime})\sim D+D^{\prime}-gP$; moreover $m$ makes $C^{(g)}$ into a
birational group.
\end{theorem}

This can be proved, according to taste, by using generic points (Weil 1948b)
or functors (Milne 1986). On combining the two theorems, we obtain a group
variety $J$ over $k$ birationally equivalent to $C^{(g)}$ with its partial
group structure. This is the jacobian variety.

\subsubsection{Notes}

\begin{plain}
\label{r35}Weil (1948b Thm 16, et seqq.) proved that the jacobian variety is
complete, a notion that he himself had introduced. Chow (1954) gave a direct
construction of the jacobian variety as a projective variety over the same
base field as the curve. Weil (1950) announced the existence of two abelian
varieties attached to a complete normal algebraic variety, namely, a Picard
variety and an Albanese variety. For a curve, both varieties equal the
jacobian variety, but in general they are distinct dual abelian varieties.
This led to a series of papers on these topics (Matsusaka, Chow, Chevalley,
Nishi, Cartier, \ldots) culminating in Grothendieck's general construction of
the Picard scheme (see Kleiman 2005).
\end{plain}

\begin{plain}
\label{r36}Weil's construction of an algebraic group from a birational group
has proved useful in other contexts, for example, in the construction of the
N\'{e}ron model of abelian variety (N\'{e}ron 1964; Artin 1964, 1986; Bosch et
al. 1990, Chapter 5).
\end{plain}

\subsection{The endomorphism algebra of an abelian variety}

The exposition in this subsection includes improvements explained by Weil in
his course at the University of Chicago, 1954-55, and other articles, which
were incorporated in Lang 1959. In particular, we use that an abelian variety
admits a projective embedding, and we use the following consequence of the
theorem of the cube: Let $f,g,h$ be regular maps from a variety $V$ to an
abelian variety $A$, and let $D$ be a divisor on $A$; then%
\begin{equation}
(f+g+h)^{\ast}D-(f+g)^{\ast}D-(g+h)^{\ast}D-(f+h)^{\ast}D+f^{\ast}D+g^{\ast
}D+h^{\ast}D\sim0. \label{e24}%
\end{equation}
We shall also need to use the intersection theory for divisors on smooth
projective varieties. As noted earlier, this is quite elementary.

\subsubsection{Review of the theory over $\mathbb{C}{}$}

Let $A$ be an abelian variety of dimension $g$ over $\mathbb{C}{}$. Then
$A(\mathbb{C}{})\simeq T/\Lambda$ where $T$ is the tangent space to $A$ at
$0$, and $\Lambda=H_{1}(A,\mathbb{Z}{})$. The endomorphism ring $\End(A)$ of
$A$ acts faithfully on $\Lambda$, and so it is a free $\mathbb{Z}{}$-module of
rank $\leq4g$. We define the characteristic polynomial $P_{\alpha}(T)$ of an
endomorphism $\alpha$ of $A$ to be its characteristic polynomial
$\det(T-\alpha\mid\Lambda)$ on $\Lambda$. It is the unique polynomial in
$\mathbb{\mathbb{Z}{}}{}[T]$ such that%
\[
P_{\alpha}(m)=\deg(m-\alpha)
\]
for all $m\in\mathbb{Z}{}$. It can also be described as the characteristic
polynomial of $\alpha$ acting on $A(\mathbb{C}{})_{\mathrm{tors}}\simeq\left(
\mathbb{Q}{}\otimes\Lambda\right)  /\Lambda$.\footnote{Choose an isomorphism
$(\mathbb{Q}{}\otimes\Lambda/\Lambda)\rightarrow(\mathbb{Q}{}/\mathbb{Z}%
{})^{2g}$, and note that $\End(\mathbb{Q}{}/\mathbb{Z}{})\simeq\mathbb{\hat
{Z}}{}$. The action of $\alpha$ on $\mathbb{Q}{}\otimes\Lambda/\Lambda$
defines an element of $M_{2g}(\mathbb{\hat{Z})}{}$, whose characteristic
polynomial is $P_{\alpha}(T).$}

Choose a Riemann form for $A$. Let $H$ be the associated positive-definite
hermitian form on the complex vector space $T$ and let $\alpha\mapsto
\alpha^{\dagger}$ be the associated Rosati involution on $\End^{0}%
(A)\overset{\textup{{\tiny def}}}{=}\End(A)\otimes\mathbb{Q}{}$. Then%
\[
H(\alpha x,y)=H(x,\alpha^{\dagger}y),\quad\text{all }x,y\in T\text{, }%
\alpha\in\End^{0}(A),
\]
and so $^{\dagger}$ is a positive involution on $\End^{0}(A)$, i.e., the trace
pairing
\[
(\alpha,\beta)\mapsto\Tr(\alpha\circ\beta^{\dagger})\colon\End^{0}%
(A)\times\End^{0}(A)\rightarrow\mathbb{Q}{}%
\]
is positive definite (see, for example, Rosen 1986).

Remarkably, Weil was able to extend these statements to abelian varieties over
arbitrary base fields.

\subsubsection{The characteristic polynomial of an endomorphism}

\begin{proposition}
\label{r15}For all integers $n\geq1$, the map $n_{A}\colon A\rightarrow A$ has
degree $n^{2g}$. Therefore, for $\ell\neq\mathrm{char}(k)$, the $\mathbb{Z}%
{}_{\ell}$-module $T_{\ell}A\overset{\textup{{\tiny def}}}{=}\varprojlim
_{n}A_{\ell^{n}}(k^{\mathrm{al}})$ is free of rank $2g$.
\end{proposition}

\begin{proof}
Let $D$ be an ample divisor on $A$ (e.g., a hyperplane section under some
projective embedding); then $D^{g}$ is a positive zero-cycle, and $(D^{g}%
)\neq0$. After possibly replacing $D$ with $D+(-1)_{A}^{\ast}D$, we may
suppose that $D$ is symmetric, i.e., $D\sim(-1)_{A}^{\ast}D$. An induction
argument using (\ref{e24}) shows that $n_{A}^{\ast}D\sim n^{2}D$, and so%
\[
(n_{A}^{\ast}D^{g})=(n_{A}^{\ast}D\cdot\ldots\cdot n_{A}^{\ast}D)\sim
(n^{2}D\cdot\ldots\cdot n^{2}D)\sim n^{2g}(D^{g})\text{.}%
\]
Therefore $n_{A}^{\ast}$ has degree $n^{2g}$.
\end{proof}

A map $f\colon W\rightarrow Q$ on a vector space $W$ over a field $Q$ is said
to be \textit{polynomial (of degree} $d$) if, for every finite linearly
independent set $\{e_{1},...,e_{n}\}$ of elements of $V$,
\[
f(a_{1}e_{1}+\cdots+a_{n}e_{n})=P(a_{1},\ldots,a_{n}),\quad x_{i}\in Q,
\]
for some $P\in Q[X_{1},\ldots,X_{n}]$ (of degree $d$). To show that a map $f$
is polynomial, it suffices to check that, for all $v,w\in W$, the map
$x\mapsto f(xv+w)\colon Q\rightarrow Q$ is a polynomial in $x$.

\begin{lemma}
\label{r8}Let $A$ be an abelian variety of dimension $g$. The map
$\alpha\mapsto\deg(\alpha)\colon\End^{0}(A)\rightarrow\mathbb{Q}$ is a
polynomial function of degree $2g$ on $\End^{0}(A){.}$
\end{lemma}

\begin{proof}
Note that $\deg(n\beta)=\deg(n_{A})\deg(\beta)=n^{2g}\deg(\beta)$, and so it
suffices to prove that $\deg(n\beta+\alpha)$ (for $n\in\mathbb{Z}{}$ and
$\beta,\alpha\in\End(A)$) is polynomial in $n$ of degree $\leq2g$. Let $D$ be
a symmetric ample divisor on $A$. A direct calculation using (\ref{e24}) shows
shows that%
\begin{equation}
\deg(n\beta+\alpha)(D^{g})=(n(n-1))^{g}(D^{g})+\text{terms of lower degree in
}n. \label{e32}%
\end{equation}
As $(D^{g})\neq0$ this completes the proof.
\end{proof}

\begin{theorem}
\label{r9}Let $\alpha\in\End(A)$. There is a unique monic polynomial
$P_{\alpha}(T)\in\mathbb{Z}{}[T]$ of degree $2g$ such that $P_{\alpha}%
(n)=\deg(n_{A}-\alpha)$ for all integers $n$.
\end{theorem}

\begin{proof}
If $P_{1}$ and $P_{2}$ both have this property, then $P_{1}-P_{2}$ has
infinitely many roots, and so is zero. For the existence, take $\beta=1_{A}$
in (\ref{e32}).
\end{proof}

We call $P_{\alpha}$ the \textit{characteristic polynomial} of $\alpha$ and we
define the \textit{trace }$\Tr(\alpha)$ of $\alpha$ by the equation
\[
P_{\alpha}(T)=T^{2g}-\Tr(\alpha)T^{2g-1}+\cdots+\deg(\alpha).
\]

\subsubsection{The endomorphism ring of an abelian variety}

Let $A$ and $B$ be abelian varieties over $k$, and let $\ell$ be a prime
$\neq\mathrm{char}(k)$. The family of $\ell$-power torsion points in
$A(k^{\mathrm{al}})$ is dense, and so the map%
\[
\Hom(A,B)\rightarrow\Hom_{\mathbb{Z}{}_{\ell}}(T_{\ell}A,T_{\ell}B)
\]
is injective. Unfortunately, this doesn't show that $\Hom(A,B)$ is finitely
generated over $\mathbb{Z}{}$.

\begin{lemma}
\label{r10}Let $\alpha\in\Hom(A,B)$; if $\alpha$ is divisible by $\ell^{n}$ in
$\Hom(T_{\ell}A,T_{\ell}B)$, then it is divisible by $\ell^{n}$ in $\Hom(A,B)$.
\end{lemma}

\begin{proof}
For each $n$, there is an exact sequence%
\[
0\rightarrow A_{\ell^{n}}\rightarrow A\overset{\ell^{n}}{\longrightarrow
}A\rightarrow0.
\]
The hypothesis implies that $\alpha$ is zero on $A_{\ell^{n}}$, and so it
factors through the quotient map $A\overset{\ell^{n}}{\longrightarrow}A$.
\end{proof}

\begin{theorem}
\label{r11}The natural map
\begin{equation}
{\mathbb{Z}_{\ell}}\otimes\Hom(A,B){\rightarrow}\Hom(T_{\ell}A,T_{\ell}B)
\label{e7}%
\end{equation}
is injective (with torsion-free cokernel). Hence $\Hom(A,B)$ is a free
${\mathbb{Z}}$-module of finite rank $\leq4\dim(A)\dim(B)$.
\end{theorem}

\begin{proof}
The essential case is that with $A=B$ and $A$ simple. Let $e_{1},\ldots,e_{m}$
be elements of $\End(A)$ linearly independent over $\mathbb{Z}{}$; we have to
show that $T_{\ell}(e_{1}),\ldots,T_{\ell}(e_{m})$ are linearly independent
over $\mathbb{Z}{}_{\ell}$.

Let $M$ (resp. $\mathbb{Q}{}M$) denote the $\mathbb{Z}{}$-module (resp.
$\mathbb{Q}{}$-vector space) generated in $\End^{0}(A)$ by the $e_{i}$.
Because $A$ is simple, every nonzero endomorphism $\alpha$ of $A$ is an
isogeny, and so $\deg(\alpha)$ is an integer $>0$. The map $\deg
\colon{\mathbb{Q}}M\rightarrow\mathbb{Q}$ is continuous for the real topology
because it is a polynomial function (\ref{r8}), and so $U\overset
{\textup{{\tiny def}}}{=}\{\alpha\in\mathbb{Q}{}M\mid\deg(\alpha)<1\}$ is an
open neighbourhood of $0$. As
\[
\left(  \mathbb{Q}M\cap\End(A)\right)  \cap U\subset\End(A)\cap U=0,
\]
we see that ${\mathbb{Q}}M\cap\End(A)$ is discrete in ${\mathbb{Q}}M$, which
implies that it is finitely generated as a ${\mathbb{Z}}$-module. Hence there
exists an integer $N>0$ such that
\begin{equation}
N(\mathbb{Q}{}M\cap\End(A))\subset M. \label{e8}%
\end{equation}

Suppose that there exist $a_{i}\in\mathbb{Z}_{\ell}$, not all zero, such that
$\sum a_{i}T_{\ell}(e_{i})=0$. For a fixed $m\in\mathbb{N}{}$, we can find
integers $n_{i}$ sufficiently close to the $a_{i}$ that the sum $\sum
n_{i}e_{i}$ is divisible by $\ell^{m}$ in $\End(T_{\ell}A)$, and hence in
$\End(A)$. Therefore%
\[
\tstyle\sum(n_{i}/\ell^{m})e_{i}\in\mathbb{Q}{}M\cap\End(A)\text{,}%
\]
and so $N\sum(n_{i}/\ell^{m})e_{i}\in M$, i.e., $n_{i}N/\ell^{m}\in
\mathbb{Z}{}$ and $\ord_{\ell}(n_{i})+\ord_{\ell}(N)\geq m$ for all $i.$ But
if $n_{i}$ is close to $a_{i}$, then $\ord_{\ell}(n_{i})=\ord_{\ell}(a_{i})$,
and so $\ord_{\ell}(a_{i})+\ord_{\ell}(N)\geq m$. As $m$ was arbitrary, we
have a contradiction.
\end{proof}

\subsubsection{The $\ell$-adic representation}

\begin{proposition}
\label{e9} For all $\ell\neq\mathrm{char}(k),P_{\alpha}(T)$ is the
characteristic polynomial of $\alpha$ acting on $T_{\ell}A$; in particular,
$\det(\alpha\mid T_{\ell}A)=\deg(\alpha)$.
\end{proposition}

\begin{proof}
For an endomorphism $\alpha$ of $A$,
\begin{equation}
\left\vert \det(T_{\ell}\alpha)\right\vert _{\ell}=\left\vert \deg
(\alpha)\right\vert _{\ell}. \label{e30}%
\end{equation}
To continue, we shall need an elementary lemma (Weil 1948b, n$^{\circ}$ 68,
Lemme 12):

\bquote A polynomial $P(T)\in\mathbb{Q}{}_{\ell}[T]$, with roots $a_{1}%
,\ldots,a_{d}$ in $\mathbb{Q}{}_{\ell}^{\mathrm{al}}$, is uniquely determined
by the numbers $\left\vert \prod\nolimits_{i=1}^{d}F(a_{i})\right\vert _{\ell
}$ as $F$ runs through the elements of $\mathbb{Z}{}[T].$\equote

\noindent For the polynomial $P_{\alpha}(T)$,
\[
\prod\nolimits_{i=1}^{d}F(a_{i})=\pm\deg(F(\alpha)),
\]
and for the characteristic polynomial $P_{\alpha,\ell}(T)$ of $\alpha$ on
$T_{\ell}A$,%
\[
\left\vert \prod\nolimits_{i=1}^{d}F(a_{i})\right\vert _{\ell}=\left\vert
\det(F(T_{\ell}\alpha))\right\vert _{\ell}.
\]
Therefore, (\ref{e30}) shows that $P_{\alpha}(T)$ and $P_{\alpha,\ell}(T)$
coincide as elements of $\mathbb{Q}{}_{\ell}[T]$.
\end{proof}

\subsubsection{Positivity}

An ample divisor $D$ on $A$ defines an isogeny $\varphi_{D}\colon A\rightarrow
A^{\vee}$ from $A$ to the dual abelian variety $A^{\vee}\overset
{\textup{{\tiny def}}}{=}\Pic^{0}(A)$, namely, $a\mapsto\lbrack D_{a}-D]$
where $D_{a}$ is the translate of $D$ by $a$. A polarization of $A$ is an
isogeny $\lambda\colon A\rightarrow A^{\vee}$ that becomes of the form
$\varphi_{D}$ over $k^{\mathrm{al}}$. As $\lambda$ is an isogeny, it has an
inverse in $\Hom^{0}(A^{\vee},A)$. The Rosati involution on $\End^{0}(A)$
corresponding to $\lambda$ is%
\[
\alpha\mapsto\alpha^{\dagger}\overset{\textup{{\tiny def}}}{=}\lambda
^{-1}\circ\alpha^{\vee}\circ\lambda.
\]
It has the following properties:%
\[
(\alpha+\beta)^{\dagger}=\alpha^{\dagger}+\beta^{\dagger},\quad(\alpha
\beta)^{\dagger}=\beta^{\dagger}\alpha^{\dagger},\quad a^{\dagger}=a\text{ for
all }a\in\mathbb{Q}{}.
\]

\begin{theorem}
\label{r12}The Rosati involution on $E\overset{\textup{{\tiny def}}}%
{=}\End^{0}(A)$ is positive, i.e., the pairing%
\[
\alpha,\beta\mapsto\Tr_{E/\mathbb{Q}{}}(\alpha\circ\beta^{\dagger})\colon
E\times E\rightarrow\mathbb{Q}{}%
\]
is positive definite.
\end{theorem}

\begin{proof}
We have to show that $\Tr(\alpha\circ\alpha^{\dagger})>0$ for all $\alpha
\neq0$. Let $D$ be the ample divisor corresponding to the polarization used in
the definition of $\dagger$. A direct calculation shows that
\[
\mathrm{Tr}(\alpha\circ\alpha^{\dagger})=\frac{2g}{(D^{g})}(D^{g-1}\cdot
\alpha^{-1}(D))
\]
(Lang 1959, V, \S 3, Thm 1). I claim that $(D^{g-1}\cdot\alpha^{-1}(D))>0$. We
may suppose that $D$ is a hyperplane section of $A$ relative to some
projective embedding. There exist hyperplane sections $H_{1},\ldots,H_{g-1}$
such that $H_{1}\cap\ldots\cap H_{g-1}\cap\alpha^{-1}D$ has dimension zero. By
dimension theory, the intersection is nonempty, and so
\[
(H_{1}\cdot\ldots\cdot H_{g-1}\cdot\alpha^{-1}D)>0.
\]
As $D\sim H_{i}$ for all $i$, we have $(D^{g-1}\cdot\alpha^{-1}D)=(H_{1}%
\cdot\ldots\cdot H_{g-1}\cdot\alpha^{-1}D)>0$.
\end{proof}

\label{semisimple}Recall that the \textit{radical} $\mathrm{rad}(R)$ of a ring
$R$ is the intersection of the maximal left ideals in $R$. It is a two-sided
ideal, and it is nilpotent if $R$ is artinian. An algebra $R$ of finite
dimension over a field $Q$ of characteristic zero is semisimple if and only if
$\mathrm{rad}(R)=0$.

\begin{corollary}
\label{r13}The $\mathbb{Q}{}$-algebra $\End^{0}(A)$ is semisimple.
\end{corollary}

\begin{proof}
If not, the radical of $\End^{0}(A)$ contains a nonzero element $\alpha$. As
$\beta\overset{\textup{{\tiny def}}}{=}\alpha\alpha^{\dagger}$ has nonzero
trace, it is a nonzero element of the radical. It is symmetric, and so
$\beta^{2}=\beta\beta^{\dagger}\neq0$, $\beta^{4}=(\beta^{2})^{2}\neq0$,
\ldots. Therefore $\beta$ is not nilpotent, which is a contradiction.
\end{proof}

\begin{corollary}
\label{r39}Let $\alpha$ be an endomorphism of $A$ such that $\alpha\circ
\alpha^{\dagger}=q_{A}$, $q\in\mathbb{Z}$. For every homomorphism $\rho
\colon\mathbb{Q}{}[\alpha]\rightarrow\mathbb{C}{}$,
\[
\rho(\alpha^{\dagger})=\overline{\rho(\alpha)}\text{ and }\left\vert
\rho\alpha\right\vert =q^{1/2}.
\]

\end{corollary}

\begin{proof}
As $\mathbb{Q}{}[\alpha]$ is stable under $^{\dagger}$, it is semisimple, and
hence a product of fields. Therefore $\mathbb{R}{}\otimes\mathbb{Q}{}[\alpha]$
is a product of copies of $\mathbb{R}{}$ and $\mathbb{C}{}$. The involution
$^{\dagger}$ on $\mathbb{Q}{}[\alpha]$ extends by continuity to an involution
of $\mathbb{R}{}\otimes\mathbb{Q}{}[\alpha]$ having the property that
$\Tr(\beta\circ\beta^{\dagger})\geq0$ for all $\beta\in\mathbb{R}{}%
\otimes\mathbb{Q}{}[\alpha]$ with inequality holding on a dense subset. The
only such involution preserves the each factor and acts as the identity on the
real factors and as complex conjugation on the complex factors. This proves
the first statement, and the second follows:%
\[
q=\rho(q_{A})=\rho(\alpha\circ\alpha^{\dagger})=\rho(\alpha)\cdot
\overline{\rho(\alpha)}=\left\vert \rho(\alpha)^{2}\right\vert \text{.}%
\]

\end{proof}

\begin{theorem}
\label{r14}Let $A$ be an abelian variety over $k=\mathbb{F}{}_{q}$, and let
$\pi\in\End(A)$ be the Frobenius endomorphism. Let $^{\dagger}$ be a Rosati
involution on $\End^{0}(A)$. For every homomorphism $\rho\colon\mathbb{Q}%
[\pi]\rightarrow\mathbb{C}$,
\[
\rho(\alpha^{\dagger})=\overline{\rho(\pi)}\text{ and }\left\vert \rho
\pi\right\vert =q^{1/2}.
\]

\end{theorem}

\begin{proof}
This will follow from (\ref{r39}) once we show that $\pi\circ\pi^{\dagger
}=q_{A}$. This can be proved by a direct calculation, but it is more
instructive to use the Weil pairing. Let $\lambda$ be the polarization
defining $^{\dagger}$. The Weil pairing is a nondegenerate skew-symmetric
pairing%
\[
e^{\lambda}\colon T_{l}A\times T_{l}A\rightarrow T_{l}\mathbb{G}_{m}%
\]
with the property that%
\[
e^{\lambda}(\alpha x,y)=e^{\lambda}(x,\alpha^{\dagger}y)\quad\quad
\text{(}x,y\in T_{l}A,\quad\alpha\in\End(A)).
\]
Now%
\[
e^{\lambda}(x,(\pi^{\dagger}\circ\pi)y)=e^{\lambda}(\pi x,\pi y)=\pi
e^{\lambda}(x,y)=qe^{\lambda}(x,y)=e^{\lambda}(x,q_{A}y).
\]
Therefore $\pi^{\dagger}\circ\pi$ and $q_{A}$ agree as endomorphisms of
$T_{l}A$, and hence as endomorphisms of $A$.
\end{proof}

The Riemann hypothesis for a curve follows from applying (\ref{r14}) to the
jacobian variety of $C$.

\begin{summary}
\label{r33}Let $A$ be an abelian variety of dimension $g$ over $k=\mathbb{F}%
{}_{q}$, and let $P(T)$ be the characteristic polynomial of the Frobenius
endomorphism $\pi$. Then $P(T)$ is a monic polynomial of degree $2g$ with
integer coefficients, and its roots $a_{1},\ldots,a_{2g}$ have absolute value
$q^{1/2}$. For all $\ell\neq p$,%
\[
P(T)=\det(T-\pi\mid V_{\ell}A),\quad V_{\ell}A\overset{\textup{{\tiny def}}%
}{=}T_{\ell}A\otimes\mathbb{Q}{}_{\ell}.
\]
For all $n\geq1$,%
\begin{equation}
\prod\nolimits_{i=1}^{2g}(1-a_{i}^{n})=|A(k_{n})|\quad\text{where }%
[k_{n}\colon k]=n, \label{e25}%
\end{equation}
and this condition determines $P$. We have%
\[
P(T)=q^{g}\cdot T^{2g}\cdot P(1/qT).
\]

\end{summary}

\begin{note}
\label{r52}Essentially everything in this subsection is in Weil 1948b, but, as
noted at the start, some of the proofs have been simplified by using Weil's
later work. In 1948 Weil didn't know that abelian varieties are projective,
and he proved (\ref{r12}) first for jacobians, where the Rosati involution is
obvious. The more direct proof of (\ref{r12}) given above is from Lang 1957.
\end{note}

\subsection{The Weil conjectures (Weil 1949)}

Weil studied equations of the form%
\begin{equation}
a_{0}X_{0}^{m_{0}}+\cdots+a_{r}X_{r}^{m_{r}}=b \label{e12}%
\end{equation}
over a finite field $k$, and obtained an expression in terms of Gauss sums for
the number of solutions in $k$. Using a relation, due to Davenport and Hasse
(1935), between Gauss sums in a finite field and in its extensions, he was
able to obtain a simple expression for the formal power series
(\textquotedblleft generating function\textquotedblright)%
\[
\sum_{1}^{\infty}\left(  \#\text{ solutions of (\ref{e12}) in }k_{n}\right)
T^{n}.
\]
In the homogeneous case%
\begin{equation}
a_{0}X_{0}^{m}+\cdots+a_{r}X_{r}^{m}=0\text{,} \label{e31}%
\end{equation}
he found that%
\[
\sum_{1}^{\infty}N_{n}T^{n-1}=\frac{d}{dT}\log\left(  \frac{1}{(1-T)\cdots
(1-q^{r-1}T)}\right)  +(-1)^{r}\frac{d}{dT}\log P(T)
\]
where $P(T)$ is a polynomial of degree $A$ equal to the number of solutions in
rational numbers $\alpha_{i}$ of the system $n\alpha_{i}\equiv1$, $\sum
\alpha_{i}\equiv0$ (mod $1$), $0<\alpha_{i}<1$. Dolbeault was able to tell
Weil that the Betti numbers of a hypersurface defined by an equation of the
form (\ref{e31}) over $\mathbb{C}{}$ are the coefficients of the polynomial%
\[
1+T^{2}+\cdots+T^{2r-2}+AT^{r}\text{.}%
\]
As he wrote (1949, p.409):

\begin{quote}
This, and other examples which we cannot discuss here, seems to lend some
support to the following conjectural statements, which are known to be true
for curves, but which I have not so far been able to prove for varieties of
higher dimension.\footnote{Weil's offhand announcement of the conjectures
misled one reviewer into stating that \textquotedblleft the purpose of this
paper is to give an exposition of known results concerning the equation
$a_{0}X_{0}^{m_{0}}+\cdots+a_{r}X_{r}^{m_{r}}=b$\textquotedblright\ (MR
29393).}
\end{quote}

\noindent\textsc{Weil Conjectures.}\label{r48} Let $V$ be a nonsingular
projective variety of dimension $d$ defined over a finite field $k$ with $q$
elements. Let $N_{n}$ denote the number points of $V$ rational over the
extension of $k$ of degree $n$. Define the zeta function of $C$ to be the
power series $Z(V,T)\in\mathbb{Q}{}[[T]]$ such that%

\begin{equation}
\frac{d\log Z(V,T)}{dT}=\sum\nolimits_{1}^{\infty}N_{n}T^{n-1}. \label{e13}%
\end{equation}
Then:

\renewcommand{\theenumi}{W{\arabic{enumi}}}

\begin{enumerate}
\item (rationality) $Z(V,T)$ is a rational function in $T$;

\item (functional equation) $Z(V,T)$ satisfies the functional equation%
\begin{equation}
Z(V,1/q^{d}T)=\pm q^{d\chi/2}\cdot T^{\chi}\cdot Z(V,T) \label{e14}%
\end{equation}
with $\chi$ equal to the Euler-Poincar\'{e} characteristic of $V$
(intersection number of the diagonal with itself);

\item (integrality) we have
\begin{equation}
Z(V,T)=\frac{P_{1}(T)\cdots P_{2d-1}(T)}{(1-T)P_{3}(T)\cdots(1-q^{d}T)}
\label{e27}%
\end{equation}
with $P_{r}(T)\in\mathbb{Z}[T]$ for all $r;$

\item (Riemann hypothesis) write $P_{r}(T)=\prod\nolimits_{1}^{B_{r}}%
(1-\alpha_{ri}T)$; then%
\[
\left\vert \alpha_{ri}\right\vert =q^{r/2}%
\]
for all $r$, $i$;

\item (Betti numbers) call the degrees $B_{r}$ of the polynomials $P_{r}$ the
\textit{Betti numbers} of $V$; if $V$ is obtained by reduction modulo a prime
ideal $\mathfrak{p}{}$ in a number field $K$ from a nonsingular projective
variety $\tilde{V}$ over $K$, then the Betti numbers of $V$ are equal to the
Betti numbers of $\tilde{V}$ (i.e., of the complex manifold $\tilde
{V}(\mathbb{C}{})$).
\end{enumerate}

\renewcommand{\theenumi}{{\alph{enumi}}}

Weil's considerations led him to the conclusion, startling at the time, that
the \textquotedblleft Betti numbers\textquotedblright\ of an algebraic variety
have a purely algebraic meaning. For a curve $C$, the Betti numbers are $1$,
$2g$, $1$ where $g$ is the smallest natural number for which the Riemann
inequality%
\[
l(D)\geq\deg(D)-g+1
\]
holds, but for dimension $>1$?

No cohomology groups appear in Weil's article and no cohomology theory is
conjectured to exist,\footnote{Weil may have been aware that there cannot
exist a good cohomology theory with $\mathbb{Q}{}$-coefficients; see
(\ref{r51}). Recall also that Weil was careful to distinguish conjecture from
speculation.} but Weil was certainly aware that cohomology provides a
heuristic explanation of (W1--W3). For example, let $\varphi$ be a finite map
of degree $\delta$ from a nonsingular complete variety $V$ to itself over
$\mathbb{C}{}$. For each $n$, let $N_{n}$ be the number of solutions, supposed
finite, of the equation $P=\varphi^{n}(P)$ or, better, the intersection number
$(\Delta\cdot\Gamma_{\varphi^{n}})$. Then standard arguments show that the
power series $Z(T)$ defined by (\ref{e13}) satisfies the functional equation%
\begin{equation}
Z(1/\delta T)=\pm(\delta^{1/2}T)^{\chi}\cdot Z(T) \label{e15}%
\end{equation}
where $\chi$ is the Euler-Poincar\'{e} characteristic of $V$ (Weil, \OE uvres
I, p.568). For a variety $V$ of dimension $d$, the Frobenius map has degree
$q^{d}$, and so (\ref{e15}) suggests (\ref{e14}).

\begin{example}
\label{r49}Let $A$ be an abelian variety of dimension $g$ over a field $k$
with $q$ elements, and let $a_{1},\ldots,a_{2g}$ be the roots of the
characteristic polynomial of the Frobenius map. Let $P_{r}(T)=\prod
(1-a_{i,r}T)$ where the $a_{i,r}$ run through the products%
\[
a_{i_{1}}\cdots a_{i_{r}},\quad0<i_{1}<\cdots<i_{r}\leq2g.
\]
It follows from (\ref{e25}) that%
\[
Z(A,T)=\frac{P_{1}(T)\cdots P_{2g-1}(T)}{(1-T)P_{3}(T)\cdots(1-q^{g}%
T)}\text{.}%
\]
The $r$th Betti number of an abelian variety of dimension $g$ over
$\mathbb{C}{}$ is $\left(  \!%
\begin{smallmatrix}
g\\
r
\end{smallmatrix}
\!\right)  $, which equals $\deg(P_{r})$, and so the Weil conjectures hold for
$A$.
\end{example}

\begin{aside}
\label{r50}To say that $Z(V,T)$ is rational means that there exist elements
$c_{1},\ldots,c_{m}\in\mathbb{Q}$,${}$ not all zero, and an integer $n_{0}$
such that, for all $n\geq n_{0}$,%
\[
c_{1}N_{n}+c_{2}N_{n+1}+\cdots+c_{m}N_{n+m-1}=0.
\]
In particular, the set of $N_{n}$ can be computed inductively from a finite subset.
\end{aside}

\section{Weil cohomology}

After Weil stated his conjectures, the conventional wisdom eventually became
that, in order to prove the rationality of the zeta functions, it was
necessary to define a good \textquotedblleft Weil \ cohomology
theory\textquotedblright\ for algebraic varieties. In 1959, Dwork\label{Dwork}
surprised everyone by finding an elementary proof, depending only on $p$-adic
analysis, of the rationality of the zeta function of an arbitrary variety over
a finite field (Dwork 1960).\footnote{\textquotedblleft His method consists in
assuming that the variety is a hypersurface in affine space (every variety is
birationally isomorphic to such a hypersurface, and an easy unscrewing lets us
pass from there to the general case); in this case he does a computation with
\textquotedblleft Gauss sums\textquotedblright\ analogous to that of Weil for
an equation $\sum a_{i}x_{i}^{n_{i}}=b$. Of course, Weil himself had tried to
extend his method and had got nowhere\ldots\textquotedblright\ Serre,
\textit{Grothendieck-Serre Correspondence} p.102. Dwork expressed $Z(V,T)$ as
a quotient of two $p$-adically entire functions, and showed that every such
function with nonzero radius of convergence is rational.} Nevertheless, a
complete understanding of the zeta function requires a cohomology theory,
whose interest anyway transcends zeta functions, and so the search continued.

After listing the axioms for a Weil cohomology theory, we explain how the
existence of such a theory implies the first three Weil conjectures. Then we
describe the standard Weil cohomologies.

For a smooth projective variety $V$, we write $C_{\mathrm{rat}}^{r}(V)$ for
the $\mathbb{Q}{}$-vector space of algebraic cycles of codimension $r$ (with
$\mathbb{Q}{}$-coefficients) modulo rational equivalence.

\subsection{The axioms for a Weil cohomology theory}

We fix an algebraically closed \textquotedblleft base\textquotedblright\ field
$k$ and a \textquotedblleft coefficient\textquotedblright\ field $Q$ of
characteristic zero. A \textit{Weil cohomology theory} is a contravariant
functor $V\rightsquigarrow H^{\ast}(V)$ from the category of nonsingular
projective varieties over $k$ to the category of finite-dimensional, graded,
anti-commutative $Q$-algebras carrying disjoint unions to direct sums and
admitting a Poincar\'{e} duality, a K\"{u}nneth formula, and a cycle map in
the following sense.

\begin{description}
\item[Poincar\'{e} duality] Let $V$ be connected of dimension $d$.

\begin{enumerate}
\item The $Q$-vector spaces $H^{r}(V)$ are zero except for $0\leq r\leq2d$,
and $H^{2d}(V)$ has dimension $1$.

\item Let $Q(-1)=H^{2}(\mathbb{P}^{1})$, and let $Q(1)$ denote its dual. For a
$Q$-vector space $V$ and integer $m$, we let $V(m)$ equal $V\otimes
_{Q}Q(1)^{\otimes m}$ or $V\otimes_{Q}Q(-1)^{\otimes-m}$ according as $m$ is
positive or negative. For each $V$, there is given a natural isomorphism
$\eta_{V}\colon H^{2d}(V)(d)\rightarrow Q$.

\item The pairings
\begin{equation}
H^{r}(V)\times H^{2d-r}(V)(d)\rightarrow H^{2d}(V)(d)\simeq Q \label{e19}%
\end{equation}
induced by the product structure on $H^{\ast}(V)$ are non-degenerate.
\end{enumerate}

\item[K\"{u}nneth formula] Let $p,q\colon V\times W\rightrightarrows V,W$ be
the projection maps. Then the map
\[
x\otimes y\mapsto p^{\ast}x\cdot q^{\ast}y:H^{\ast}(V)\otimes H^{\ast
}(W)\rightarrow H^{\ast}(V\times W)
\]
is an isomorphism of graded $Q$-algebras$.$
\end{description}

\begin{description}
\item[Cycle map] There are given group homomorphisms
\[
cl_{V}:C_{\text{\textrm{rat}}}^{r}(V)\rightarrow H^{2r}(V)(r)
\]
satisfying the following conditions:

\begin{enumerate}
\item (functoriality) for every regular map $\phi\colon V\rightarrow W$,%
\begin{equation}
\left\{
\begin{array}
[c]{rcl}%
\phi^{\ast}\circ cl_{W} & = & cl_{V}\circ\phi^{\ast}\\
\phi_{\ast}\circ cl_{V} & = & cl_{W}\circ\phi_{\ast}%
\end{array}
\right.  \label{e20}%
\end{equation}

\item (multiplicativity) for every $Y\in C_{\mathrm{rat}}^{r}(V)$ and $Z\in
C_{\mathrm{rat}}^{s}(W),$
\[
cl_{V\times W}(Y\times Z)=cl_{V}(Y)\otimes cl_{W}(Z).
\]

\item (normalization) If $P$ is a point, so that $C_{\text{\textrm{rat}}%
}^{\ast}(P)\simeq\mathbb{Q}$ and $H^{\ast}(P)\simeq Q$, then $cl_{P}$ is the
natural inclusion map.
\end{enumerate}
\end{description}

\smallskip\noindent Let $\phi\colon V\rightarrow W$ be a regular map of
nonsingular projective varieties over $k$, and let $\phi^{\ast}$ be the map
$H^{\ast}(\phi)\colon H^{\ast}(W)\rightarrow H^{\ast}(V)$. Because the pairing
(\ref{e19}) is nondegenerate, there is a unique linear map
\[
\phi_{\ast}\colon H^{\ast}(V)\rightarrow H^{\ast-2c}(W)(-c),\quad c=\dim
V-\dim W
\]
such that the projection formula
\[
\eta_{W}(\phi_{\ast}(x)\cdot y)=\eta_{V}(x\cdot\phi^{\ast}y)
\]
holds for all $x\in H^{r}(V)$, $y\in H^{2\dim V-r}(W)(\dim V)$. This explains
the maps on the left of the equals signs in (\ref{e20}), and the maps on the
right refer to the standard operations on the groups of algebraic cycles
modulo rational equivalence (Fulton 1984, Chapter I).

\subsubsection{The Lefschetz trace formula}

Let $H^{\ast}$ be a Weil cohomology over the base field $k$, and let $V$ be
nonsingular projective variety over $k$. We use the K\"{u}nneth formula to
identify $H^{\ast}(V\times V)$ with $H^{\ast}(V)\otimes H^{\ast}(V)$.

Let $\phi\colon V\rightarrow V$ be a regular map. We shall need an expression
for the class of the graph $\Gamma_{\phi}$ of $\phi$ in
\[
H^{2d}(V\times V)(d)\simeq\bigoplus\nolimits_{r=0}^{2d}H^{r}(V)\otimes
H^{2d-r}(V)(d).
\]
Let $(e_{i}^{r})_{i}$ be a basis for $H^{r}(V)$ and let $(f_{i}^{2d-r})_{i}$
be the dual basis in $H^{2d-r}(V)(d)$; we choose $e^{0}=1$, so that $\eta
_{V}(f^{2d})=1$. Then%
\[
cl_{V\times V}(\Gamma_{\phi})=\sum\nolimits_{r,i}\phi^{\ast}(e_{i}^{r})\otimes
f_{i}^{2d-r}.
\]

\begin{theorem}
[Lefschetz trace formula]\label{r23} Let $\phi\colon V\rightarrow V$ be a
regular map such that $\Gamma_{\phi}\cdot\Delta$ is defined. Then
\[
(\Gamma_{\phi}\cdot\Delta)=\sum_{r=0}^{2d}(-1)^{r}\Tr(\phi\mid H^{r}(V))
\]

\end{theorem}

\begin{proof}
We the above notations, we have%
\begin{align*}
cl_{V\times V}(\Gamma_{\phi})  &  =\sum_{r,i}\phi^{\ast}(e_{i}^{r})\otimes
f_{i}^{2d-r}\quad\text{and}\\
cl_{V\times V}(\Delta)  &  =\sum_{r,i}e_{i}^{r}\otimes f_{i}^{2d-r}=\sum
_{r,i}(-1)^{r(2d-r)}f_{i}^{2d-r}\otimes e_{i}^{r}=\sum_{r,i}(-1)^{r}%
f_{i}^{2d-r}\otimes e_{i}^{r}.
\end{align*}
Thus%
\begin{align*}
cl_{V\times V}(\Gamma_{\phi}\cdot\Delta)  &  =\sum_{r,i}(-1)^{r}\phi^{\ast
}(e_{i}^{r})f_{i}^{2d-r}\otimes f^{2d}\\
&  =\sum_{r=0}^{2d}(-1)^{r}\Tr(\phi^{\ast})(f^{2d}\otimes f^{2d})
\end{align*}
because $\phi^{\ast}(e_{i}^{r})f_{j}^{2d-r}$ is the coefficient of $e_{j}^{r}$
when $\phi^{\ast}(e_{i}^{r})$ is expressed in terms of the basis $(e_{j}^{r}%
)$. On applying $\eta_{V\times V}$ to both sides, we obtain the required formula.
\end{proof}

This is Lefschetz's original 1924 proof (see Steenrod 1957, p.27).

\subsubsection{There is no Weil cohomology theory with coefficients in
$\mathbb{Q}{}_{p}$ or $\mathbb{R}{}$.}

\begin{plain}
\label{r51}Recall (\ref{r47}) that for a simple abelian variety $A$ over a
finite field $k$, $E\overset{\textup{{\tiny def}}}{=}\End^{0}(A)$ is a
division algebra with centre $F\overset{\textup{{\tiny def}}}{=}\mathbb{Q}%
{}[\pi]$, and $2\dim(A)=[E\colon F]^{1/2}\cdot\lbrack F\colon\mathbb{Q}{}]$.
Let $Q$ be a field of characteristic zero. In order for $E$ to act on a
$Q$-vector space of dimension $2\dim(A)$, the field $Q$ must split $E$, i.e.,
$Q\otimes_{\mathbb{Q}{}}E$ must be a matrix algebra over $Q\otimes
_{\mathbb{Q}{}}F$. The endomorphism algebra of a supersingular elliptic curve
over a finite field containing $\mathbb{F}{}_{p^{2}}$ is a division quaternion
algebra over $\mathbb{Q}{}$ that is nonsplit exactly at $p$ and the real prime
(Hasse 1936). Therefore, there cannot be a Weil cohomology with coefficients
in $\mathbb{Q}{}_{p}$ or $\mathbb{R}$ (hence not in $\mathbb{Q}{}$
either).\footnote{It was Serre who explained this to Grothendieck, sometime in
the 1950s.} Tate's formula (\ref{e26}) doesn't forbid there being a Weil
cohomology with coeffients in $\mathbb{Q}{}_{\ell}$, $\ell\neq p$, or in the
field of fractions of the Witt vectors over $k$. That Weil cohomology theories
exist over these fields (see below) is an example of Yhprum's law in
mathematics: everything that can work, will work.
\end{plain}

\subsection{Proof of the Weil conjectures W1--W3}

Let $V_{0}$ be a nonsingular projective variety of dimension $d$ over
$k_{0}=\mathbb{F}{}_{q}$, and let $V$ be the variety obtained by extension of
scalars to the algebraic closure $k$ of $k_{0}$. Let $\pi\colon V\rightarrow
V$ be the Frobenius map (relative to $V_{0}/k_{0}$). We assume that there
exists a Weil cohomology theory over $k$ with coefficients in $Q$.

\begin{proposition}
\label{r24} The zeta function%
\begin{equation}
Z(V,T)=\frac{P_{1}(T)\cdots P_{2d-1}(T)}{(1-T)P_{2}(T)\cdots(1-q^{d}T)}
\label{e21}%
\end{equation}
with $P_{r}(V,T)=\det(1-\pi T\mid H^{r}(V,Q)).$
\end{proposition}

\begin{proof}
Recall that%
\[
\log Z(V,T)\overset{\textup{{\tiny def}}}{=}\sum\nolimits_{n>0}N_{n}%
\frac{T^{n}}{n}\text{.}%
\]
The fixed points of $\pi^{n}$ have multiplicity $1$, and so%
\[
N_{n}=(\Gamma_{\pi^{n}}\cdot\Delta)\overset{\text{(\ref{r23})}}{=}\sum
_{r=0}^{2d}(-1)^{r}\Tr(\pi^{n}\mid H^{r}(V,Q)).
\]
The following elementary statement completes the proof: Let $\alpha$ be an
endomorphism of a finite-dimensional vector space $V$; then
\begin{equation}
\log\left(  \det(1-\alpha T\mid V)\right)  =-\sum_{n\geq1}\Tr(\alpha^{n}\mid
V)\frac{T^{n}}{n}. \label{e40}%
\end{equation}

\end{proof}

\begin{proposition}
\label{r29}The zeta function $Z(V,T)\in\mathbb{Q}{}[T]$.
\end{proposition}

\begin{proof}
We know that $Z(V,T)\in Q[T]$. To proceed, we shall need the following
elementary criterion:\bquote Let $f(T)=\sum_{i\geq0}a_{i}T^{i}\in Q[[T]]$;
then $f(T)\in Q[T]$ if and only if there exist integers $m$ and $n_{0}$ such
that the Hankel determinant%
\[
\left\vert \!\!%
\begin{array}
[c]{llll}%
a_{n} & a_{n+1} & \cdots & a_{n+m-1}\\
a_{n+1} & a_{n+2} & \cdots & a_{n+m}\\
a_{n+2} & a_{n+3} & \cdots & a_{n+m+1}\\
\quad\vdots & \quad\vdots &  & \quad\vdots\\
a_{n+m-1} & a_{n+m} & \cdots & a_{n+2m-2}%
\end{array}
\!\!\right\vert
\]
is $0$ for all $n\geq n_{0}$ (this is a restatement of (\ref{r50}); it is an
exercise in Chap. 4 of Bourbaki's \textit{Alg\`{e}bre}).\equote The power
series $Z(V,T)$ satisfies this criterion in $Q[[T]]$, and hence also in
$\mathbb{Q}{}[[T]]$.
\end{proof}

As in the case of curves, the zeta function can be written%
\[
Z(V,T)=\prod_{v\in|V|}\frac{1}{1-T^{\deg(v)}}%
\]
where $v$ runs over the set $|V|$ of closed points of $V$ (as a scheme).
Hence
\[
Z(V,T)=1+a_{1}T+a_{2}T^{2}+\cdots
\]
with the $a_{i}\in\mathbb{Z}{}$. When we write
\[
Z(V,T)=\frac{P(T)}{R(T)},\quad P,R\in\mathbb{Q}{}[T],\quad\gcd(P,R)=1,
\]
we can normalize $P$ and $R$ so that%
\begin{align*}
P(T)  &  =1+b_{1}T+\cdots\quad\in\mathbb{Q}{}[T]\\
R(T)  &  =1+c_{1}T+\cdots\quad\in\mathbb{Q}{}[T].
\end{align*}

\begin{proposition}
\label{r30}Let $P,R$ be as above. Then $P$ and $R$ are uniquely determined by
$V$, and they have coefficients in $\mathbb{Z}{}$.
\end{proposition}

\begin{proof}
The uniqueness follows from unique factorization in $\mathbb{Q}{}[T]$. If some
coefficient of $R$ is not an integer, then $\beta^{-1}$ is not an algebraic
integer for some root $\beta$ of $R$. Hence $\ord_{l}(\beta)>0$ for some prime
$l$, and $Z(V,\beta)=1+a_{1}\beta+a_{2}\beta^{2}+\cdots$ converges
$l$-adically. This contradicts the fact that $\beta$ is a pole of $Z(V,T)$. We
have shown that $R(T)$ has coefficients in $\mathbb{Z}{}$. As $Z(V,T)^{-1}%
\in\mathbb{Z}{}[[T]]$, the same argument applies to $P(T)$.
\end{proof}

\begin{proposition}
\label{r28} We have
\[
Z(V,1/q^{d}T)=\pm q^{d\chi/2}\cdot T^{\chi}\cdot Z(V,T)
\]
where $\chi=(\Delta\cdot\Delta)$.
\end{proposition}

\begin{proof}
The Frobenius map $\pi$ has degree $q^{d}$. Therefore, for any closed point
$P$ of $V$, $\pi^{\ast}P=q^{d}P$, and%
\begin{equation}
\pi^{\ast}cl_{V}(P)=cl_{V}(\pi^{\ast}P)=cl_{V}(q^{d}P)=q^{d}cl_{V}(P).
\label{e22}%
\end{equation}
Thus $\pi$ acts as multiplication by $q^{d}$ on $H^{2d}(V)(d)$. From this, and
Poincar\'{e} duality, it follows that if $\alpha_{1},\ldots,\alpha_{s}$ are
the eigenvalues of $\pi$ acting on $H^{r}(V)$, then $q^{d}/\alpha_{1}%
,\ldots,q^{d}/\alpha_{s}$ are the eigenvalues of $\pi$ acting on $H^{2d-r}%
(V)$. An easy calculation now shows that the required formula holds with
$\chi$ replaced by
\[
\sum(-1)^{r}\dim H^{r}(V),
\]
but the Lefschetz trace formula (\ref{r23}) with $\phi=\id$ shows that this
sum equals $(\Delta\cdot\Delta)$.
\end{proof}

\begin{plain}
\label{r26}In the expression (\ref{e21}) for $Z(V,T),$ we have not shown that
the polynomials $P_{r}$ have coefficients in $\mathbb{Q}{}$ nor that they are
independent of the Weil cohomology theory. However, (\ref{r30}) shows that
this is true for the numerator and denominator after we have removed any
common factors. If the $P_{r}(T)$ are relatively prime in pairs, then there
can be no cancellation, and $P_{r}(T)=1+\sum_{i}a_{r,i}T^{i}$ has coefficients
in $\mathbb{Z}{}$ and is independent of the Weil cohomology (because this is
true of the irreducible factors of the numerator and denominator). If
$\left\vert \iota(\alpha)\right\vert =q^{r/2}$ for every eigenvalue $\alpha$
of $\pi$ acting on $H^{r}(V,Q)$ and embedding $\iota$ of $Q(\alpha)$ into
$\mathbb{C}{}$, then the $P_{r}(T)$ are relatively prime in pairs, and so the
Weil conjectures W1--W4 are true for $V$.
\end{plain}

\subsection{Application to rings of correspondences}

We show that the (mere) existence of a Weil cohomology theory implies that the
$\mathbb{Q}{}$-algebra of correspondences for numerical equivalence on an
algebraic variety is a semisimple $\mathbb{Q}{}$-algebra of finite dimension.

\begin{plain}
\label{r75}Let $V$ be a connected nonsingular projective variety of dimension
$d$. An \textit{algebraic }$r$\textit{-cycle} on $V$ is a formal sum $Z=\sum
n_{i}Z_{i}$ with $n_{i}\in\mathbb{Z}{}$ and $Z_{i}$ an irreducible closed
subvariety of codimension $r$; such cycles form a group $C^{r}(V)$.
\end{plain}

\begin{plain}
\label{r76}Let $\sim$ be an adequate equivalence relation on the family of
groups $C^{r}(V)$. Then $C_{\sim}^{\ast}(V)\overset{\textup{{\tiny def}}}%
{=}\bigoplus_{r\geq0}C^{r}/\sim$ becomes a graded ring under intersection
product; moreover, push-forwards and pull-backs of algebraic cycles with
respect to regular maps are well-defined. Let $A_{\sim}^{r}(V)=C_{\sim}%
^{r}(V)\otimes\mathbb{Q}$. There is a bilinear map%
\[
A_{\sim}^{\dim(V_{1})+r}(V_{1}\times V_{2})\times A_{\sim}^{\dim(V_{2}%
)+s}(V_{2}\times V_{3})\rightarrow A_{\sim}^{\dim(V_{1})+r+s}(V_{1}\times
V_{3})
\]
sending $(f,g)$ to%
\[
g\circ f\overset{\textup{{\tiny def}}}{=}(p_{13})_{\ast}(p_{12}^{\ast}f\cdot
p_{23}^{\ast}g)
\]
where $p_{ij}$ is the projection $V_{1}\times V_{2}\times V_{3}\rightarrow
V_{i}\times V_{j}$. This is associative in an obvious sense. In particular,
$A_{\sim}^{\dim(V)}(V\times V)$ is a $\mathbb{Q}{}$-algebra.
\end{plain}

\begin{plain}
\label{r77}Two algebraic $r$-cycles $f,g$ are \textit{numerically equivalent}
if $(f\cdot h)=(g\cdot h)$ for all $(d-r)$-cycles $h$ for which the
intersection products are defined. This is an adequate equivalence relation,
and so we get a $\mathbb{Q}{}$-algebra $A_{\mathrm{num}}^{d}(V\times V)$.
\end{plain}

Let $H$ be a Weil cohomology theory with coefficient field $Q$, and let
$A_{H}^{r}(V)$ (resp. $A_{H}^{r}(V,Q)$) denote the $\mathbb{Q}{}$-subspace
(resp. $Q$-subspace) of $H^{2r}(V)(r)$ spanned by the algebraic classes.

\begin{plain}
\label{r71}The $\mathbb{Q}{}$-vector space $A_{\mathrm{num}}^{r}(V)$ is
finite-dimensional. To see this, let $f_{1},\ldots,f_{s}$ be elements of
$A_{\mathrm{H}}^{d-r}(V)$ spanning the $Q$-subspace $A_{H}^{d-r}(V,Q)$ of
$H^{2d-2r}(V)(d-r)$; then the kernel of the map%
\[
x\mapsto(x\cdot f_{1},\ldots,x\cdot f_{s})\colon A_{H}^{r}(V)\rightarrow
\mathbb{Q}{}^{s}%
\]
consists of the elements of $A_{H}^{r}(V)$ numerically equivalent to zero, and
so its image is $A_{\mathrm{num}}^{r}(V)$.
\end{plain}

\begin{plain}
\label{r72}Let $A_{\mathrm{num}}^{r}(V,Q)$ denote the quotient of $A_{H}%
^{r}(V,Q)$ by the left kernel of the pairing
\[
A_{H}^{r}(V,Q)\times A_{H}^{d-r}(V,Q)\rightarrow A_{H}^{d}(V,Q)\simeq
Q\text{.}%
\]
Then $A_{H}^{r}(V)\rightarrow A_{\mathrm{num}}^{r}(V,Q)$ factors through
$A_{\mathrm{num}}^{r}(V)$, and I claim that the map
\[
a\otimes f\mapsto af\colon Q\otimes A_{\mathrm{num}}^{r}(V)\rightarrow
A_{\mathrm{num}}^{r}(V,Q)
\]
is an isomorphism. As $A_{\mathrm{num}}^{r}(V,Q)$ is spanned by the image of
$A_{\mathrm{num}}^{r}(V)$, the map is obviously surjective. Let $e_{1}%
,\ldots,e_{m}$ be a $\mathbb{Q}{}$-basis for $A_{\mathrm{num}}^{r}(V)$, and
let $f_{1},\ldots,f_{m}$ be the dual basis in $A_{\mathrm{num}}^{d-r}(V)$. If
$\sum_{i=1}^{m}a_{i}\otimes e_{i}$ ($a_{i}\in Q$) becomes zero in
$A_{\mathrm{num}}^{r}(V,Q)$, then $a_{j}=(\sum a_{i}e_{i})\cdot f_{j}=0$ for
all $j$. Thus the map is injective.
\end{plain}

\begin{theorem}
\label{r74}The $\mathbb{Q}{}$-algebra $A_{\mathrm{num}}^{\dim(V)}(V\times V)$
is finite-dimensional and semisimple.
\end{theorem}

\begin{proof}
Let $d=\dim(V)$ and $B=A_{\mathrm{num}}^{d}(V\times V)$. Then $B$ is a
finite-dimensional $\mathbb{Q}{}$-algebra (\ref{r71}), and the pairing%
\[
f,g\mapsto(f\cdot g)\colon B\times B\rightarrow\mathbb{Q}{}%
\]
is nondegenerate. Let $f$ be an element of the radical $\mathrm{rad}{}(B)$ of
$B$. We have to show that $(f\cdot g)=0$ for all $g\in B$.

Let $A=A_{H}^{d}(V\times V,Q)$. Then $A$ is a finite-dimensional $Q$-algebra,
and there is a surjective homomorphism
\[
A\overset{\textup{{\tiny def}}}{=}A_{H}^{d}(V\times V,Q)\overset
{S}{\longrightarrow}A_{\mathrm{num}}^{d}(V\times V,Q)\overset{\text{\ref{r72}%
}}{\simeq}Q\otimes B.
\]
As the ring $A/\mathrm{rad}(A)$ is semisimple, so also is its quotient
$(Q\otimes B)/S(\mathrm{rad}(A))$. Therefore $S(\mathrm{rad}(A))\supset
\mathrm{rad}(Q\otimes B)$, and so there exists an $f^{\prime}\in\mathrm{rad}%
{}(A)$ mapping to $1\otimes f$. For all $g\in A$,%
\begin{equation}
(f^{\prime}\cdot g^{t})=\sum\nolimits_{i}(-1)^{i}\Tr(f^{\prime}\circ g\mid
H^{i}(V)) \label{eq6}%
\end{equation}
--- this can be proved exactly as (\ref{r23}). But $f^{\prime}\circ g^{t}%
\in\mathrm{rad}{}(A)$; therefore it is nilpotent, and so its trace on
$H^{i}(V)$ is zero. Hence $(f^{\prime}\cdot g^{t})=0$, and so $(1\otimes
f\cdot S(g^{t}))=0$ for all $g\in A$. It follows that $f=0$.
\end{proof}

Theorem \ref{r74} was extracted from Jannsen 1992.

\subsection{Etale cohomology}

\hfill\begin{minipage}{3.0in}
\textit{From all the work of Grothendieck, it is without doubt \'{e}tale cohomology
which has exercised the most profound influence on the development of
arithmetic geometry in the last fifty years.}\\
Illusie 2014.\footnotemark
\end{minipage}\footnotetext{De toute l'oeuvre de Grothendieck, c'est sans
doute la cohomologie \'{e}tale qui aura exerc\'{e} l'influence la plus
profonde sur l'\'{e}volution de la g\'{e}om\'{e}trie arithm\'{e}tique dans les
cinquante derni\`{e}res ann\'{e}es.}\bigskip

With his definition of fibre spaces on algebraic varieties, Weil began the
process of introducing into abstract algebraic geometry the powerful
topological methods used in the study of complex algebraic varieties. He
introduced fibre spaces in a 1949 conference talk, and then, in more detail,
in a course at the University of Chicago (Weil 1952). For the first time, he
made use of the Zariski topology in his definition of an abstract variety, and
he equipped his varieties with this topology. He required a fibre space to be
locally trival for the Zariski topology on the base variety. Weil's theory
works much as expected, but some fibre spaces that one expects (from topology)
to be locally trivial are not, because the Zariski topology has too few open sets.

In a seminar in April 1958, Serre enlarged the scope of Weil's theory by
admitting also fibre spaces that are only \textquotedblleft locally
isotrivial\textquotedblright\ in the following sense: there exists a covering
$V=\bigcup_{i}U_{i}$ of the base variety $V$ by open subvarieties $U_{i}$ and
finite \'{e}tale maps $U_{i}^{\prime}\rightarrow U_{i}$ such that the fibre
space becomes trivial when pulled back to each $U_{i}^{\prime}$. For an
algebraic group $G$ over $k$, Serre defined $H^{1}(V,G)$ to be the set of
isomorphism classes of principal fibre spaces on $V$ under $G$, which he
considered to be the \textquotedblleft good $H^{1}$\textquotedblright. At the
end of the seminar, Grothendieck said to Serre that this will give the Weil
cohomology in all dimensions!\footnote{D\`{e}s la fin de l'expos\'{e} oral,
Grothendieck m'a dit: cela va donner la cohomologie de Weil en toute
dimension!} (Serre 2001, p.125, p.255). By the time Serre wrote up his seminar
in September 1958, he was able to include a reference to Grothendieck's
announcement (1958b) of a \textquotedblleft Weil cohomology\textquotedblright.\ 

Grothendieck's claim raised two questions:

\begin{enumerate}
[label=(\Alph*)]

\item when $G$ is commutative, is it possible to define higher cohomology groups?

\item assuming (A) are they the \textquotedblleft true\textquotedblright%
\ cohomology groups when $G$ is finite?
\end{enumerate}

\noindent To answer (A), Grothendieck observed that to define a sheaf theory
and a sheaf-cohomology, much less is needed than a topological space. In
particular the \textquotedblleft open subsets\textquotedblright\ need not be subsets.

Specifically, let $\mathsf{C}$ be an essentially small category admitting
finite fibred products and a final object $V$, and suppose that for each
object $U$ of $\mathsf{C}$ there is given a family of \textquotedblleft
coverings\textquotedblright\ $(U_{i}\rightarrow U)_{i}$. The system of
coverings is said to be a \textit{Grothendieck topology} on $\mathsf{\newline
C}$ if it satisfies the following conditions:

\begin{enumerate}
\item (base change) if $(U_{i}\rightarrow U)$ is a covering and $U^{\prime
}\rightarrow U$ is a morphism in $\mathsf{C}$, then $(U_{i}\times_{U}%
U^{\prime}\rightarrow U^{\prime})$ is a covering;

\item (local nature) if $(U_{i}\rightarrow U)_{i}$ is a covering, and, for
each $i$, $(U_{i,j}\rightarrow U_{i})_{j}$ is a covering, then the family of
composites $(U_{i,j}\rightarrow U)_{i,j}$ is a covering;

\item a family consisting of a single isomorphism $\varphi\colon U^{\prime
}\rightarrow U$ is a covering.
\end{enumerate}

\noindent For example, let $V$ be a topological space, and consider the
category $\mathsf{C}$ whose objects are the open subsets of $V$ with the
inclusions as morphisms; then the coverings of open subsets in the usual sense
define a Grothendieck topology on $\mathsf{C}$.

Consider a category $\mathsf{C}$ equipped with a Grothendieck topology. A
\textit{presheaf} is simply a contravariant functor from $\mathsf{C}$ to the
category of abelian groups. A presheaf $P$ is a \textit{sheaf} if, for every
covering $(U_{i}\rightarrow U)_{i}$, the sequence
\[
P(U)\rightarrow\prod_{i}P(U_{i})\rightrightarrows\prod_{i,j}P(U_{i}\times
_{U}U_{j})
\]
is exact. With these definitions, the sheaf theory in Grothendieck 1957
carries over almost word-for-word. The category of sheaves is abelian,
satisfies Grothendieck's conditions (AB5) and (AB3*), and admits a family of
generators. Therefore, it has enough injectives, and the cohomology groups can
be defined to be the right derived functors of $F\rightsquigarrow F(V)$.

When Grothendieck defined the \'{e}tale topology on a variety (or scheme) $V$,
he took as coverings those in Serre's definition of \textquotedblleft locally
isotrivial\textquotedblright, but Mike Artin realized that it was better to
allow as coverings all surjective families of \'{e}tale morphisms. With his
definition the local rings satisfy Hensel's lemma.

A test for (B) is: \bquote(C) let $V$ be a nonsingular algebraic variety over
$\mathbb{C}{}$, and let $\Lambda$ be a finite abelian group; do the \'{e}tale
cohomology groups $H^{r}(V_{\mathrm{et}},\Lambda)$ coincide with the singular
cohomology groups $H^{r}(V^{\text{an}},\Lambda)$?\equote For $r=0$, this just
says that an algebraic variety over $\mathbb{C}{}$ is connected for the
complex topology if it is connected for the Zariski topology. For $r=1$, it is
the Riemann existence theorem,\footnote{Riemann used the Dirichlet principle
(unproven at the time)\ to show that on every compact Riemann surface $S$
there are enough meromorphic functions to realize $S$ as a projective
algebraic curve; this proves the Riemann existence theorem for nonsingular
projective curves.} which says that every finite covering of $V^{\text{an}}$
is algebraic. In particular, (C) is true for curves. It was probably this that
made Grothendieck optimistic that (C) is true in all degrees.\bquote
Grothendieck thought always in relative terms: one space over another. Once
the cohomology of curves (over an algebraically closed field) was understood,
we could expect similar results for the direct images for a \textit{relative}
curve \ldots, and, \textquotedblleft by unscrewing
[d\'{e}vissage]\textquotedblright\ldots\ for the higher $H^{r}$%
.\footnote{Grothendieck pensait toujours en termes relatifs: un espace
au-dessus d'un autre. Une fois la cohomologie des courbes (sur un corps
alg\'{e}briquement clos) tir\'{e}e au clair, on pouvait esp\'{e}rer des
r\'{e}sultats similaires pour les images directes pour une courbe
\textit{relative} (les th\'{e}or\`{e}mes de sp\'{e}cialisation du $\pi_{1}$
devaient le lui sugg\'{e}rer), et, \textquotedblleft par
d\'{e}vissage\textquotedblright\ (fibrations en courbes, suites spectrales de
Leray), atteindre les $H^{i}$ sup\'{e}rieurs.} (Illusie 2014, p177.)\equote
Initially (in 1958), Serre was less sure: \bquote Of course, the Zariski
topology gives a $\pi_{1}$ and $H^{1}$ that are too small, and I had fixed
that defect. But was that enough? My reflexes as a topologist told me that we
must also deal with the higher homotopy groups: $\pi_{2}$, $\pi_{3}$,
etc.\footnote{Bien s\^{u}r, la topologie de Zariski donne un $\pi_{1}$ et un
$H^{1}$ trop petits, et j'avais rem\'{e}di\'{e} \`{a} ce d\'{e}faut. Mais
\'{e}tait-ce suffisant? Mes r\'{e}flexes de topologue me disaient qu'il
fallait aussi s'occuper des groupes d'homotopie sup\'{e}rieurs: $\pi_{2}$,
$\pi_{3}$, etc.} (Serre 2001, p.255.)\equote Artin proved (C) by showing that,
in the relative dimension one case, the Zariski topology is sufficiently fine
to give coverings by $K(\pi,1)$'s (SGA 4, XI).\footnote{This is rather Serre's
way of viewing Artin's proof. As Serre wrote (email July 2015): \bquote[This]
was the way I saw it, and I liked it for two reasons : a) it is a kind of
explanation why higher homotopy groups don't matter: they don't occur in these
nice Artin neighbourhoods; b) the fundamental group of such a neighbourhood
has roughly the same structure (iterated extension of free groups) as the
braid groups which were so dear to Emil Artin; in particular, it is what I
called a "good group" : its cohomology is the same when it is viewed as a
discrete group or as a profinite group.\equote}

Although Grothendieck had the idea for \'{e}tale cohomology in 1958,

\begin{quote}
a few years passed before this idea really took shape: Grothendieck did not
see how to start. He also had other occupations.\footnote{quelques ann\'{e}es
s'\'{e}coul\`{e}rent avant que cette id\'{e}e ne prenne r\'{e}ellement forme:
Grothendieck ne voyait pas comment d\'{e}marrer. Il avait aussi d'autres
occupations.} (Illusie 2014, p.175.)
\end{quote}

\noindent When Grothendieck came to Harvard in 1961, Mike Artin asked him:

\begin{quote}
if it was all right if I thought about it, and so that was the beginning\ldots
\ [Grothendieck] didn't work on it until I proved the first theorem. \ldots\ I
thought about it that fall \ldots\ And then I gave a seminar\ldots\ (Segel
2009, p.358.)
\end{quote}

\noindent Serre writes (email July 2015):

\begin{quote}
Grothendieck, after the seminar lecture where I defined \textquotedblleft the
good $H^{1}$\textquotedblright, had the idea that the higher cohomology groups
would also be the good ones. But, as far as I know, he could not prove their
expected properties\ldots\ It was Mike Artin, in his Harvard seminar notes of
1962, who really started the game, by going beyond $H^{1}$. For instance, he
proved that a smooth space of dimension $2$ minus a point has (locally) the
same cohomology as a 3-sphere (Artin 1962, p.110). After that, he and
Grothendieck took up, with SGA 4: la locomotive de Bures \'{e}tait lanc\'{e}e.
\end{quote}

\noindent In 1963-64, Artin and Grothendieck organized their famous SGA 4 seminar.

The \'{e}tale topology gives good cohomology groups only for torsion groups.
To obtain a Weil cohomology theory, it is necessary to define%
\[
H^{r}(V_{\mathrm{et}},\mathbb{Z}{}_{\ell})=\varprojlim\nolimits_{n}%
H^{r}(V_{\mathrm{et}},\mathbb{Z}{}/\ell^{n}\mathbb{Z}{}),
\]
and then tensor with $\mathbb{Q}{}_{\ell}$ to get $H^{r}(V_{\mathrm{et}%
},\mathbb{Q}{}_{\ell})$. This does give a Weil cohomology theory, and so the
Weil conjectures (W1--W3) hold with%
\[
P_{r}(V,T)=\det(1-\pi T\mid H^{r}(V_{\mathrm{et}},\mathbb{Q}{}_{\ell})).
\]
More generally, Grothendieck (1964) proved that, for every algebraic variety
$V_{0}$ over a finite field $k_{0}$,%
\begin{equation}
Z(V_{0},T)=\prod\nolimits_{r}\det(1-\pi T\mid H_{c}^{r}(V_{\mathrm{et}%
},\mathbb{Q}{}_{\ell}))^{(-1)^{r+1}} \label{e37}%
\end{equation}
where $H_{c}$ denotes cohomology with compact support. In the situation of
(W5),
\[
H^{r}(V_{\mathrm{et}},\mathbb{Q}{}_{\ell})\simeq H^{r}(\tilde{V}(\mathbb{C}%
{}),\mathbb{Q}{})\otimes\mathbb{Q}{}_{\ell}%
\]
(proper and smooth base change theorem), and so the $\ell$-adic Betti numbers
of $V$ are independent of $\ell$. If the Riemann hypothesis holds, then they
equal Weil's Betti numbers.

\smallskip In the years since it was defined, \'{e}tale cohomology has become
such a fundamental tool that today's arithmetic geometers have trouble
imagining an age in which it didn't exist.

\subsection{de Rham cohomology (characteristic zero)}

Let $V$ be a nonsingular algebraic variety over a field $k$. Define
$H_{\mathrm{dR}}^{\ast}(V)$ to be the (hyper)cohomology of the complex%
\[
\Omega_{V/k}^{\bullet}=\quad\mathcal{O}{}_{X}\overset{d}{\longrightarrow
}\Omega_{V/k}^{1}\overset{d}{\longrightarrow}\cdots\overset{d}{\longrightarrow
}\Omega_{V/k}^{r}\overset{d}{\longrightarrow}\cdots
\]
of sheaves for the Zariski topology on $V$. When $k=\mathbb{C}{}$, we can also
define $H_{\mathrm{dR}}^{\ast}(V^{\text{an}})$ by replacing $\Omega
_{V/k}^{\bullet}$ with the complex of sheaves of holomorphic differentials on
$V^{\text{an}}$ for the complex topology. Then%
\[
H^{r}(V^{\text{an}},\mathbb{Q}{})\otimes_{\mathbb{Q}{}}\mathbb{C}{}\simeq
H_{\mathrm{dR}}^{r}(V^{\text{an}})
\]
for all $r$.

When $k=\mathbb{C}{}$, there is a canonical homomorphism%
\[
H_{\mathrm{dR}}^{\ast}(V)\rightarrow H_{\mathrm{dR}}^{\ast}(V^{\text{an}%
})\text{.}%
\]
In a letter to Atiyah in 1963, Grothendieck proved that this is an isomorphism
(Grothendieck 1966). Thus, for a nonsingular algebraic variety over a field
$k$ of characteristic zero, there are algebraically defined cohomology groups
$H_{\mathrm{dR}}^{\ast}(V)$ such that, for every embedding $\rho\colon
k\rightarrow\mathbb{C}{}$,%
\[
H_{\mathrm{dR}}^{\ast}(V)\otimes_{k,\rho}\mathbb{C}{}\simeq H_{\mathrm{dR}%
}^{\ast}(\rho V)\simeq H_{\mathrm{dR}}^{\ast}(\left(  \rho V\right)
^{\text{an}}).
\]
This gives a Weil cohomology theory with coefficients in $k$.

\subsection{$p$-adic cohomology}

The \'{e}tale topology gives Weil cohomologies with coefficients in
$\mathbb{Q}{}_{\ell}$ for all primes $\ell$ different from the characteristic
of $k$. Dwork's early result (see p.\pageref{Dwork}) suggested that there
should also be $p$-adic Weil cohomology theories, i.e., a cohomology theories
with coefficients in a field containing $\mathbb{Q}{}_{p}$.

Let $V$ be an algebraic variety over a field $k$ of characteristic $p\neq0$.
When Serre defined the cohomology groups of coherent sheaves on algebraic
varieties, he asked whether the formula%
\[
\beta_{r}(V)\overset{?}{=}\sum_{i+j=r}\dim_{k}H^{j}(V,\Omega_{V/k}^{i})
\]
gives the \textquotedblleft true\textquotedblright\ Betti numbers, namely,
those intervening in the Weil conjectures (Serre 1954, p.520). An example of
Igusa (1955) showed that this formula gives (at best) an upper bound for
Weil's Betti numbers. Of course, the groups $\bigoplus_{i+j=r}H^{j}%
(V,\Omega_{V/k}^{i})$ wouldn't give a Weil cohomology theory because the
coefficient field has characteristic $p$. Serre (1958) next considered the
Zariski cohomology groups $H^{r}(V,W\mathcal{O}{}_{V})$ where $W\mathcal{O}%
{}_{V}$ is the sheaf of Witt vectors over $\mathcal{O}{}_{V}$ (a ring of
characteristic zero), but found that they did not have good properties (they
would have given only the $H^{0,r}$ part of the cohomology).

In 1966, Grothendieck discovered how to obtain the de Rham cohomology groups
in characteristic zero without using differentials, and suggested that his
method could be modified to give a good $p$-adic cohomology theory in
characteristic $p$.

Let $V$ be a nonsingular variety over a field $k$ of characteristic zero.
Define $\inf(V/k)$ to be the category whose objects are open subsets $U$ of
$V$ together with thickening of $U$, i.e., an immersion $U\hookrightarrow~T$
defined by a nilpotent ideal in $\mathcal{O}{}_{T}$. Define a covering family
of an object $(U,U\hookrightarrow T)$ of $\inf(V/k)$ to be a family
$(U_{i},U_{i}\hookrightarrow T_{i})_{i}$ with $(T_{i})_{i}$ a Zariski open
covering of $T$ and $U_{i}=U\times_{T}T_{i}$. These coverings define the
\textquotedblleft infinitesimal\textquotedblright\ Grothendieck topology on
$\inf(V/k)$. There is a structure sheaf $\mathcal{O}{}_{V_{\inf}}$ on
$V_{\inf}$, and Grothendieck proves that%
\[
H^{\ast}(V_{\inf},\mathcal{O}{}_{V_{\inf}})\simeq H_{\mathrm{dR}}^{\ast}(V)
\]
(Grothendieck 1968, 4.1).

This doesn't work in characteristic $p$, but Grothendieck suggested that by
adding divided powers to the thickenings, one should obtain a good cohomology
in characteristic $p$. There were technical problems at the prime $2$, but
Berthelot resolved these in this thesis to give a good definition of the
\textquotedblleft crystalline\textquotedblright\ site, and he developed a
comprehensive treatment of crystalline cohomology (Berthelot 1974). This is a
cohomology theory with coefficients in the ring of Witt vectors over the base
field $k$. On tensoring it with the field of fractions, we obtain a Weil cohomology.

About 1975, Bloch extended Serre's sheaf $W\mathcal{O}{}_{V}$ to a
\textquotedblleft de Rham-Witt\ complex\textquotedblright\
\[
W\Omega_{V/k}^{\bullet}:\quad W\mathcal{O}{}_{V}\overset{d}{\longrightarrow
}W\Omega_{V/k}^{1}\overset{d}{\longrightarrow}\cdots
\]
and showed that (except for some small $p$) the Zariski cohomology of this
complex is canonically isomorphic to crystalline cohomology (Bloch 1977).
Bloch used $K$-theory to define $W\Omega_{V/k}^{\bullet}$ (he was interested
in relating $K$-theory to crystalline cohomology among other things). Deligne
suggested a simpler, more direct, definition of the de Rham-Witt complex and
this approach was developed in detail by Illusie and Raynaud (Illusie 1983).

Although $p$-adic cohomology is more difficult to define than $\ell$-adic
cohomology, it is often easier to compute with it. It is essential for
understanding $p$-phenomena, for example, $p$-torsion, in characteristic $p$.

\begin{note}
\label{r53}The above account of the origins of $p$-adic cohomology is too
brief --- there were other approaches and other contributors.
\end{note}

\section{The standard conjectures}

\hfill\begin{minipage}{3.2in}
\textit{Alongside the problem of resolution of singularities, the proof of the standard conjectures seems to me to be the most urgent task in algebraic geometry.}\\
Grothendieck 1969.
\end{minipage}\bigskip\bigskip

We have seen how to deduce the first three of the Weil conjectures from the
existence of a Weil cohomology. What more is needed to deduce the Riemann
hypothesis? About 1964, Bombieri and Grothendieck independently found the
answer: we need a K\"{u}nneth formula and a Hodge index theorem for algebraic
classes. Before explaining this, we look at the analogous question over
$\mathbb{C}{}$.

\subsection{A k\"{a}hlerian analogue}

In his 1954 ICM talk, Weil sketched a transcendental proof of the inequality
$\sigma(\xi\circ\xi^{\prime})>0$ for correspondences on a complex curve, and wrote:

\begin{quote}
\ldots\ this is precisely how I first persuaded myself of the truth of the
abstract theorem even before I had perceived the connection between the trace
$\sigma$ and Castelnuovo's equivalence defect.
\end{quote}

\noindent In a letter to Weil in 1959, Serre wrote:

\begin{quote}
In fact, a similar process, based on Hodge theory, applies to varieties of any
dimension, and one obtains both the positivity of certain traces, and the
determination of the absolute values of certain eigenvalues in perfect analogy
with your beloved conjectures on the zeta
functions.\footnote{\textquotedblleft En fait, un proc\'{e}d\'{e} analogue,
bas\'{e} sur la th\'{e}orie de Hodge, s'applique aux vari\'{e}t\'{e}s de
dimension quelconque, et l'on obtient \`{a} la fois la positivit\'{e} de
certaines traces, et la d\'{e}termination des valeurs absolues de certaines
valeurs propres, en parfaite analogie avec tes ch\`{e}res conjectures sur les
fonctions z\^{e}ta.\textquotedblright\ }
\end{quote}

\noindent We now explain this. More concretely, we consider the following
problem: Let $V$ be a connected nonsingular projective variety of dimension
$d$ over $\mathbb{C}{}$, and let $f\colon V\rightarrow V$ be an endomorphism
of degree $q^{d}$; find conditions on $f$ ensuring that the eigenvalues of $f$
acting on $H^{r}(V,\mathbb{Q}{})$ have absolute value $q^{r/2}$ for all $r$.

For a curve $V$, no conditions are needed. The action of $f$ on the cohomology
of $V$ preserves the Hodge decomposition%
\[
H^{1}(V,\mathbb{C}{})\simeq H^{1,0}(V)\oplus H^{0,1}(V),\quad H^{i,j}%
(V)\overset{\textup{{\tiny def}}}{=}H^{j}(V,\Omega^{i}),
\]
and the projection $H^{1}(V,\mathbb{C}{})\rightarrow H^{1,0}(V)$ realizes
$H^{1}(V,\mathbb{Z}{})\subset$ $H^{1}(V,\mathbb{C}{})$ as a lattice in
$H^{1,0}(V)$, stable under the action of $f$. Define a hermitian form on
$H^{1,0}(V)$ by%
\[
\langle\omega,\omega^{\prime}\rangle=\frac{1}{2\pi i}\int_{V}\omega\wedge
\bar{\omega}^{\prime}.
\]
This is positive definite. As $f^{\ast}$ acts on $H_{\mathrm{dR}}^{2}(V)$ as
multiplication by $\deg(f)=q$,%
\[
\langle f^{\ast}\omega,f^{\ast}\omega^{\prime}\rangle\overset
{\textup{{\tiny def}}}{=}\frac{1}{2\pi i}\int_{V}f^{\ast}(\omega\wedge
\bar{\omega}^{\prime})=\frac{q}{2\pi i}\int_{V}\omega\wedge\bar{\omega
}^{\prime}\overset{\textup{{\tiny def}}}{=}q\langle\omega,\omega^{\prime
}\rangle\text{.}%
\]
Hence, $q^{-1/2}f$ is a unitary operator on $H^{1,0}(V)$, and so its
eigenvalues $a_{1},\ldots,a_{g}$ have absolute value $1$. The eigenvalues of
$f^{\ast}$ on $H^{1}(V,\mathbb{Q}{})$ are $q^{1/2}a_{1},\ldots,q^{1/2}%
a_{g},q^{1/2}\bar{a}_{1},\ldots,q^{1/2}\bar{a}_{g}$, and so they have absolute
value $q^{1/2}$.

In higher dimensions, an extra condition is certainly needed (consider a
product). Serre realized that it was necessary to introduce a polarization.

\begin{theorem}
[Serre 1960, Thm 1]\label{r54}Let $V$ be a connected nonsingular projective
variety of dimension $d$ over $\mathbb{C}{}$, and let $f$ be an endomorphism
$V$. Suppose that there exists an integer $q>0$ and a hyperplane section $E$
of $V$ such that $f^{-1}(E)$ is algebraically equivalent to $qE$. Then, for
all $r\geq0$, the eigenvalues of $f$ acting on $H^{r}(V,\mathbb{Q}{})$ have
absolute value $q^{r/2}$.
\end{theorem}

For a variety over a finite field and $f$ the Frobenius map, $f^{-1}(E)$ is
equivalent to $qE$, and so this is truly a k\"{a}hlerian analogue of the
Riemann hypothesis over finite fields. Note that the condition $f^{-1}(E)\sim
qE$ implies that $(f^{-1}E^{d})=q^{d}(E^{d})$, and hence that $f$ has degree
$q^{d}$.

The proof is an application of two famous theorems. Throughout, $V$ is as in
the statement of the theorem. Let $E$ be an ample divisor on $V$, let $u$ be
its class in $H^{2}(V,\mathbb{Q}{})$, and let $L$ be the \textquotedblleft
Lefschetz operator\textquotedblright\
\[
x\mapsto u\cup x\colon H^{\ast}(V,\mathbb{Q}{})\rightarrow H^{\ast
+2}(V,\mathbb{Q}{})(1)\text{.}%
\]

\begin{theorem}
[Hard Lefschetz]\label{r55}For $r\leq d$, the map%
\[
L^{d-r}\colon H^{r}(V,\mathbb{Q}{})\rightarrow H^{2d-r}(V,\mathbb{Q}{})(d-r)
\]
is an isomorphism.
\end{theorem}

\begin{proof}
It suffices to prove this after tensoring with $\mathbb{C}{}$. Lefschetz's
original \textquotedblleft topological proof\textquotedblright\ (1924) is
inadequate, but there are analytic proofs (e.g., Weil 1958, IV, n$^{\circ}$6,
Cor. to Thm 5).
\end{proof}

Now let $H^{r}(V)=H^{r}(V,\mathbb{C}{})$, and omit the Tate twists. Suppose
that $r\leq d$, and consider%
\[
\begin{tikzcd}
H^{r-2}(V)\arrow{r}{L}&H^{r}(V)\arrow{r}{L^{d-r}}[swap]{\simeq}%
&H^{2d-r}(V)\arrow{r}{L}&H^{2d-r+2}(V)\text{.}%
\end{tikzcd}
\]
The composite of the maps is an isomorphism (\ref{r55}), and so%
\[
H^{r}(V)=P^{r}(V)\oplus LH^{r-2}(V)
\]
with $P^{r}(V)=\Ker(H^{r}(V)\xrightarrow{L^{d-r+1}}H^{2d-r+2}(V)$. On
repeating this argument, we obtain the first of the following decompositions,
and the second is proved similarly:%
\[
\renewcommand{\arraystretch}{1.3}H^{r}(V)=\left\{
\begin{array}
[c]{ll}%
\dstyle\bigoplus\nolimits_{j\geq0}L^{j}P^{r-2j} & \text{if }r\leq d\\
\dstyle\bigoplus\nolimits_{j\geq r-d}L^{j}P^{r-2j} & \text{if }r\geq d.
\end{array}
\right.
\]
In other words, every element $x$ of $H^{r}(V)$ has a unique expression as a
sum%
\begin{equation}
x=\sum_{j\geq\max(r-d,0)}L^{j}x_{j},\quad x_{j}\in P^{r-2j}(V). \label{e16}%
\end{equation}
The cohomology classes in $P^{r}(V)$, $r\leq d$, are said to be
\textit{primitive}.

The \textit{Weil operator} $C\colon H^{\ast}(V{})\rightarrow H^{\ast}(V{})$ is
the linear map such that $Cx=i^{a-b}x$ if $x$ is of type $(a,b)$. It is an
automorphism of $H^{\ast}(V)$ as a $\mathbb{C}{}$-algebra, and $C^{2}$ acts on
$H^{r}(V)$ as multiplication by $(-1)^{r}$ (Weil 1958, n$^{\circ}$5, p.74).

Using the decomposition (\ref{e16}), we define an operator $\ast\colon
H^{r}(V)\rightarrow H^{2d-r}(V)$ by%
\[
\ast x=\sum_{j\geq\max(r-d,0)}(-1)^{\frac{r(r+1)}{2}}\cdot\tfrac{j!}%
{(d-r-j)!}\cdot C(L^{d-r-j}x_{j}).
\]
For $\omega\in H^{r}(V)$, let%
\[
\renewcommand{\arraystretch}{1.3}I(\omega)=\left\{
\begin{array}
[c]{cc}%
\int_{V}\omega & \text{if }r=2d\,\,\\
0 & \text{if }r<2d,
\end{array}
\right.
\]
and let%
\[
I(x,y)=I(x\cdot y).
\]

\begin{lemma}
\label{r68}For $x,y\in H^{\ast}(V)$,%
\[
I(x,\ast y)=I(y,\ast x)\text{ and }I(x,\ast x)>0\text{ if }x\neq0.
\]

\end{lemma}

\begin{proof}
Weil 1958, IV, n$^{\circ}$7, Thm 7.
\end{proof}

For $x=\sum L^{j}x_{j}$ and $y=\sum L^{j}y_{j}$ with $x_{j}$, $y_{j}$
primitive of degree $r-2j,$ put%
\[
A(x,y)=\sum_{j\geq\max(r-d,0)}(-1)^{\frac{r(r+1)}{2}}\cdot\tfrac{j!}%
{(d-r-j)!}\cdot(-1)^{j}\cdot I(u^{d-r+2j}\cdot x_{j}\cdot y_{j}).
\]

\begin{theorem}
\label{r56}The map $A$ is a bilinear form on $H^{r}(V)$, and%
\begin{align*}
A(y,x)  &  =(-1)^{r}A(x,y), & A(Cx,Cy)  &  =A(x,y)\\
A(x,Cy)  &  =A(y,Cx), & A(x,C\bar{x})  &  >0\text{ if }x\neq0.
\end{align*}

\end{theorem}

\begin{proof}
The first two statements are obvious, and the second two follow from
(\ref{r68}) because%
\[
A(x,Cy)=\sum_{j\geq\max(r-d,0)}I(L^{j}x_{j},\ast L^{j}y_{j}).
\]
See Weil 1958, IV, n$^{\circ}$7, p.78.
\end{proof}

Note that the intersection form $I$ on $H^{d}(V)$ is symmetric or
skew-symmetric according as $d$ is even or odd.

\begin{theorem}
[Hodge Index]\label{r67}Assume that the dimension $d$ of $V$ is even. Then the
signature of the intersection form on $H^{d}(V)$ is $\sum_{a,b}(-1)^{a}%
h^{a,b}(V).$
\end{theorem}

\begin{proof}
Exercise, using (\ref{r56}). See Weil 1958, IV, n$^{\circ}$7, Thm 8.
\end{proof}

To deduce (\ref{r1}) from the theorem, it is necessary to show that the
nonalgebraic cycles contribute only positive terms. This is what Hodge did in
his 1937 paper.

We now prove Theorem \ref{r54}. It follows from (\ref{r56}) that the
sesquilinear form%
\begin{equation}
(x,y)\mapsto A(x,C\bar{y})\colon H^{r}(V)\times H^{r}(V)\rightarrow\mathbb{C}
\label{e33}%
\end{equation}
is hermitian and positive definite. Let $g_{r}=q^{-r/2}H^{r}(f)$. Then $g_{r}$
respects the structure of $H^{\ast}(V)$ as a bigraded $\mathbb{C}{}$-algebra,
the form $I$, and the operators $a\mapsto\bar{a}$ and $a\mapsto La$.
Therefore, it respects the form (\ref{e33}), i.e., it is a unitary operator,
and so its eigenvalues have absolute value $1$. This completes the proof of
the theorem.

This proof extends to correspondences. For curves, it then becomes the
argument in Weil's ICM talk; for higher dimensions, it becomes the proof of
Theorem 2 of Serre 1960.

\subsection{Weil forms}

As there is no Weil cohomology theory in nonzero characteristic with
coefficients in a real field, it is not possible to realize the Frobenius map
as a unitary operator. Instead, we go back to Weil's original idea.

Let $H$ be a Weil cohomology theory over $k$ with coefficient field $Q$. From
the K\"{u}nneth formula and Poincar\'{e} duality, we obtain isomorphisms%
\[
H^{2d}(V\times V)(d)\simeq\bigoplus\nolimits_{r=0}^{2d}\left(  H^{r}(V)\otimes
H^{2d-r}(V)(d)\right)  \simeq\bigoplus\nolimits_{r=0}^{2d}%
\End_{Q\text{-linear}}(H^{r}(V)).
\]
Let $\pi_{r}$ be the $r$th K\"{u}nneth projector. Under the isomorphism, the
subring%
\[
H^{2d}(V\times V)_{r}\overset{\textup{{\tiny def}}}{=}\pi_{r}\circ
H^{2d}(V\times V)\circ\pi_{r}%
\]
of $H^{2d}(V\times V)(d)$ corresponds to $\End_{Q\text{-linear}}(H^{r}(V))$.

Let $A_{H}^{r}(-)$ denote the $\mathbb{Q}{}$-subspace of $H^{2r}(-)(r)$
generated by the algebraic classes. Then $A_{H}^{d}(V\times V)$ is the
$\mathbb{Q}{}$-algebra of correspondences on $V$ for homological equivalence
(see \ref{r76}). Assume that the K\"{u}nneth projectors $\pi_{r}$ are
algebraic, and let%
\[
A_{H}^{d}(V\times V)_{r}=A_{H}^{d}(V\times V)\cap H^{2d}(V\times V)_{r}%
=\pi_{r}\circ A_{H}^{d}(V\times V)\circ\pi_{r}\text{.}%
\]
Then%
\[
A_{H}^{d}(V\times V)=\bigoplus\nolimits_{r=0}^{2d}A_{H}^{d}(V\times V)_{r}.
\]

Let%
\[
\phi\colon H^{r}(V)\times H^{r}(V)\rightarrow Q{}(-r)
\]
be a nondegenerate bilinear form. For $\alpha\in\End(H^{r}(V))$, we let
$\alpha^{\prime}$ denote the adjoint of $\alpha$ with respect to $\phi$:%
\[
\phi(\alpha x,y)=\phi(x,\alpha^{\prime}y).
\]

\begin{definition}
\label{r97}We call $\phi$ a \textit{Weil form} if it satisfies the following conditions:
\begin{enumerate}
\item $\phi$ is symmetric or skew-symmetric according as $r$ is even or odd;

\item for all $\alpha\in A_{H}^{d}(V\times V)_{r}$, the adjoint $\alpha
^{\prime}\in A_{H}^{d}(V\times V)_{r}$; moreover, $\Tr(\alpha\circ
\alpha^{\prime})\in\mathbb{Q}{}$, and $\Tr(\alpha\circ\alpha^{\prime})$ $>0$
if $\alpha\neq0$.
\end{enumerate}
\end{definition}

\begin{example}
\label{r94}Let $V$ be a nonsingular projective variety over $\mathbb{C}{}$.
For all $r\geq0$, the pairing%
\[
H^{r}(V)\times H^{r}(V)\rightarrow\mathbb{C}{},\quad x,y\mapsto(x\cdot\ast y)
\]
is a Weil form (Serre 1960, Thm 2).
\end{example}

\begin{example}
\label{r63}Let $C$ be a curve over $k$, and let $J$ be its jacobian. The Weil
pairing $\phi\colon T_{\ell}J\times T_{\ell}J\rightarrow T_{\ell}%
\mathbb{G}_{m}$ extends by linearity to a $\mathbb{Q}{}_{\ell}$-bilinear form%
\[
H_{1}(C_{\mathrm{et}},\mathbb{Q}{}_{\ell})\times H_{1}(C_{\mathrm{et}%
},\mathbb{Q}{}_{\ell})\rightarrow\mathbb{Q}_{\ell}{}(1)
\]
on $H_{1}(C_{\mathrm{et}},\mathbb{Q}{}_{\ell})=\mathbb{Q}{}_{\ell}\otimes
T_{\ell}J$. This is Weil form (Weil 1948b, VI, n$^{\circ}$48, Thm 25).
\end{example}

\begin{proposition}
\label{r69}If there exists a Weil form on $H^{r}(V)$, then the $\mathbb{Q}{}%
$-algebra $A_{H}^{d}(V\times V)_{r}$ is semisimple.
\end{proposition}

\begin{proof}
It admits a positive involution $\alpha\mapsto\alpha^{\prime}$, and so we can
argue as in the proof of (\ref{r13}).
\end{proof}

\begin{proposition}
\label{r57}Let $\phi$ be a Weil form on $H^{r}(V)$, and let $\alpha$ be an
element of $A^{d}(V\times V)_{r}$ such that $\alpha\circ\alpha^{\prime}$ is an
integer $q$. For every homomorphism $\rho\colon\mathbb{Q}{}[\alpha
]\rightarrow\mathbb{C}{}$,%
\[
\rho(\alpha^{\prime})=\overline{\rho(\alpha)}\text{ and }\left\vert \rho
\alpha\right\vert =q^{1/2}.
\]

\end{proposition}

\begin{proof}
The proof is the same as that of (\ref{r39}).
\end{proof}

\subsection{The standard conjectures}

Roughly speaking, the standard conjectures state that the groups of algebraic
cycles modulo homological equivalence behave like the cohomology groups of a
K\"{a}hler manifold. Our exposition in this subsection follows Grothendieck
1969 and Kleiman 1968.\footnote{According to Illusie (2010): Grothendieck gave
a series of lectures on motives at the IH\'{E}S. One part was about the
standard conjectures. He asked John Coates to write down notes. Coates did it,
but the same thing happened: they were returned to him with many corrections.
Coates was discouraged and quit. Eventually, it was Kleiman who wrote down the
notes in \textit{Dix expos\'{e}s sur la cohomologie des sch\'{e}mas}.}

Let $H$ be a Weil cohomology theory over $k$ with coefficient field $Q$. We
assume that the hard Lefschetz theorem holds for $H$.\footnote{For $\ell$-adic
\'{e}tale cohomology, this was proved by Deligne as a consequence of his proof
of the Weil conjectures, and it follows for the other standard Weil cohomology
theories.} This means the following: let $L$ be the Lefschetz operator
$x\mapsto u\cdot x$ defined by the class $u$ of an ample divisor in
$H^{2}(V)(1)$; then, for all $r\leq d$, the map%
\[
L^{d-r}\colon H^{r}(V)\rightarrow H^{2d-r}(V)(d-r)
\]
is an isomorphism. As before, this gives decompositions%
\[
H^{r}(V)=\bigoplus_{j\geq\max(r-d,0)}L^{j}P^{r-2j}(V)
\]
with $P^{r}(V)$ equal to the kernel of $L^{d-r+1}\colon H^{r}(V)\rightarrow
H^{2d-r+2}(V)$. Hence $x\in H^{r}(V)$ has a unique expression as a sum%
\begin{equation}
x=\sum_{j\geq\max(r-d,0)}L^{j}x_{j},\quad x_{j}\in P^{r-2j}(V). \label{e34}%
\end{equation}
Define an operator $\Lambda\colon H^{r}(V)\rightarrow H^{r-2}(V)$ by%
\[
\Lambda x=\sum_{j\geq\max(r-d,0)}L^{j-1}x_{j}.
\]

\subsubsection{\noindent The standard conjectures of Lefschetz type}

\begin{description}
\item[A($V,L$):] For all $2r\leq d$, the isomorphism $L^{d-2r}\colon
H^{2r}(V)(r)\rightarrow H^{2d-2r}(V)(d-r)$ restricts to an isomorphism%
\[
A_{H}^{r}(V)\rightarrow A_{H}^{d-r}(V).
\]
Equivalently, $x\in H^{2r}(V)(r)$ is algebraic if $L^{d-2r}x$ is algebraic.

\item[B($V$):] The operator $\Lambda$ is algebraic, i.e., it lies in the image
of
\[
A_{H}^{d-1}(V\times V)\rightarrow H^{2d-2}(V\times V)(d-1)\simeq
\bigoplus\nolimits_{r\geq0}\Hom(H^{r+2}(V),H^{r}(V)).
\]

\item[C($V$):] The K\"{u}nneth projectors $\pi_{r}$ are algebraic.
Equivalently, the K\"{u}nneth isomorphism%
\[
H^{\ast}(V\times V)\simeq H^{\ast}(V)\otimes H^{\ast}(V)
\]
induces an isomorphism%
\[
A_{H}^{\ast}(V\times V)\simeq A_{H}^{\ast}(V)\otimes A_{H}^{\ast}(V).
\]

\end{description}

\begin{proposition}
\label{r73}There are the following relations among the conjectures.

\begin{enumerate}
\item Conjecture $A(V\times V,L\otimes1+1\otimes L)$ implies $B(V).$

\item If $B(V)$ holds for one choice of $L$, then it holds for all.

\item Conjecture $B(V)$ implies $A(V,L)$ (all $L$) and $C(V)$.
\end{enumerate}
\end{proposition}

\begin{proof}
Kleiman 1994, Theorem 4-1.
\end{proof}

Thus $A(V,L)$ holds for all $V$ and $L$ if and only if $B(V)$ holds for all
$V$; moreover, each conjecture implies Conjecture $C$.

\begin{example}
\label{r79}Let $k=\mathbb{F}{}$. A smooth projective variety $V$ over $k$
arises from a variety $V_{0}$ defined over a finite subfield $k_{0}$ of $k$.
Let $\pi$ be the corresponding Frobenius endomorphism of $V$, and let
$P_{r}(T)=\det(1-\pi T\mid H^{r}(V))$. According to the Cayley-Hamilton
theorem, $P_{r}(\pi)$ acts as zero on $H^{r}(V)$. Assume that the $P_{r}$ are
relatively prime (this is true, for example, if the Riemann hypothesis holds).
According to the Chinese remainder theorem, there are polynomials $P^{r}(T)\in
Q[T]$ such that%
\[
P^{r}(T)=\left\{
\begin{array}
[c]{ll}%
1 & \text{mod }P_{r}(T)\\
0 & \text{mod }P_{s}(T)\text{ for }s\neq r.
\end{array}
\right.
\]
Now $P^{r}(\pi)$ projects $H^{\ast}(V)$ onto $H^{r}(V)$, and so Conjecture
$C(V)$ is true.
\end{example}

\begin{example}
\label{r82}Conjecture $B(V)$ holds if $V$ is an abelian variety or a surface
with $\dim H^{1}(V)$\thinspace equal to twice the dimension of the Picard
variety of $V$ (Kleiman 1968, 2. Appendix).
\end{example}

These are essentially the only cases where the standard conjectures of
Lefschetz type are known (see Kleiman 1994).

\subsubsection{The standard conjecture of Hodge type}

For $r\leq d$, let $A_{H}^{r}(V)_{pr}$ denote the \textquotedblleft
primitive\textquotedblright\ part $A_{H}^{r}(V)\cap P^{2r}(V)$ of $A_{H}%
^{r}(V)$. Conjecture $A(V,L)$ implies that%
\[
A_{H}^{r}(V)=\bigoplus_{j\geq\max(2r-d,0)}L^{j}A^{r-j}(V)_{pr}\text{.}%
\]

\begin{description}
\item[I($V,L$):] For $r\leq d$, the symmetric bilinear form%
\[
x,y\mapsto(-1)^{r}(x\cdot y\cdot u^{d-2r})\colon A_{H}^{r}(V)_{pr}\times
A_{H}^{r}(V)_{pr}\rightarrow\mathbb{Q}{}%
\]
is positive definite.
\end{description}

In characteristic zero, the standard conjecture of Hodge type follows from
Hodge theory (see \ref{r68}). In nonzero characteristic, almost nothing is
known except for surfaces where it becomes Theorem \ref{r1}.

\subsubsection{Consequences of the standard conjectures}

\begin{proposition}
\label{r62}Assume the standard conjectures. For every $x\in A_{H}^{r}(V)$,
there exists a $y\in A_{H}^{d-r}(V)$ such that $x\cdot y\neq0$.
\end{proposition}

\begin{proof}
We may suppose that $x=L^{j}x_{j}$ with $x_{j}\in A_{H}^{r-j}(V)_{pr}$. Now
\[
(L^{j}x_{j}\cdot L^{j}x_{j}\cdot u^{d-2r})=(x_{j}\cdot x_{j}\cdot
u^{d-2r+2j})>0.
\]

\end{proof}

Using the decomposition (\ref{e34}), we define an operator $\ast\colon
H^{r}(V)\rightarrow H^{2d-r}(V)$ by\footnote{This differs from the definition
in k\"{a}hlerian geometry by some scalar factors. Over $\mathbb{C}{}$, our
form $x,y\mapsto(x\cdot\ast y)$ is positive definite on some direct summands
of $H^{r}(V)$ and negative definite on others, but this suffices to imply that
the involution $\alpha\mapsto\alpha^{\prime}$ is positive.}%
\[
\ast x=\sum_{j\geq\max(r-d,0)}(-1)^{\frac{(r-2j)(r-2j+1)}{2}}L^{d-r+j}%
(x_{j}).
\]

\begin{theorem}
\label{r61}Assume the standard conjectures. Then%
\[
H^{r}(V)\times H^{r}(V)\rightarrow Q,\quad x,y\mapsto(x\cdot\ast y),
\]
is a Weil form.
\end{theorem}

\begin{proof}
Kleiman 1968, 3.11.
\end{proof}

\begin{corollary}
The $\mathbb{Q}{}$-algebra $A_{H}^{d}(V\times V)_{r}$ is semisimple.
\end{corollary}

\begin{proof}
It admits a positive involution; see (\ref{r29}).
\end{proof}

\begin{theorem}
\label{r58}Assume the standard conjectures. Let $V$ be a connected nonsingular
projective variety of dimension $d$ over $k{}$, and let $f$ be an endomorphism
$V$. Suppose that there exists an integer $q>0$ and a hyperplane section $E$
of $V$ such that $f^{-1}(E)$ is algebraically equivalent to $qE$. Then, for
all $r\geq0$, the eigenvalues of $f$ acting on $H^{r}(V,\mathbb{Q}{})$ have
absolute value $q^{r/2}$.
\end{theorem}

\begin{proof}
Use $E$ to define the Lefschetz operator. Then $\alpha\circ\alpha^{\prime}=q$,
and we can apply (\ref{r57}).
\end{proof}

\begin{aside}
\label{r92}In particular, the standard conjectures imply that the Frobenius
endomorphism acts semisimply on \'{e}tale cohomology over $\mathbb{F}{}$. For
abelian varieties, this was proved by Weil in the 1940s, but there has been
almost no progress since then.
\end{aside}

\subsection{The standard conjectures and equivalences on algebraic cycles}

Our statement of the standard conjectures is relative to a choice of a Weil
cohomology theory. Grothendieck initially stated the standard conjectures in a
letter to Serre\footnote{\textit{Grothendieck-Serre Correspondence}, p.232,
1965.} for algebraic cycles modulo algebraic equivalence, but an example of
Griffiths shows that they are false in that context.

\label{relation}Recall that two algebraic cycles $Z$ and $Z^{\prime}$ on a
variety $V$ are \textit{rationally equivalent} if there exists an algebraic
cycle $\mathcal{Z}{}$ on $V\times\mathbb{P}{}^{1}$ such that $\mathcal{Z}%
{}_{0}=Z$ and $\mathcal{Z}{}_{1}=Z^{\prime}$; that they are
\textit{algebraically equivalent} if there exists a curve $T$ and an algebraic
cycle $\mathcal{Z}{}$ on $V\times T$ such that $\mathcal{Z}{}_{t_{0}}=Z$ and
$\mathcal{Z}{}_{t_{1}}=Z^{\prime}$ for two points $t_{0},t_{1}\in T(k)$; that
they are \textit{homologically equivalent }relative to some fixed Weil
cohomology $H$ if they have the same class in $H^{2\ast}(V)(\ast)$; and that
they are \textit{numerically equivalent} if $(Z\cdot Y)=(Z^{\prime}\cdot Y)$
for all algebraic cycles $Y$ of complementary dimension. We have%

\[
\mathrm{rat}\implies\mathrm{alg}\implies\mathrm{hom}\implies\mathrm{num}.
\]

For divisors, rational equivalence coincides with linear equivalence. Rational
equivalence certainly differs from algebraic equivalence, except over the
algebraic closure of a finite fields, where all four equivalence relations are
conjectured to coincide (folklore).

For divisors, algebraic equivalence coincides with numerical equivalence
(Matsusaka 1957). For many decades, it was believed that algebraic equivalence
and numerical equivalence coincide --- one of Severi's \textquotedblleft
self-evident\textquotedblright\ postulates even has this as a consequence${}$
(Brigaglia et al. 2004, p.327).

Griffiths (1969) surprised everyone by showing that, even in the classical
situation${}$, algebraic equivalence differs from homological
equivalence.\footnote{In fact, they aren't even close. The first possible
counterexample is for $1$-cycles on a $3$-fold, and, indeed, for a complex
algebraic variety $V$ of dimension $3$, the vector space%
\[
\frac{\left\{  \text{one-cycles homologically equivalent to zero}\right\}
}{\left\{  \text{one-cycles algebraically equivalent to zero}\right\}
}\otimes\mathbb{Q}{}%
\]
may be infinite dimensional (Clemens 1983).} However, there being no
counterexample, it remains a folklore conjecture that numerical equivalence
coincides with homological equivalence for the standard Weil cohomologies. For
$\ell$-adic \'{e}tale cohomology, this conjecture is stated in Tate 1964.

The standard conjectures for a Weil cohomology theory $H$ imply that numerical
equivalence coincides with homological equivalence for $H$ (see \ref{r62}). It
would be useful to have a statement of the standard conjectures independent of
any Weil cohomology theory. One possibility is to state them for the
$\mathbb{Q}{}$-vector spaces of algebraic cycles modulo smash-nilpotent
equivalence in the sense of Voevodsky 1995 (which implies homological
equivalence, and is conjectured to equal numerical equivalence).

\subsection{The standard conjectures and the conjectures of Hodge and Tate}

Grothendieck hoped that his standard conjectures would be more accessible than
the conjectures of Hodge and Tate, but these conjectures appear to be closely
intertwined. Before explaining this, I recall the statements of the Hodge and
Tate conjectures.

\begin{conjecture}
[Hodge]\label{r65}Let $V$ be a smooth projective variety over $\mathbb{C}{}$.
For all $r\geq0$, the $\mathbb{Q}{}$-subspace of $H^{2r}(V,\mathbb{Q}{})$
spanned by the algebraic classes is $H^{2r}(V,\mathbb{Q}{})\cap H^{r,r}(V)$.
\end{conjecture}

\begin{conjecture}
[Tate]\label{r66}Let $V_{0}$ be a smooth projective variety over a field
$k_{0}$ finitely generated over its prime field, and let $\ell$ be a prime
$\neq\mathrm{char}(k)$. For all $r\geq0$, the $\mathbb{Q}{}_{\ell}$-subspace
of $H^{2r}(V_{\mathrm{et}},\mathbb{Q}{}_{\ell}(r))$ spanned by the algebraic
classes consists exactly of those fixed by the action of $\Gal(k/k_{0})$. Here
$k$ is a separable algebraic closure of $k_{0}$ and $V=(V_{0})_{k}$.
\end{conjecture}

By the full Tate conjecture, I mean the Tate conjecture plus $\mathrm{num}%
=\mathrm{hom}(\ell)$ (cf. Tate 1994, 2.9).

\begin{E}
\label{r83}The Hodge conjecture implies the standard conjecture of Lefschetz
type over $\mathbb{C}{}$ (obviously). Conversely, the standard conjecture of
Lefschetz type over $\mathbb{C}{}$ implies the Hodge conjecture for abelian
varieties (Abdulali 1994, Andr\'{e} 1996), and hence the standard conjecture
of Hodge type for abelian varieties in all characteristics (Milne 2002).
\end{E}

\begin{E}
\label{r84}Tate's conjecture implies the standard conjecture of Lefschetz type
(obviously). If the full Tate conjecture is true over finite fields of
characteristic $p$, then the standard conjecture of Hodge type holds in
characteristic $p$.
\end{E}

Here is a sketch of the proof of the last statement. Let $k$ denote an
algebraic closure of $\mathbb{F}{}_{p}$. The author showed that the Hodge
conjecture for CM abelian varieties over $\mathbb{C}{}$ implies the standard
conjecture of Hodge type for abelian varieties over $k$. The proof uses only
that Hodge classes on CM abelian varieties are \textit{almost}-algebraic at
$p$ (see (\ref{r90}) below for this notion). The Tate conjecture for finite
subfields of $k$ implies the standard conjecture of Lefschetz type over $k$
(obviously), and, using ideas of Abdulali and Andr\'{e}, one can deduce that
Hodge classes on CM abelian varieties are almost-algebraic at $p$; therefore
the standard conjecture of Hodge type holds for abelian varieties over $k$.
The full Tate conjecture implies that the category of motives over $k$ is
generated by the motives of abelian varieties, and so the Hodge standard
conjecture holds for all nonsingular projective varieties over $k$. A
specialization argument now proves it for all varieties in characteristic $p$.

\section{Motives}

For Grothendieck, the theory of \textquotedblleft motifs\textquotedblright%
\ took on an almost mystical meaning as the reality that lay beneath the
\textquotedblleft shimmering ambiguous surface of things\textquotedblright.
But in their simplest form, as a universal Weil cohomology theory, the idea is
easy to explain. Having already written a \textquotedblleft
popular\textquotedblright\ article on motives (Milne 2009b), I shall be brief.

Let $\sim$ be an adequate equivalence relation on algebraic cycles, for
example, one of those listed on p.\pageref{relation}. We want to define a Weil
cohomology theory that is universal among those for which $\sim\Rightarrow
\hom$. We also want to have $\mathbb{Q}{}$ as the field of coefficients, but
we know that this is impossible if we require the target category to be that
of $\mathbb{Q}{}$-vector spaces, and so we only ask that the target category
be $\mathbb{Q}{}$-linear and have certain other good properties.

Fix a field $k$ and an adequate equivalence relation $\sim$. The category
$\Corr_{\sim}(k)$ of correspondences over $k$ has one object $hV$ for each
nonsingular projective variety over $k$; its morphisms are defined by%
\[
\Hom(hV,hW)=\bigoplus\nolimits_{r}\Hom^{r}(hV,hW)=\bigoplus\nolimits_{r}%
A_{\sim}^{\dim(V)+r}(V\times W).
\]
The bilinear map in (\ref{r76}) can now be written%
\[
\Hom^{r}(V_{1},V_{2})\times\Hom^{s}(V_{2},V_{3})\rightarrow\Hom^{r+s}%
(V_{1},V_{3}),
\]
and its associativity means that $\Corr_{\sim}(k)$ is a category. Let $H$ be a
Weil cohomology theory. An element of $\Hom^{r}(hV,hW)$ defines a homomorphism
$H^{\ast}(V)\rightarrow H^{\ast+r}(W)$ of degree $r$. In order to have the
gradations preserved, we consider
the category $\Corr_{\sim}^{0}(k)$ of correspondences of degree $0$, i.e., now%
\[
\Hom(hV,hW)=\Hom^{0}(hV,hW)=A_{\sim}^{\dim V}(V\times W).
\]
Every Weil cohomology theory such that $\sim\Rightarrow\hom$ factors uniquely
through $\Corr_{\sim}^{0}(k)$ as a \textit{functor to graded vector spaces}.

However, $\Corr_{\sim}^{0}(k)$ is not an abelian category. There being no good
notion of an abelian envelope of a category, we define $\mathcal{M}{}_{\sim
}^{\mathrm{eff}}(k)$ to be the pseudo-abelian envelope of $\Corr_{\sim}%
^{0}(k)$. This has objects $h(V,e)$ with $V$ as before and $e$ an idempotent
in the ring $A_{\sim}^{\dim(V)}(V\times V)$; its morphisms are defined by%
\[
\Hom(h(V,e),h(W,f))=f\circ\Hom(hV,hW)\circ e.
\]
This is a pseudo-abelian category, i.e., idempotent endomorphisms have kernels
and cokernels. The functor
\[
hV\rightsquigarrow h(V,\id)\colon\Corr_{\sim}^{0}(k)\rightarrow\mathcal{M}%
_{\sim}^{\mathrm{eff}}(k)
\]
fully faithful and universal among functors from $\Corr_{\sim}^{0}(k)$ to
pseudo-abelian categories. Therefore, every Weil cohomology theory such that
$\sim\Rightarrow\hom$ factors uniquely through $\mathcal{M}_{\sim
}^{\mathrm{eff}}(k)$.

The category $\mathcal{M}{}_{\sim}^{\mathrm{eff}}(k)$ is the \textit{category
of effective motives} over $k$. It is a useful category, but we need to
enlarge it in order to have duals. One of the axioms for a Weil cohomology
theory requires $H^{2}(\mathbb{P}{}^{1})$ to have dimension $1$. This means
that the object $H^{2}(\mathbb{P}{}^{1})$ is invertible in the category of
vector spaces, i.e., the functor $W\rightsquigarrow H^{2}(\mathbb{P}{}%
^{1})\otimes_{Q}W$ is an equivalence of categories. In $\mathcal{M}{}_{\sim
}^{\mathrm{eff}}(k)$, the object $h\mathbb{P}{}^{1}$ decomposes into a direct
sum%
\[
h\mathbb{P}{}^{1}=h^{0}\mathbb{P}{}^{1}\oplus h^{2}\mathbb{P}{}^{1},
\]
and we formally invert $h^{2}\mathbb{P}{}^{1}$. When we do this, we obtain a
category $\mathcal{M}{}_{\sim}(k)$ whose objects are triples $h(V,e,m)$ with
$h(V,e)$ as before and $m\in\mathbb{Z}{}$. Morphisms are defined by%
\[
\Hom(h(V,e,m),h(W,f,n))=f\circ A_{\sim}^{\dim(V)+n-m}(V,W)\circ e.
\]
This is the category of \textit{motives over }$k$. It contains the category of
effective motives as a full subcategory.

The category $\Corr_{\sim}^{0}(k)$ has direct sums and tensor products:%
\begin{align*}
hV\oplus hW  &  =h(V\sqcup W)\\
hV\otimes hW  &  =h(V\times W).
\end{align*}
Both of these structures extend to the category of motives. The functor
$-\otimes h^{2}(\mathbb{P}{}^{1})$ is $h(V,e,m)\mapsto h(V,e,m+1)$, which is
an equivalence of categories (because we allow negative $m$).

\subsection{Properties of the category of motives}

For the language of tensor categories, we refer the reader to Deligne and
Milne 1982.

\begin{E}
\label{r85}The category of motives is a $\mathbb{Q}{}{}$-linear pseudo-abelian
category. If $\sim=$\textrm{num}, it is semisimple abelian. If $\sim
=\mathrm{rat}$ and $k$ is not algebraic over a finite field, then it is not abelian.
\end{E}

To say that $\mathcal{M}{}_{\sim}(k)$ is $\mathbb{Q}{}$-linear just means that
the $\Hom$ sets are $\mathbb{Q}{}$-vector spaces and composition is
$\mathbb{Q}{}$-bilinear. Thus $\mathcal{M}{}_{\sim}(k)$ is $\mathbb{Q}{}%
$-linear and pseudo-abelian by definition. If $\sim=$\textrm{num}, the $\Hom$
sets are finite-dimensional $\mathbb{Q}{}$-vector spaces, and the
$\mathbb{Q}{}$-algebras $\End(M)$ are semisimple (\ref{r74}). A pseudo-abelian
category with these properties is a semisimple abelian category (i.e., an
abelian category such that every object is a direct sum of simple objects).
For the last statement, see Scholl 1994, 3.1.

\begin{E}
\label{r86}The category of motives is a rigid tensor category. If $V$ is
nonsingular projective variety of dimension $d$, then there is a
(Poincar\'{e}) duality%
\[
h(V)^{\vee}\simeq h(V)(d).
\]

\end{E}

A tensor category is a symmetric monoidal category. This means that every
finite (unordered) family of objects has a tensor product, well defined up to
a given isomorphism. It is rigid if every object $X$ admits a dual object
$X^{\vee}$. The proof of (\ref{r86}) can be found in Saavedra Rivano 1972, VI, 4.1.3.5.

\begin{E}
\label{r87}Assume that the standard conjecture $C$ holds for some Weil
cohomology theory.

\begin{enumerate}
\item The category $\mathcal{M}{}_{\sim}(k)$ is abelian if and only if $\sim
=$\textrm{nu}$\mathrm{m}$.

\item The category $\mathcal{\mathcal{M}{}}_{\text{\textrm{num}}}{}(k)$ is a
graded tannakian category.

\item The Weil cohomology theories such that hom$=$num correspond to fibre
functors on the Tannakian category $\mathcal{\mathcal{M}{}}%
_{\text{\textrm{num}}}{}(k)$.
\end{enumerate}
\end{E}

(a) We have seen (\ref{r74}) that $\mathcal{M}{}_{\sim}(k)$ is abelian if
$\sim=$\textrm{nu}$\mathrm{m}$; the converse is proved in Andr\'{e} 1996, Appendice.

(b) Let $V$ be a nonsingular projective variety over $k$. Let $\pi_{0}%
,\ldots,\pi_{2d}$ be the images of the K\"{u}nneth projectors in
$A_{\text{num}}^{d}(V\times V)$. We define a gradation on $\mathcal{\not M
}_{\text{\textrm{num}}}{}(k)$ by setting%
\[
h(V,e,m)^{r}=h(V,e\pi_{r+2m},m).
\]
Now the commutativity constraint can be modified so that%
\[
\dim(h(V,e,m))=\sum\nolimits_{r\geq0}\dim_{Q}(eH^{r}(V)).
\]
(rather than the alternating sum). Thus $\dim$($M)\geq0$ for all motives $M$,
which implies that $\mathcal{M}{}_{\text{\textrm{num}}}{}(k)$ is tannakian
(Deligne 1990, 7.1).

(c) Let $\omega$ be a fibre functor on $\mathcal{M}{}_{\text{\textrm{num}}}%
{}(k)$. Then $V\rightsquigarrow\omega(hV)$ is a Weil cohomology theory such
that hom$=$num, and every such cohomology theory arises in this way.

\begin{E}
\label{r89}The standard conjectures imply that $\mathcal{M}{}_{\text{num}}(k)$
has a canonical polarization.
\end{E}

The notion of a Weil form can be defined in any tannakian category over
$\mathbb{Q}{}$. For example, a Weil form on a motive $M$ of weight $r$ is a
map%
\[
\phi\colon M\otimes M\rightarrow\1(-r)
\]
with the correct parity such that the map sending an endomorphism of $M$ to
its adjoint is a positive involution on the $\mathbb{Q}{}$-algebra $\End(M)$.
To give a \textit{polarization} on $\mathcal{M}{}_{\text{num}}(k)$ is to give
a set of Weil forms (said to be \textit{positive} for the polarization) for
each motive satisfying certain compatibility conditions; for example, if
$\phi$ and $\phi^{\prime}$ are positive, then so also are $\phi\oplus
\phi^{\prime}$ and $\phi\otimes\phi^{\prime}$. The standard conjectures
(especially of Hodge type) imply that $\mathcal{M}{}_{\text{num}}(k)$ admits a
polarization for which the Weil forms defined by ample divisors are positive.

\medskip Let $\1$ denote the identity object $h^{0}(\mathbb{P}{}^{0})$ in
$\mathcal{M}{}_{\sim}(k)$. Then, almost by definition,%
\[
A_{H}^{r}(V)\simeq\Hom(\1,h^{2r}(V)(r)).
\]
Therefore $V\rightsquigarrow h^{r}(V)$ has the properties expected of an
\textquotedblleft abstract\textquotedblright\ Weil cohomology theory.

\subsection{Alternatives to the Hodge, Tate, and standard conjectures}

In view of the absence of progress on the Hodge, Tate, or standard conjectures
since they were stated more than fifty years ago, Deligne has suggested that,
rather than attempting to proving these conjectures, we should look for a good
theory of motives, based on \textquotedblleft
algebraically-defined\textquotedblright, but not necessarily algebraic,
cycles. I discuss two possibilities for such algebraically-defined of cycles.

\subsubsection{Absolute Hodge classes and rational Tate classes}

Consider an algebraic variety $V$ over an algebraically closed field $k$ of
characteristic zero. If $k$ is not too big, then we can choose an embedding
$\sigma\colon k\rightarrow\mathbb{C}{}$ of $k$ in $\mathbb{C}{}$ and define a
cohomology class on $k$ to be \textit{Hodge relative to} $\sigma$ if it
becomes Hodge on $\sigma V$. The problem with this definition is that it
depends on the choice of $\sigma$. To remedy this, Deligne defines a
cohomology class on $V$ to be \textit{absolutely Hodge} if it is Hodge
relative to every embedding $\sigma$. The problem with this definition is
that, a priori, we know little more about absolute Hodge classes than we do
about algebraic classes. To give substance to his theory, Deligne proved that
all relative Hodge classes on abelian varieties are absolute. Hence they
satisfy the standard conjectures and the Hodge conjecture, and so we have a
theory of abelian motives over fields of characteristic zero that is much as
Grothendieck envisaged it (Deligne 1982).

However, as Deligne points out, his theory works only in characteristic zero,
which limits its usefulness for arithmetic questions. Let $A$ be an abelian
variety over the algebraic closure $k$ of a finite field, and lift $A$ in two
different ways to abelian varieties $A_{1}$ and $A_{2}$ in characteristic
zero; let $\gamma_{1}$ and $\gamma_{2}$ be absolute Hodge classes on $A_{1}$
and $A_{2}$ of complementary dimension; then $\gamma_{1}$ and $\gamma_{2}$
define $l$-adic cohomology classes on $A$ for all primes $l$, and hence
intersection numbers $\left(  \gamma_{1}\cdot\gamma_{2}\right)  _{l}%
\in\mathbb{Q}{}_{l}$. The Hodge conjecture implies that $\left(  \gamma
_{1}\cdot\gamma_{2}\right)  _{l}$ lies in $\mathbb{Q}{}$ and is independent of
$l$, but this is not known.

However, the author has defined the notion of a \textquotedblleft good theory
of rational Tate classes\textquotedblright\ on abelian varieties over finite
fields, which would extend Deligne's theory to mixed characteristic. It is
known that there exists at most one such theory and, if it exists, the
rational Tate classes it gives satisfy the standard conjectures and the Tate
conjecture; thus, if it exists, we would have a theory of abelian motives in
mixed characteristic that is much as Grothendieck envisaged it (Milne 2009a).

\subsubsection{Almost-algebraic classes}

\begin{plain}
\label{r90}Let $V$ be an algebraic variety over an algebraically closed field
$k$ of characteristic zero. An \textit{almost-algebraic class} on $V$ of
codimension $r$ is an absolute Hodge class $\gamma$ on $V$ of codimension $r$
such that there exists a cartesian square%
\[
\begin{tikzcd}
\mathcal{V}\arrow{d}{f}& V\arrow{l}\arrow{d}\\
S & \Spec k\arrow{l}
\end{tikzcd}
\]
and a global section $\tilde{\gamma}$ of $R^{2r}f_{\ast}\mathbb{A}{}(r)$
satisfying the following conditions (cf. Tate 1994, p.76),

\begin{itemize}
\item $S$ is the spectrum of a regular integral domain of finite type over
$\mathbb{Z}{};$

\item $f$ is smooth and projective;

\item the fibre of $\tilde{\gamma}$ over $\Spec(k)$ is $\gamma$, and the
reduction of $\tilde{\gamma}$ at $s$ is algebraic for all closed points $s$ in
a dense open subset $U$ of $S$.
\end{itemize}

\noindent\ If the above data can be chosen so that $(p)$ is in the image of
the natural map $U\rightarrow\Spec\mathbb{Z}{}$, then we say that $\gamma$ is
\textit{almost-algebraic at }$p$ --- in particular, this means that $V$ has
good reduction at $p$.
\end{plain}

\begin{plain}
\label{r91}Note that the residue field $\kappa(s)$ at a closed point $s$ of
$S$ is finite, and so the K\"{u}nneth components of the diagonal are
almost-algebraic (\ref{r79}). Therefore the space of almost-algebraic classes
on $X$ is a graded $\mathbb{Q}{}$-subalgebra%
\[
AA^{\ast}(X)=\bigoplus\nolimits_{r\geq0}AA^{r}(X)
\]
of the $\mathbb{Q}{}$-algebra of $AH^{\ast}(X)$ of absolute Hodge classes. For
any regular map $f\colon Y\rightarrow X$ of complete smooth varieties, the
maps $f_{\ast}$ and $f^{\ast}$ send almost-algebraic classes to
almost-algebraic classes. Similar statements hold for the $\mathbb{Q}{}%
$-algebra of classes almost-algebraic at $p$.
\end{plain}

\subsection{Beyond pure motives}

In the above I considered only pure motives with rational coefficients. This
theory should be generalized in (at least) three different directions. Let $k$
be a field.

\begin{itemize}
\item There should be a category of mixed motives over $k$. This should be an
abelian category whose whose semisimple objects form the category of pure
motives over $k$. Every mixed motive should be equipped with a weight
filtration whose quotients are pure motives. Every algebraic variety over $k$
(not necessarily nonsingular or projective) should define a mixed motive.

\item There should be a category of complexes of motives over $k$. This should
be a triangulated category with a $t$-structure whose heart is the category of
mixed motives. Each of the standard Weil cohomology theories lifts in a
natural way to a functor from all algebraic varieties over $k$ to a
triangulated category; these functors should factor through the category of
complexes of motives.

\item Everything should work mutatis mutandis over $\mathbb{Z}{}$.
\end{itemize}

\noindent Much of this was envisaged by Grothendieck.\footnote{Certainly,
Grothendieck envisaged mixed motives and motives with coefficients in
$\mathbb{Z}{}$ (see the footnote on p.5 of Kleiman 1994 and Grothendieck's
letters to Serre and Illusie).} There has been much progress on these
questions, which I shall not attempt to summarize (see Andr\'{e} 2004, Mazza
et al. 2006, Murre et al. 2013).

\section{Deligne's proof of the Riemann hypothesis over finite fields}

Grothendieck attempted to deduce the Riemann hypothesis in arbitrary
dimensions from the curves case,\footnote{Serre writes (email July 2015):
\bquote He was not the first one. So did Weil, around 1965; he looked at
surfaces, made some computations, and tried (vainly) to prove a positivity
result. He talked to me about it, but he did not publish anything.\equote} but
no \textquotedblleft d\'{e}vissage\textquotedblright\ worked for
him.\footnote{\textquotedblleft I have no comments on your attempts to
generalize the Weil-Castelnuovo inequality\ldots; as you know, I have a sketch
of a proof of the Weil conjectures starting from the curves
case.\textquotedblright\ (Grothendieck, \textit{Grothendieck-Serre
Correspondence}, p.88, 1959). Grothendieck hoped to prove the Weil conjectures
by showing that every variety is birationally a quotient of a product of
curves, but Serre constructed a counterexample. See \textit{Grothendieck-Serre
Correspondence}, pp.145--148, 1964, for a discussion of this and
Grothendieck's \textquotedblleft second attack\textquotedblright\ on the Weil
conjectures.} After he announced the standard conjectures, the conventional
wisdom became that, to prove the Riemann hypothesis in dimension $>1$, one
should prove the standard conjectures.\footnote{\textquotedblleft Pour aller
plus loin et g\'{e}n\'{e}raliser compl\`{e}tement les r\'{e}sultats obtenus
par Weil pour le cas $n=1$, il faudrait prouver les propri\'{e}t\'{e}s
suivantes de la cohomologie $\ell$-adique \ldots\lbrack the standard
conjectures]\textquotedblright\ Dieudonn\'{e} 1974, IX n$^{\circ}$19, p. 224.}
However, Deligne recognized the intractability of the standard conjectures and
looked for other approaches. In 1973 he startled the mathematical world by
announcing a proof of the Riemann hypothesis for all smooth projective
varieties over finite fields.\footnote{A little earlier, he had proved the
Riemann hypothesis for varieties whose motive, roughly speaking, lies in the
category generated by abelian varieties (Deligne 1972).} How this came about
is best described in the following conversation.\footnote{This is my
transcription (slightly edited) of part of a telephone conversation which took
place at the ceremony announcing the award of the 2013 Abel Prize to
Deligne.}\medskip

\textsc{Gowers:} Another question I had. Given the clearly absolutely
remarkable nature of your proof of the last remaining Weil conjecture, it does
make one very curious to know what gave you the idea that you had a chance of
proving it at all. Given that the proof was very unexpected, it's hard to
understand how you could have known that it was worth working on.\smallskip

\textsc{Deligne:}\textbf{ }That's a nice story. In part because of Serre, and
also from listening to lectures of Godement, I had some interest in
automorphic forms. Serre understood that the $p^{11/2}$ in the Ramanujan
conjecture should have a relation with the Weil conjecture itself. A lot of
work had been done by Eichler and Shimura, and by Verdier, and so I understood
the connection between the two. Then I read about some work of Rankin, which
proved, not the estimate one wanted, but something which was a $1/4$ off ---
the easy results were $1/2$ off from what one wanted to have. As soon as I saw
something like that I knew one had to understand what he was doing to see if
one could do something similar in other situations. And so I looked at Rankin,
and there I saw that he was using a result of Landau --- the idea was that
when you had a product defining a zeta function you could get information on
the local factors out of information on the pole of the zeta function itself.
The poles were given in various cases quite easily by Grothendieck's theory.
So then it was quite natural to see what one could do with this method of
Rankin and Landau using that we had information on the pole. I did not know at
first how far I could go. The first case I could handle was a hypersurface of
odd dimension in projective space. But that was a completely new case already,
so then I had confidence that one could go all the way. It was just a matter
of technique.\smallskip

\textsc{Gowers:} It is always nice to hear that kind of thing. Certainly, that
conveys the idea that there was a certain natural sequence of steps that
eventually led to this amazing proof.\smallskip

\textsc{Deligne:} Yes, but in order to be able to see those steps it was
crucial that I was not only following lectures in algebraic geometry but some
things that looked quite different (it would be less different now) the theory
of automorphic forms. It was the discrepancy in what one could do in the two
areas that gave the solution to what had to be done.\smallskip

\textsc{Gowers:} Was that just a piece of good luck that you happened to know
about both things.\smallskip

\textsc{Deligne }(emphatically): Yes.

\subsection{Landau, Rankin, and the Ramanujan conjecture}

\subsubsection{Landau's theorem}

Consider a Dirichlet series $D(s)=\sum_{n=1}^{\infty}c_{n}n^{-s}$. Recall that
there is a unique real number (possibly $\infty$ or $-\infty$) such that
$D(s)$ converges absolutely for $\Re(s)>\sigma_{0}$ and does not converge
absolutely for $\Re(s)<\sigma_{0}$. Landau's theorem states that, if the
$c_{n}$ are real and nonnegative, then $D(s)$ has a singularity at
$s=\sigma_{0}$. For example, suppose that%
\[
D(s)=\sum_{n\geq1}c_{n}n^{-s}=\prod_{p\text{ prime}}\frac{1}{(1-a_{p,1}%
p^{-s})\cdots(1-a_{p,m}p^{-s})};
\]
if the $c_{n}$ are real and nonnegative, and $\sum c_{n}n^{-s}$ converges
absolutely for $\Re(s)>\sigma_{0}$, then%
\[
|a_{p,j}|\leq p^{\sigma_{0}},\quad\text{all }p\text{, }j.
\]

\subsubsection{Ramanujan's conjecture}

Let $f=\sum_{n\geq1}c(n)q^{n}$, $c(1)=1$, be a cusp form of weight $2k$ which
is a normalized eigenfunction of the Hecke operators $T(n)$. For example,
Ramanujan's function%
\[
\Delta(\tau)=\sum_{n=1}^{\infty}\tau(n)q^{n}\overset{\textup{{\tiny def}}}%
{=}q\prod_{n=1}^{\infty}(1-q^{n})^{24}%
\]
is such function of weight $12$. Let%
\[
\Phi_{f}(s)=\sum_{n=1}^{\infty}c(n)n^{-s}%
\]
be the Dirichlet series associated with $f$. Because $f$ is an eigenfunction%
\[
\Phi_{f}(s)=\prod_{p\text{ prime}}\frac{1}{1-c(p)p^{-s}+p^{2k-1-2s}}\text{.}%
\]
Write%
\[
(1-c(p)T+p^{2k-1}T^{2}=(1-a_{p}T)(1-a_{p}^{\prime}T).
\]
The Ramanujan-Petersson conjecture says that $a_{p}$ and $a_{p}^{\prime}$ are
complex conjugate. As $a_{p}a_{p}^{\prime}=p^{2k-1}$, this says that%
\[
\left\vert a_{p}\right\vert =p^{k-1/2};
\]
equivalently,%
\[
\left\vert c(n)\right\vert \leq n^{k-1/2}\sigma_{0}(n).
\]
For $f=\Delta$, this becomes Ramanujan's original conjecture, $|\tau
(p)|\leq2p^{\frac{11}{2}}$. Hecke showed that, for a cusp form $f$ of weight
$2k$,%
\[
\left\vert c(n)\right\vert =O(n^{k}).
\]

There is much to be said about geometric interpretations of Ramanujan's
function --- see Serre 1968 and the letter from Weil to Serre at the end of
the article.

\subsubsection{Rankin's theorem}

Let $f=\sum c(n)q^{n}$ be a cusp form of weight $2k$, and let $\mathcal{F}%
{}(s)=\sum_{n\geq1}\left\vert c(n)\right\vert ^{2}n^{-s}$. Rankin (1939) shows
that $\mathcal{F}{}(s)$ can be continued to a meromorphic function on the
whole plane with singularities only at $s=k$ and $s=k-1$.\footnote{Selberg did
something similar about the same time. Their approach to proving the analytic
continuation of convolution $L$-functions has become known as the
Rankin-Selberg method.} Using Landau's theorem, he deduced that%
\[
|c(n)|=O(n^{k-1/5}).
\]

\subsubsection{A remark of Langlands}

Langlands remarked (1970, pp.21--22) that Rankin's idea could be used to
\textit{prove} a generalized Ramanujan conjecture provided one knew enough
about the poles of a certain family of Dirichlet
series.\footnote{\textquotedblleft The remark about the consequences of
functoriality for Frobenius-Hecke conjugacy classes of an automorphic
representation, namely that their eigenvalues are frequently all of absolute
value of 1, occurred to me on a train platform in Philadelphia, as I thought
about the famous Selberg-Rankin estimate. \textquotedblright\ Langlands 2015,
p.200}

Given an automorphic representation $\pi=\bigotimes\pi_{p}$ of an algebraic
group $G$ and a representation $\sigma$ of the corresponding $L$-group,
Langlands defines an $L$-function%
\[
L(s,\sigma,\pi)=\prod_{\lambda}\left\{  \left(  \text{power of }\pi\right)
\cdot\left(  \Gamma\text{-factor}\right)  \cdot\prod_{p\text{ prime}}\frac
{1}{1-\lambda(t_{p})p^{-s}}\right\}  ^{m(\lambda)}%
\]
Here $\lambda$ runs over certain characters, $m(\lambda)$ is the multiplicity
of $\lambda$ in $\sigma$, and $\left\{  t_{p}\right\}  $ is the (Frobenius)
conjugacy class attached to $\pi_{p}$. Under certain hypotheses, the
generalized Ramanujan conjecture states that, for all $\lambda$ and $p$,%
\[
\left\vert \lambda(t_{p})\right\vert =1\text{.}%
\]

Assume that, for all $\sigma$, the $L$-series $L(s,\sigma,\pi)$ is analytic
for $\Re(s)>1$. The same is then true of $D(s,\sigma)\overset
{\textup{{\tiny def}}}{=}\prod_{\lambda}\left\{  \prod\nolimits_{p\text{
prime}}\frac{1}{1-\lambda(t_{p})p^{-s}}\right\}  ^{m(\lambda)}$ because the
$\Gamma$-function has no zeros. Let $\sigma=\rho\otimes\bar{\rho}$. The
logarithm of $D(s,\sigma)$ is $\sum_{p}\sum_{n=1}^{\infty}\frac{\text{trace
}\sigma^{n}(t_{p})}{n}p^{-ns}$. As%
\[
\text{trace }\sigma^{n}(t_{p})=\text{trace }\rho^{n}(t_{p})\cdot\text{trace
}\bar{\rho}^{n}(t_{p})=|\text{trace }\rho^{n}(t_{p})|^{2},
\]
the series for $\log D(s,\sigma)$ has positive coefficients. The same is
therefore true of the series for $D(s,\sigma)$. We can now apply Landau's
theorem to deduce that $|\lambda(t_{p})|\leq p$. If $\lambda$ occurs in some
$\sigma$, we can choose $\rho$ so that $m\lambda$ occurs in $\rho$. Then
$(m\lambda)(t_{p})=\lambda(t_{p})^{m}$ is an eigenvalue of $\rho(t_{p})$ and
$\overline{\lambda(t_{p})}^{m}$ is an eigenvalue of $\bar{\rho}$; hence
$\left\vert \lambda(t_{p})\right\vert ^{2m}$ is an eigenvalue of $\sigma$, and
$|\lambda(t_{p})|\leq p^{1/2m}$ for all $m$. On letting $m\rightarrow\infty$,
we see that $|\lambda(t_{p})|\leq1$ for all $\lambda$. Replacing $\lambda$
with $-\lambda$, we deduce that $|\lambda(t_{p})|=1$.

Note that it is the flexibility of Langlands's construction that makes this
kind of argument possible. According to Katz (1976, p.288), Deligne studied
Rankin's original paper in an effort to understand this remark of Langlands.

\subsection{Grothendieck's theorem}

Let $U{}$ be a connected topological space. Recall that a \textit{local system
of }$\mathbb{Q}{}$\textit{-vector spaces} on $U{}$ is a sheaf $\mathcal{E}{}$
on $U{}$ that is locally isomorphic to the constant sheaf $\mathbb{Q}{}^{n}$
for some $n$. For example, let $f\colon V\rightarrow U{}$ be a smooth
projective map of algebraic varieties over $\mathbb{C}{}$; for $r\geq0$,
$R^{r}f_{\ast}\mathbb{Q}{}$ is a locally constant sheaves of $\mathbb{Q}{}%
$-vector spaces with fibre%
\[
(R^{r}f_{\ast}\mathbb{Q}{})_{u}\simeq H^{r}(V_{u},\mathbb{Q}{})
\]
for all $u\in U{}(\mathbb{C}{})$. Fix a point $o\in U{}$. There is a canonical
(monodromy) action $\rho$ of $\pi_{1}(U{},o)$ on the finite-dimensional
$\mathbb{Q}{}$-vector space $\mathcal{E}{}_{o}$, and the functor
$\mathcal{E}{}\rightsquigarrow(\mathcal{E}{}_{o},\rho)$ is an equivalence of categories.

Now let $U{}$ be a connected nonsingular algebraic variety over a field $k$.
In this case, there is a notion of a \textit{local system of }$\mathbb{Q}%
{}_{\ell}$\textit{-vector spaces} on $U{}$ (often called a smooth or lisse
sheaf). For a smooth projective map of algebraic varieties $f\colon
V\rightarrow U{}$ and integer $r$, the direct image sheaf $R^{r}f_{\ast
}\mathbb{Q}{}_{\ell}$ on $U{}$ is a locally constant sheaf of $\mathbb{Q}%
{}_{\ell}$-vector spaces with fibre%
\[
(R^{r}f_{\ast}\mathbb{Q}{}_{\ell})_{u}\simeq H^{r}(V_{u},\mathbb{Q}{}_{\ell})
\]
for all $u\in U{}$.

The \'{e}tale fundamental group $\pi_{1}(U{})$ of $U{}$ classifies the finite
\'{e}tale coverings of $U{}$.\footnote{When the base field is $\mathbb{C}{}$,
the topological fundamental group of $V^{\text{an}}$ classifies the covering
spaces of $V^{\text{an}}$ (not necessarily finite). The Riemann existence
theorem implies that the two groups have the same finite quotients, and so the
\'{e}tale fundamental group is the profinite completion of the topological
fundamental group.} Let $\mathcal{E}{}$ be a local system on $U{}$, and let
$E=\mathcal{E}{}_{\eta}$ be its fibre over the generic point $\eta$ of $U{}$.
Again, there is a \textquotedblleft monodromy\textquotedblright\ action $\rho$
of $\pi_{1}(U{})$ on the finite-dimensional $\mathbb{Q}{}_{\ell}$-vector space
$E$, and the functor $\mathcal{E}{}\rightsquigarrow(E,\rho)$ is an equivalence
of categories.

Now let $U{}_{0}$ be a nonsingular geometrically-connected variety over finite
field $k_{0}$; let $k$ be an algebraic closure of $k_{0}$, and let $U{}%
=(U{}_{0})_{k}$. Then there is an exact sequence%
\[
0\rightarrow\pi_{1}(U{})\rightarrow\pi_{1}(U{}_{0})\rightarrow\Gal(k/k_{0}%
)\rightarrow0.
\]
The maps reflect the fact that a finite extension $k^{\prime}/k_{0}$ pulls
back to a finite covering $U{}^{\prime}\rightarrow U{}_{0}$ of $U_{0}$ and a
finite \'{e}tale covering of $U{}_{0}$ pulls back to a finite \'{e}tale
covering of $U{}$.

Let $\mathcal{E}{}$ be a local system of $\mathbb{Q}{}_{\ell}$-vector spaces
on $U{}_{0}$. Grothendieck (1964)\footnote{In fact, Grothendieck only sketched
the proof of his theorem in this Bourbaki talk, but there are detailed proofs
in the literature, for example, in the author's book on \'{e}tale cohomology.}
proved the following trace formula%
\begin{equation}
\sum\nolimits_{u\in|U{}|^{\pi}}\Tr(\pi_{u}\mid\mathcal{E}{}_{u})=\sum
\nolimits_{r}(-1)^{r}\Tr(\pi\mid H_{c}^{r}(U{},\mathcal{E}{})). \label{e35}%
\end{equation}
Here $|U|$ denotes the set of closed points of $U$, $\pi_{u}$ denotes the
local Frobenius element acting on the fibre $\mathcal{E}{}_{u}$ of
$\mathcal{E}{}$ at $u$, and $\pi$ denotes the reciprocal of the usual
Frobenius element of $\Gal(k/k_{0})$. Let
\[
Z(U{}_{0},\mathcal{E}{}_{0},T)=\prod_{u\in\left\vert U_{0}\right\vert }%
\frac{1}{\det(1-\pi_{u}T^{\deg(u)}\mid\mathcal{E}{}_{u})};
\]
written multiplicatively, (\ref{e35}) becomes%
\begin{equation}
Z(U{}_{0},\mathcal{E}{}_{0},T)=\prod\nolimits_{r}\det(1-\pi T\mid H_{c}%
^{r}(U{},E))^{(-1)^{r+1}} \label{e36}%
\end{equation}
(cf. (\ref{e40}), p.\pageref{e40}). For the constant sheaf $\mathbb{Q}{}%
_{\ell}$, (\ref{e36}) becomes ((\ref{e37}), p.\pageref{e37}). Grothendieck
proves (\ref{e35}) by a d\'{e}vissage to the case that $U{}$ is of dimension
$1$.\footnote{This proof of the rationality of the zeta function of an
arbitrary variety is typical of Grothendieck: find a generalization that
allows the statement to be proved by a \textquotedblleft
devissage\textquotedblright\ to the (relative) dimension one case.}

We shall need (\ref{e36}) in the case that $U_{0}$ is an affine curve, for
example, $\mathbb{A}{}^{1}$. In this case it becomes%
\begin{equation}
\prod_{u\in|U_{0}|}\frac{1}{\det(1-\pi_{u}T^{\deg(u)}\mid\mathcal{E}{}_{u}%
)}=\frac{\det(1-\pi T\mid H_{c}^{1}(U{},E))}{\det(1-\pi T\mid E_{\pi_{1}(U{}%
)}(-1))}\text{.} \label{e38}%
\end{equation}

\noindent Here $E_{\pi_{1}(U{})}$ is the largest quotient of $E$ on which
$\pi_{1}(U{})$ acts trivially (so the action of $\pi_{1}(U{}_{0})$ on it
factors through $\Gal(k/k_{0}))$.

\subsection{Proof of the Riemann hypothesis for hypersurfaces of odd degree}

Following Deligne, we first consider a hypersurface $V_{0}$ in $\mathbb{P}%
{}_{k_{0}}^{d+1}$ of odd degree $\delta$. For a hypersurface, the cohomology
groups coincide with those of the ambient space except in the middle degree,
and so%
\[
Z(V,T)=\frac{P_{d}(V,T)}{(1-T)(1-qT)\cdots(1-q^{d}T)},\quad
\]
with $P_{d}(V,T)=\det(1-\pi T\mid H^{d}(V,\mathbb{Q}{}_{\ell}))$. Note that
\[
P_{d}(V,T)=Z(V,T)(1-T)\cdots(1-q^{d}T)\in\mathbb{Q}{}[T].
\]

We embed our hypersurface in a one-dimensional family. The homogeneous
polynomials of degree $\delta$ in $d+2$ variables, considered up to
multiplication by a nonzero scalar, form a projective space $\mathbb{P}{}^{N}$
with $N=\left(
\begin{smallmatrix}
d+\delta+1\\
\delta
\end{smallmatrix}
\right)  $. There is a diagram%
\[
\begin{tikzcd}
H\arrow[hook]{r}\arrow{d}{f}
&\mathbb{P}^{N}\times\mathbb{P}^{d+1}\arrow{ld}\\
\mathbb{P}^{N}%
\end{tikzcd}
\]
in which the fibre $H_{P}$ over a point $P$ of $\mathbb{P}{}^{N}(k)$ is the
hypersurface (possibly reducible) in $\mathbb{P}{}^{d+1}$ defined by $P$. Let
$P_{0}$ be the polynomial defining $V_{0}$. We choose a \textquotedblleft
general\textquotedblright\ line through $P_{0}$ in $\mathbb{P}{}^{N}$ and
discard the finitely many points where the hypersurface is singular of not
connected. In this way, we obtain a smooth projective map $f\colon
\mathcal{V}{}_{0}\rightarrow U_{0}\subset\mathbb{P}{}^{1}$ whose fibres are
hypersurfaces of degree $\delta$ in $\mathbb{P}^{d+1}$; let $o\in U_{0}$ be
such that the fibre over $o$ is our given hypersurface $V_{0}$.

Let $\mathcal{E}{}_{0}=R^{d}f_{\ast}\mathbb{Q}{}_{\ell}$ and let $E$ denote
the corresponding $\pi_{1}(U_{0})$-module. Then%
\begin{equation}
Z(U_{0},\mathcal{E}{}_{0},T)=\prod_{u\in|U_{0}|}\frac{1}{P_{d}(V_{u}%
,T^{\deg(u)})}=\frac{\det(1-\pi T\mid H_{c}^{1}(U,E))}{\det(1-\pi T\mid
E_{\pi_{1}(U)}(-1))}. \label{e41}%
\end{equation}

There is a canonical pairing of sheaves%
\[
R^{d}f_{\ast}\mathbb{Q}_{\ell}\times R^{d}f_{\ast}\mathbb{Q}{}_{\ell
}\rightarrow R^{2d}f_{\ast}\mathbb{Q}{}_{\ell}\simeq\mathbb{Q}{}_{\ell}(-d)
\]
which, on each fibre becomes cup product%
\[
H^{d}(\mathcal{V}{}_{u},\mathbb{Q}{}_{\ell})\times H^{d}(\mathcal{V}{}%
_{u},\mathbb{Q}{}_{\ell})\rightarrow H^{2d}(\mathcal{V}{}_{u},\mathbb{Q}%
{}_{\ell})\simeq\mathbb{Q}{}_{\ell}(-d).
\]
This gives a pairing%
\[
\psi\colon E_{0}\times E_{0}\rightarrow\mathbb{Q}{}_{\ell}(-d)
\]
which is skew-symmetric (because $d$ is odd), non-degenerate (by Poincar\'{e}
duality for the geometric generic fibre of $f$), and $\pi_{1}(U_{0}%
)$-invariant (because it arises from a map of sheaves on $U_{0}$). Deligne
proves that, if the line through $P_{0}$ in $\mathbb{P}{}^{N}$ is chosen to be
sufficiently general,\footnote{This may require a finite extension of the base
field $k_{0}$.} then the image of $\pi_{1}(U)$ in $\SP(E,\psi)$ is dense for
the Zariski topology (i.e., there is \textquotedblleft big geometric
monodromy\textquotedblright). Thus the coinvariants of $\pi_{1}(U)$ in $E$ are
equal to the coinvariants of $\SP(E,\psi)$, and the invariant theory of the
symplectic group is well understood.

Note that%
\[
\log Z(U_{0},\mathcal{E}{}_{0},T)=\sum_{u\in|U_{0}|}\sum_{n}\Tr(\pi_{u}%
^{n}\mid E_{u})\frac{(T^{\deg u})^{n}}{n}.
\]
The coefficients $\Tr(\pi_{u}^{n}\mid E_{u})$ are rational. When we replace
$\mathcal{E}_{0}$ with $\mathcal{E}{}_{0}^{\otimes2m}$, we replace
$\Tr(\pi_{u}^{n}\mid E_{u})$ with%
\[
\Tr(\pi_{u}^{n}\mid E_{u}^{\otimes2m})=\Tr(\pi_{u}^{n}\mid E_{u})^{2m}%
\]
which is $\geq0$.

To apply Landau's theorem, we need to find the poles of $Z(U_{0}%
,\mathcal{E}_{0}^{\otimes2m},q^{-s})$, which, according to Grothendieck's
theorem (see (\ref{e41})) are equal to the zeros of $\det(1-\pi T\mid
E_{\pi_{1}(U)}(-1))$. Note that%
\[
\Hom(E^{\otimes2m},\mathbb{Q}{}_{\ell})^{\SP(\psi)}=\Hom((E^{\otimes
2m})_{\SP(\psi)},\mathbb{Q}{}_{\ell}).
\]
The map%
\[
x_{1}\otimes\cdots\otimes x_{2m}\mapsto\psi(x_{1},x_{2})\cdots\psi
(x_{2m-1},x_{2m})\colon E^{\otimes2m}\rightarrow\mathbb{Q}{}_{\ell}(-dm)
\]
is invariant under $\SP(\psi)$, and $\Hom(E^{\otimes2m},\mathbb{Q}{}_{\ell
})^{\SP(\psi)}$ has a basis of maps $\{f_{1},\ldots,f_{M}\}$obtained from this
one by permutations. The map%
\begin{equation}
a\mapsto(f_{1}(a),\ldots,f_{M}(a))\colon E^{\otimes2m}\rightarrow\mathbb{Q}%
{}_{\ell}(-dm)^{M} \label{e39}%
\end{equation}
induces an isomorphism%
\[
(E^{\otimes2m})_{\SP(\psi)}\rightarrow\mathbb{Q}{}_{\ell}(-dm)^{M}.
\]
The Frobenius element $\pi\in\Gal(k/k_{0})$ acts on $\mathbb{Q}{}_{\ell}(-md)$
as $q^{md}$, and so it acts on $E_{\pi_{1}(U_{0})}^{\otimes2m}(-1)$ as
$q^{md+1}.$ We can now apply Landau's theorem to deduce that%
\[
\left\vert \alpha^{2m}\right\vert \leq q^{md+1}%
\]
for every eigenvalue $\alpha$ of $\pi_{o}$ acting on $\mathcal{E}{}_{o}%
=H^{d}(V_{0},\mathbb{Q}{}_{\ell}).$ On taking the $2m$th root and letting
$m\rightarrow\infty$, we find that $|\alpha|\leq q^{d/2}$. As $q^{d}/\alpha$
is also an eigenvalue of $\pi_{o}$ on $H^{d}(V_{0},\mathbb{Q}{}_{\ell})$, we
deduce that $|\alpha|=q^{d/2}$.

\subsection{Proof of the Riemann hypothesis for nonsingular projective
varieties}

We shall prove the following statement by induction on the dimension of
$V_{0}$. According to the discussion in (\ref{r26}), it completes the proof of
Conjectures W1--W4.

\begin{theorem}
\label{r81}Let $V_{0}$ be an nonsingular projective variety over $\mathbb{F}%
{}_{q}$. Then the eigenvalues of $\pi$ acting on $H^{r}(V,\mathbb{Q}{}_{\ell
})$ are algebraic numbers, all of whose complex conjugates have absolute value
$q^{r/2}$.
\end{theorem}

\subsubsection{The main lemma (restricted form)}

The following is abstracted from the proof of Riemann hypothesis for
hypersurfaces of odd degree.

\begin{mlemma}
\label{r96}Let $\mathcal{E}_{0}$ be a local system of $\mathbb{Q}_{\ell}%
$-vector spaces on $U_{0}$, and let $E$ be the corresponding $\pi_{1}(U_{0}%
)$-module. Let $d$ be an integer. Assume:

\renewcommand{\theenumi}{{\roman{enumi}}}

\begin{enumerate}
\item \ (Rationality.) For all closed points $u\in U_{0}$, the characteristic
polynomial of $\pi_{u}$ acting on $\mathcal{E}{}_{u}$ has rational coefficients.

\item \ There exists a nondegenerate $\pi_{1}(U_{0})$-invariant skew-symmetric
form
\[
\psi\colon E\times E\rightarrow\mathbb{Q}_{\ell}(-d).
\]

\item \ (Big geometric monodromy.) The image of $\pi_{1}(U)$ in $Sp(E,\psi)$
is open.
\end{enumerate}

\renewcommand{\theenumi}{{\alph{enumi}}}

Then:

\begin{enumerate}
\item \ For all closed points $u\in U_{_{0}}$, the eigenvalues of $\pi_{u}$
acting on $\mathcal{E}_{u}$ have absolute value $(q^{\deg u})^{d/2}$.

\item \ The characteristic polynomial of $\pi$ acting on $H_{c}^{1}%
(U,\mathcal{E})$ is rational, and its eigenvalues all have absolute value
$\leq q^{d/2+1}$.
\end{enumerate}
\end{mlemma}

\begin{proof}
As before, $E_{\pi_{1}(U)}^{\otimes2m}$ is isomorphic to a direct sum of
copies of $\mathbb{Q}{}_{\ell}(-md)$, from which (a) follows. Now (b) follows
from (\ref{e38}).
\end{proof}

\subsubsection{A reduction}

\begin{proposition}
\label{r95}Assume that for all nonsingular varieties $V_{0}$ of even dimension
over $\mathbb{F}{}_{q}$, every eigenvalue $\alpha$ of $\pi$ on $H^{\dim
V}(V,\mathbb{Q}{}_{\ell})$ is an algebraic number such that%
\[
q^{\frac{\dim V}{2}-\frac{1}{2}}<|\alpha^{\prime}|<q^{\frac{\dim V}{2}%
+\frac{1}{2}}%
\]
for all complex conjugates $\alpha^{\prime}$ of $\alpha$. Then Theorem
\ref{r81} holds for all nonsingular projective varieties over $\mathbb{F}%
{}_{q}$.
\end{proposition}

\begin{proof}
Let $V_{0}$ be a smooth projective variety of dimension $d$ (not necessarily
even) over $\mathbb{F}{}_{q}$, and let $\alpha$ be an eigenvalue of $\pi$ on
$H^{d}(V,\mathbb{Q}{}_{\ell})$. The K\"{u}nneth formula shows that $\alpha
^{m}$ occurs among the eigenvalues of $\pi$ acting on $H^{dm}(V^{m}%
,\mathbb{Q}{}_{\ell})$ for all $m\in\mathbb{N}{}$. The hypothesis in the
proposition applied to an even $m$ shows that%
\[
q^{\frac{md}{2}-\frac{1}{2}}<|\alpha^{\prime}|^{m}<q^{\frac{md}{2}+\frac{1}%
{2}}%
\]
for all conjugates $\alpha^{\prime}$ of $\alpha$. On taking the $m$th root and
letting $m$ tend to infinity over the even integers, we find that%
\[
|\alpha^{\prime}|=q^{\frac{d}{2}}\text{.}%
\]

Now let $\alpha$ be an eigenvalue of $\pi$ acting on $H^{r}(V,\mathbb{Q}%
{}_{\ell})$. If $r>d$, then $\alpha^{d}$ is an eigenvalue of $\pi$ acting on
\[
H^{r}(V,\mathbb{Q}{}_{\ell})^{\otimes d}\otimes H^{0}(V,\mathbb{Q}{}_{\ell
})^{\otimes r-d}\subset H^{rd}(V^{r},\mathbb{Q}{}_{\ell})
\]
because $\pi$ acts as $1$ on $H^{0}(V,\mathbb{Q}{}_{\ell})$. Therefore
$\alpha^{d}$ is algebraic and $|\alpha^{d}|=q^{\frac{rd}{2}}$, and similarly
for its conjugates. The case $r<d$ can be treated similarly using that $\pi$
acts as $q^{d}$ on $H^{2d}(V,\mathbb{Q}{}_{\ell})$, or by using Poincar\'{e} duality.
\end{proof}

\subsubsection{Completion of the proof}

To deduce Theorem \ref{r81} from (\ref{r96}), Deligne uses the theory of
Lefschetz pencils and their cohomology (specifically vanishing cycle theory
and the Picard-Lefschetz formula). In the complex setting, this theory was
introduced by Picard for surfaces and by Lefschetz (1924) for higher
dimensional varieties. It was transferred to the abstract setting by Deligne
and Katz in SGA 7 (1967--1969). Deligne had earlier used these techniques to
prove the following weaker result:

\begin{quote}
let $V_{0}$ be a smooth projective variety over a finite field $k_{0}$ of odd
characeristic that lifts, together with a polarization, to characteristic
zero. Then the polynomials $\det(1-\pi T\mid H^{r}(V,\mathbb{Q}{}_{\ell}))$
have integer coefficients independent of $\ell$ (Verdier 1972).
\end{quote}

\noindent Thus, Deligne was already an expert in the application of these
methods to zeta functions.

Let $V$ be a nonsingular projective variety $V$ of dimension $d\geq2$, and
embed $V$ into a projective space $\mathbb{P}{}^{m}$. The hyperplanes
$H\colon\sum a_{i}T_{i}=0$ in $\mathbb{P}{}^{m}$ form a projective space,
called the \textit{dual projective space }$\check{\mathbb{P}}^{m}$. A
\textit{pencil} of hyperplanes is a line $D=\{\alpha H_{0}+\beta H_{\infty
}\mid(\alpha\colon\beta)\in\mathbb{P}{}^{1}(k)\}$ in $\check{\mathbb{P}}^{m}$.
The \textit{axis} of the pencil is%
\[
A=H_{0}\cap H_{\infty}=\bigcap\nolimits_{t\in D}H_{t}.
\]
Such a pencil is said to be a \textit{Lefschetz pencil} for $V$ if (a) the
axis $A$ of the pencil cuts $V$ transversally; (b) the hyperplane sections
$V_{t}\overset{\textup{{\tiny def}}}{=}V\cap D_{t}$ are nonsingular for all
$t$ in some open dense subset $U$ of $D$; (c) for all $t\notin U$, the
hyperplane section $V_{t}$ has only a single singularity and that singularity
is an ordinary double point. Given such a Lefschetz pencil, we can blow $V$ up
along the $A\cap V$ to obtain a proper flat map%
\[
f\colon V^{\ast}\rightarrow D
\]
such that the fibre over $t$ in $D$ is $V_{t}=V\cap H_{t}$.

We now prove Theorem \ref{r81} (and hence the Riemann hypothesis). Let $V_{0}$
be a nonsingular projective variety of even dimension $d\geq2$ over a finite
field $k_{0}$. Embed $V_{0}$ in a projective space. After possibly replacing
the embedding with its square and $k_{0}$ with a finite extension, we may
suppose that there exists a Lefschetz pencil and hence a proper map $f\colon
V_{0}^{\ast}\rightarrow D=\mathbb{P}{}^{1}$. Recall (\ref{r95}) that we have
to prove that%
\[
q^{d/2-1/2}<\left\vert \alpha\right\vert <q^{d/2+1/2}%
\]
for all eigenvalues $\alpha$ on $H^{d}(V,\mathbb{Q}{}_{\ell})$. It suffices to
do this with $V_{0}$ replaced by $V_{0}^{\ast}$. From the Leray spectral
sequence,
\[
H^{r}(\mathbb{P}{}^{1},R^{s}f_{\ast}\mathbb{Q}{}_{\ell})\implies
H^{r+s}(V,\mathbb{Q}{}_{\ell}),
\]
we see that it suffices to prove a similar statement for each of the groups%
\[
H^{2}(\mathbb{P}{}^{1},R^{d-2}f_{\ast}\mathbb{Q}{}_{\ell}),\quad
H^{1}(\mathbb{P}{}^{1},R^{d-1}f_{\ast}\mathbb{Q}{}_{\ell}),\quad
H^{0}(\mathbb{P}{}^{1},R^{d}f_{\ast}\mathbb{Q}{}_{\ell}).
\]
The most difficult group is the middle one. Here Deligne applies the Main
Lemma to a certain quotient subquotient $\mathcal{E}{}$ of $R^{d-1}f_{\ast
}\mathbb{Q}{}_{\ell}$. The hypothesis (i) of the Main Lemma is true from the
induction hypothesis; the skew-symmetric form required for (ii) is provided by
an intersection form; finally, for (iii), Deligne was able to appeal to a
theorem of Kazhdan and Margulis. Now the Main Lemma shows that $|\alpha|\leq
q^{d/2+1/2}$ for an eigenvalue of $\pi$ on $H^{1}(\mathbb{P}{}^{1}%
,\mathcal{E}){}$, and a duality argument shows that $q^{d-1/2}\leq|\alpha|$.

\begin{note}
Deligne's proof of the Riemann hypothesis over finite fields is well explained
in his original paper (Deligne 1974) and elsewhere. There is also a purely
$p$-adic proof (Kedlaya 2006).
\end{note}

\subsection{Beyond the Riemann hypothesis over finite fields}

In a seminar at IHES, Nov. 1973 to Feb. 1974, Deligne improved his results on
zeta functions. In particular he proved stronger forms of the Main Lemma.
These results were published in Deligne 1980, which has become known as Weil
II. Even beyond his earlier results on the Riemann hypothesis, these results
have found a vast array of applications, which I shall not attempt to
summarize (see the various writings of Katz, especially Katz 2001, and the
book Kiehl and Weissauer 2001).

\section{The Hasse-Weil zeta function}

\subsection{The Hasse-Weil conjecture}

Let $V$ be a nonsingular projective variety over a number field $K$. For
almost all prime ideals $\mathfrak{p}{}$ in $\mathcal{O}{}_{K}$ (called good),
$V$ defines a nonsingular projective variety $V(\mathfrak{p}{})$ over the
residue field $\kappa(\mathfrak{p}{})=\mathcal{O}{}_{K}/\mathfrak{p}{}$.
Define the zeta function of $V$ to be%
\[
\zeta(V,s)=\prod\nolimits_{\mathfrak{p}{}\text{ good}}\zeta(V(\mathfrak{p}%
),s).
\]
For example, if $V$ is a point over $\mathbb{Q}{}$, then $\zeta(X,s)$ is the
original Riemann zeta function $\prod\nolimits_{p}\frac{1}{1-p^{-s}}$. Using
the Riemann hypothesis for the varieties $V(\mathfrak{p}{})$, one sees that
the that the product converges for $\Re(s)>\dim(V)+1$. It is possible to
define factors corresponding to the bad primes and the infinite primes. The
completed zeta function is then conjectured to extend analytically to a
meromorphic function on the whole complex plane and satisfy a functional
equation. The function $\zeta(V,s)$ is usually called the \textit{Hasse-Weil
zeta function} of $V$, and the conjecture the \textit{Hasse-Weil conjecture}.

More precisely, let $r$ be a natural number. Serre (1970) defined:

\begin{enumerate}
\item polynomials $P_{r}(V(\mathfrak{p}{}),T)\in\mathbb{Z}{}[T]$ for each good
prime (namely, those occurring in the zeta function for $V(\mathfrak{p}{})$);

\item polynomials $Q_{\mathfrak{p}{}}(T)\in\mathbb{Z}{}[T]$ for each bad prime
$\mathfrak{p;}$

\item gamma factors $\Gamma_{v}$ for each infinite prime $v$ (depending on the
Hodge numbers of $V\otimes_{k}k_{v}$);

\item ${}$ a rational number $A>0$.
\end{enumerate}

\noindent Set%
\[
\xi(s)=A^{s/2}\cdot\prod_{\mathfrak{p}{}\text{ good}}\frac{1}{P_{r}%
(V(\mathfrak{p}{}),N\mathfrak{p}{}^{-1})}\cdot\prod_{\mathfrak{p}{}\text{
bad}}\frac{1}{Q_{\mathfrak{p}{}}(N\mathfrak{p}{}^{-1})}\cdot\prod_{v\text{
infinite}}\Gamma_{v}(s).
\]
The Hasse-Weil conjecture says that $\zeta_{r}(s)$ extends to a meromorphic
function on the whole complex plane and satisfies a functional equation%
\[
\xi(s)=w\xi(r+1-s),\quad w=\pm1.
\]

\subsection{History}

According to Weil's recollections (\OE uvres, II, p.529),\footnote{Peu avant
la guerre, si mes souvenirs sont exacts, G. de Rham me raconta qu'un de ses
\'{e}tudiants de Gen\`{e}ve, Pierre Humbert, \'{e}tait all\'{e} \`{a}
G\"{o}ttingen avec l'intention d'y travailler sous la direction de Hasse, et
que celui-ci lui avait propos\'{e} un probl\'{e}me sur lequel de Rham
d\'{e}sirait mon avis. Une courbe elliptique $C$ \'{e}tant donn\'{e}e sur le
corps des rationnels, il s'agissait principalement, il me semble,
d'\'{e}tudier le produit infini des fonctions z\^{e}ta des courbes $C_{p}$
obtenues en r\'{e}duisant $C$ modulo $p$ pour tout nombre premier $p$ pour
lequel $C_{p}$ est de genre $1$; plus pr\'{e}cis\'{e}ment, il fallait
rechercher si ce produit poss\`{e}de un prolongement analytique et une
\'{e}quation fonctionnelle. J'ignore si Pierre Humbert, ou bien Hasse, avaient
examin\'{e} aucun cas particulier. En tout cas, d'apr\`{e}s de Rham, Pierre
Humbert se sentait d\'{e}courag\'{e} et craignait de perdre son temps et sa
peine.} Hasse defined the Hasse-Weil zeta function for an elliptic curve
$\mathbb{Q}{}$, and set the Hasse-Weil conjecture in this case as a thesis
problem! Initially, Weil was sceptical of the conjecture, but he proved it for
curves of the form $Y^{m}=aX^{n}+b$ over number fields by expressing their
zeta functions in terms of Hecke $L$-functions.\footnote{Weil also saw that
the analogous conjecture over global function fields can sometimes be deduced
from the Weil conjectures.} In particular, Weil showed that the zeta functions
of the elliptic curves $Y^{2}=aX^{3}+b$ and $Y^{2}=aX^{4}+b$ can be expressed
in terms of Hecke $L$-functions, and he suggested that the same should be true
for all elliptic curves with complex multiplication. This was proved by
Deuring in a \textquotedblleft beautiful series\textquotedblright\ of papers.

Deuring's result was extended to all abelian varieties with complex
multiplication by Shimura and Taniyama and Weil.

In a different direction, Eichler and Shimura proved the Hasse-Weil conjecture
for elliptic modular curves by identifying their zeta functions with the
Mellin transforms of modular forms.

Wiles et al.\ proved Hasse's \textquotedblleft thesis
problem\textquotedblright\ as part of their work on Fermat's Last Theorem (for
a popular account of this work, see Darmon 1999).

\subsection{Automorphic $L$-functions}

The only hope one has of proving the Hasse-Weil conjecture for a variety is by
identifying its zeta function with a known function, which has (or is
conjectured to have) a meromorphic continuation and functional equation. In a
letter to Weil in 1967, Langlands defined a vast collection of $L$-functions,
now called automorphic (Langlands 1970). He conjectures that \textit{all}
$L$-functions arising from algebraic varieties over number fields (i.e., all
motivic $L$-functions) can be expressed as products of automorphic
$L$-functions.\footnote{Or perhaps there should even be an identification of
the tannakian category defined by motives with a subcategory of a category
defined by automorphic representations. (\textquotedblleft Wenn man sich die
Langlands-Korrespondenz als eine Identifikation einer durch Motive definierten
Tannaka-Kategorie mit einer Unterkategorie einer durch automorphe
Darstellungen definierten Kategorie vorstellt\ldots\textquotedblright%
\ Langlands 2015, p.201.)} The above examples and results on the zeta
functions of Shimura varieties have made Langlands very optimistic:

\begin{quote}
With the help of Shimura varieties mathematicians certainly have, for me,
answered one main question: is it possible to express all motivic
$L$-functions as products of automorphic $L$-functions? The answer is now
beyond any doubt, \textquotedblleft Yes!\textquotedblright. Although no
general proof is available, this response to the available examples and
partial proofs is fully justified.\footnote{Mit Hilfe der
Shimuravariet\"{a}ten haben die Mathematiker gewi\ss \ eine, f\"{u}r mich,
Hauptfrage beantwortet: wird es m\"{o}glich sein, alle motivischen
$L$-Funktionen als Produkte von automorphen $L$-Funktionen auszudr\"{u}cken?
Die Antwort ist jetzt zweifellos, \textquotedblleft Ja!\textquotedblright.
Obwohl kein allgemeiner Beweis vorhanden ist, ist diese Antwort von den
vorhanden Beispielen und Belegst\"{u}cken her v\"{o}llig berechtigt.}
(Langlands 2015, p.205.)
\end{quote}

\subsection{The functoriality principle and conjectures on algebraic cycles}

As noted earlier, Langlands attaches an $L$-function $L(s,\sigma,\pi)$ to an
automorphic representation $\pi$ of a reductive algebraic group $G$ over
$\mathbb{Q}{}$ and a representation $\sigma$ of the corresponding $L$-group
$^{L}G$. The functoriality principle asserts that a homomorphism of $L$-groups
$^{L}H\rightarrow{}^{L}G$ entails a strong relationship between the
automorphic representations of $H$ and $G$. This remarkable conjecture
includes the Artin conjecture, the existence of nonabelian class field
theories, and many other fundamental statements as special cases. And perhaps
even more.

Langlands doesn't share the pessimism noted earlier concerning the standard
conjectures and the conjectures of Hodge and Tate. He has suggested that the
outstanding conjectures in the Langlands program are more closely entwined
with the outstanding conjectures on algebraic cycles than is usually
recognized, and that a proof of the first may lead to a proof of the second.

\begin{quote}
I believe that it is necessary first to prove functoriality and then
afterwards, with the help of the knowledge and tools obtained, to develop a
theory of correspondences and simultaneously of motives over $\mathbb{Q}{}$
and other global fields, as well as $\mathbb{C}{}$.\footnote{Ich habe schon an
verschiedenen Stellen behauptet, da\ss es meines Erachtens n\"{o}tig ist, erst
die Funktorialit\"{a}t zu beweisen und dann nachher, mit Hilfe der damit
erschlossenen Kenntnisse und zur Verf\"{u}gung gestellten Hilfsmittel, eine
Theorie der Korrespondenz und, gleichzeitig, der Motive, \"{u}ber
$\mathbb{Q}{}$ und anderen globalen K\"{o}rper, sowie \"{u}ber $\mathbb{C}{}$
zu begr\"{u}nden.} (Langlands 2015, p.205).
\end{quote}

In support of this statement, he mentions two examples (Langlands 2007, 2015).
The first, Deligne's proof of the Riemann hypothesis over finite fields, has
already been discussed in some detail. I now discuss the second example.

\subsubsection{The problem}

Let $A$ be a polarized abelian variety with complex multiplication over
$\mathbb{Q}{}^{\mathrm{al}}$. The theory of Shimura and Taniyama describes the
action on $A$ and its torsion points of the Galois group of $\mathbb{Q}%
{}^{\mathrm{al}}$ over the reflex field of $A$. In the case that $A$ is an
elliptic curve, the reflex field is a quadratic imaginary field; since one
knows how complex conjugation acts, in this case we have a description of how
the full Galois group $\Gal(\mathbb{Q}{}^{\mathrm{al}}/\mathbb{Q}{})$ acts. Is
it possible to find such a description in higher dimensions?

Shimura studied this question, but concluded rather pessimistically that
\textquotedblleft In the higher-dimensional case, however, no such general
answer seems possible.\textquotedblright\ However, Grothendieck's theory of
motives suggests the framework for an answer. The Hodge conjecture implies the
existence of tannakian category of CM-motives over $\mathbb{Q}$, whose motivic
Galois group is an extension%
\[
1\rightarrow S\rightarrow T\rightarrow\Gal(\mathbb{Q}{}^{\mathrm{al}%
}/\mathbb{Q}{})\rightarrow1
\]
of $\Gal(\mathbb{Q}{}^{\mathrm{al}}/\mathbb{Q}{})$ (regarded as a pro-constant
group scheme) by the Serre group $S$ (a certain pro-torus); \'{e}tale
cohomology defines a section to $T\rightarrow\Gal(\mathbb{Q}{}^{\mathrm{al}%
}/\mathbb{Q}{})$ over the finite ad\`{e}les. Let us call this entire system
the \textit{motivic Taniyama group}. This system describes how
$\Gal(\mathbb{Q}{}^{\mathrm{al}}/\mathbb{Q}{})$ acts on CM abelian varieties
over $\mathbb{Q}{}$ and their torsion points, and so the problem now becomes
that of giving an explicit description of the motivic Taniyama group.

\subsubsection{The solution}

Shimura varieties play an important role in the work of Langlands, both as a
test of his ideas and for their applications. For example, Langlands writes (2007):

\begin{quote}
Endoscopy, a feature of nonabelian harmonic analysis on reductive groups over
local or global fields, arose implicitly in a number of contexts\ldots\ It
arose for me in the context of the trace formula and Shimura varieties.
\end{quote}

\noindent Langlands formulated a number of conjectures concerning the zeta
functions of Shimura varieties.

One problem that arose in his work is the following. Let $S(G,X)$ denote the
Shimura variety over $\mathbb{C}{}$ defined by a Shimura datum $(G,X)$.
According to the Shimura conjecture $S(G,X)$ has a canonical model
$S(G,X)_{E}$ over a certain algebraic subfield $E$ of $\mathbb{C}{}$. The
$\Gamma$-factor of the zeta function of $S(G,X)_{E}$ at a complex prime
$v\colon E\hookrightarrow\mathbb{C}{}$ depends on the Hodge theory of the
complex variety $v(S(G,X)_{E}).$ For the given embedding of $E$ in
$\mathbb{C}{}$, there is no problem because $v(S(G,X)_{E})=S(G,X)$, but for a
different embedding, what is $v(S(G,X)_{E})$? This question leads to the
following problem:

\begin{quote}
Let $\sigma$ be an automorphism of $\mathbb{C}{}$ (as an abstract field); how
can we realize $\sigma S(G,X)$ as the Shimura variety attached to another
(explicitly defined) Shimura datum $(G^{\prime},X^{\prime})$?
\end{quote}

In his Corvallis talk (1979), Langlands states a conjectural solution to this
problem. In particular, he constructs a \textquotedblleft
one-cocycle\textquotedblright\ which explains how to twist $(G,X)$ to obtain
$(G^{\prime},X^{\prime})$. When he explained his construction to Deligne, the
latter recognized that the one-cocycle also gives a construction of a
conjectural Taniyama group. The descriptions of the action of $\Gal(\mathbb{Q}%
^{\mathrm{al}}/\mathbb{Q}{})$ on CM abelian varieties and their torsion points
given by the motivic Taniyama group and by Langlands's conjectural Taniyama
group are both consistent with the results of Shimura and Taniyama. Deligne
proved that this property characterizes the groups, and so the two are equal.
This solves the problem posed above.

Concerning Langlands's conjugacy conjecture itself, this was proved in the
following way. For those Shimura varieties with the property that each
connected component can be described by the moduli of abelian varieties,
Shimura's conjecture was proved in many cases by Shimura and his students and
in general by Deligne. To obtain a proof for a general Shimura variety,
Piatetski-Shapiro suggested embedding the Shimura variety in a larger Shimura
variety that contains many Shimura subvarieties of type $A_{1}$. After Borovoi
had unsuccessfully tried to use Piatetski-Shapiro's idea to prove Shimura's
conjecture directly, the author used it to prove Langlands's conjugation
conjecture, which has Shimura's conjecture as a consequence. No direct proof
of Shimura's conjecture is known.

\subsection{Epilogue}

Stepanov (1969) introduced a new elementary approach for proving the Riemann
hypothesis for curves, which requires only the Riemann-Roch theorem for
curves. This approach was completed and simplified by Bombieri (1974) (see
also Schmidt 1973).

How would our story have been changed if this proof had been found in the 1930s?

\subsubsection{Acknowledgement}

I thank F. Oort and J-P. Serre for their comments on the first draft of this
article, and Serre for contributing the letter of Weil. It would not have been
possible to write this article without the University of Michigan Science
Library, which provided me with a copy of the 1953 thesis of Frank Quigley,
and where I was able to find such works as Castelnuovo's \textit{Memoire
Scelte \ldots} and Severi's \textit{Trattato ...} available on open shelves.

\begin{small}

\section*{Bibliography}

``Colmez and Serre 2001'' is cited as ``\textit{Grothendieck-Serre Correspondence}''.

\parindent 0pt \everypar{\hangindent1.5em\hangafter1}

Abdulali, S. 1994. Algebraic cycles in families of abelian varieties. Canad.
J. Math. 46 (1994), no. 6, 1121--1134.

Andr\'{e}, Y. 1996. Pour une th\'{e}orie inconditionnelle des motifs. Inst.
Hautes \'{E}tudes Sci. Publ. Math. No. 83 (1996), 5--49.

Andr\'{e}, Y. 2004. Une introduction aux motifs (motifs purs, motifs mixtes,
p\'{e}riodes). Panoramas et Synth\`{e}ses, 17. Soci\'{e}t\'{e}
Math\'{e}matique de France, Paris, 2004.

Artin, M. 1962. Grothendieck topologies, Notes of a seminar, Spring 1962
(Harvard University, Department of Mathematics).

Artin, M. 1964. Th\'{e}or\`{e}me de Weil sur la construction d'un groupe \`{a}
partir d'une loi rationelle. Sch\'{e}mas en Groupes (S\'{e}m.
G\'{e}om\'{e}trie Alg\'{e}brique, Inst. Hautes \'{E}tudes Sci., 1963/64) Fasc.
5b, Expose 18, 22 pp. Inst. Hautes \'{E}tudes Sci., Paris.

Artin, M. 1986. N\'{e}ron models. Arithmetic geometry (Storrs, Conn., 1984),
213--230, Springer, New York, 1986.

Berthelot, P. 1974. Cohomologie cristalline des sch\'{e}mas de
caract\'{e}ristique $p>0$. Lecture Notes in Mathematics, Vol. 407.
Springer-Verlag, Berlin-New York, 1974.

Bloch, S. 1977. Algebraic K-theory and crystalline cohomology. Inst. Hautes
\'{E}tudes Sci. Publ. Math. No. 47, (1977), 187--268.

Bombieri, E. 1974. Counting points on curves over finite fields (d'apr\`{e}s
S. A. Stepanov). S\'{e}minaire Bourbaki, 25\`{e}me ann\'{e}e (1972/1973), Exp.
No. 430, pp. 234--241. Lecture Notes in Math., Vol. 383, Springer, Berlin, 1974.

Bosch, S.; L\"{u}tkebohmert, W.; Raynaud, M. 1990. N\'{e}ron models.
Ergebnisse der Mathematik und ihrer Grenzgebiete (3) 21. Springer-Verlag,
Berlin, 1990.

Brigaglia, A; Ciliberto, C; Pedrini, C. 2004. The Italian school of algebraic
geometry and Abel's legacy. The legacy of Niels Henrik Abel, 295--347,
Springer, Berlin, 2004.

Bronowski, J. 1938. Curves whose grade is not positive in the theory of the
base. J. Lond. Math. Soc. 13, 86-90 (1938).

Castelnuovo, G. 1906. Sulle serie algebriche di gruppi di punti appartenenti
ad una curva algebrica. Rom. Acc. L. Rend. (5) 15, No.1, 337-344 (1906).

Chow, W-L. 1954. The Jacobian variety of an algebraic curve. Amer. J. Math.
76, (1954). 453--476.

Clemens, H. 1983. Homological equivalence, modulo algebraic equivalence, is
not finitely generated. Inst. Hautes \'{E}tudes Sci. Publ. Math. No. 58
(1983), 19--38 (1984).

Colmez, P.; Serre, J-P. (Eds) 2001. Correspondance Grothendieck-Serre. Edited
by Pierre Colmez and Jean-Pierre Serre. Documents Math\'{e}matiques (Paris),
2. Soci\'{e}t\'{e} Math\'{e}matique de France, Paris, 2001. (A bilingual
edition is available.)

Darmon, H. 1999. A proof of the full Shimura-Taniyama-Weil conjecture is
announced. Notices Amer. Math. Soc. 46 (1999), no. 11, 1397--1401.

Davenport, H.; Hasse, H. 1935. Die Nullstellen der Kongruenzzetafunktionen in
gewissen zyklischen F\"{a}llen. J. Reine Angew. Math. 172 (1935), 151--182.

Deligne, P. 1972. La conjecture de Weil pour les surfaces $K3$. Invent. Math.
15 (1972), 206--226.

Deligne, P. 1974. La conjecture de Weil. I. Inst. Hautes \'{E}tudes Sci. Publ.
Math. No. 43 (1974), 273--307.

Deligne, P. 1980. La conjecture de Weil. II. Inst. Hautes \'{E}tudes Sci.
Publ. Math. No. 52 (1980), 137--252.

Deligne, P. 1982. Hodge cycles on abelian varieties (notes by J.S. Milne). In:
Hodge Cycles, Motives, and Shimura Varieties, Lecture Notes in Math. 900,
Springer-Verlag, 1982, pp. 9--100.

Deligne, P. 1990. Cat\'{e}gories tannakiennes. The Grothendieck Festschrift,
Vol. II, 111--195, Progr. Math., 87, Birkh\"{a}user Boston, Boston, MA, 1990

Deligne, P.; Milne, J., 1982. Tannakian categories. In: Hodge Cycles, Motives,
and Shimura Varieties, Lecture Notes in Math. 900, Springer-Verlag, 1982, pp. 101--228.

Deuring, M. 1941. Die Typen der Multiplikatorenringe elliptischer
Funktionenk\"orper. Abh. Math. Sem. Hansischen Univ. 14, (1941). 197--272.

Dieudonn\'{e}, J. 1974. Cours de g\'{e}ometrie alg\'{e}briques, 1, Presses
Universitaires de France, 1974.

Dwork, B. 1960. On the rationality of the zeta function of an algebraic
variety. Amer. J. Math. 82 (1960) 631--648.

Fulton, W. 1984. Intersection theory. Ergebnisse der Mathematik und ihrer
Grenzgebiete (3), 2. Springer-Verlag, Berlin, 1984.

Griffiths, P. 1969. On the periods of certain rational integrals. II. Ann. of
Math. (2) 90 (1969) 496--541.

Grothendieck, A. 1957. Sur quelques points d'alg\`{e}bre homologique.
T\^{o}hoku Math. J. (2) 9 (1957) 119--221.

Grothendieck, A. 1958a. Sur une note de Mattuck-Tate. J. Reine Angew. Math.
200 (1958) 208--215.

Grothendieck, A. 1958b. The cohomology theory of abstract algebraic varieties.
1960 Proc. Internat. Congress Math. (Edinburgh, 1958) pp. 103--118 Cambridge
Univ. Press, New York.

Grothendieck, A. 1964. Formule de Lefschetz et rationalit\'{e} des fonctions
$L$. S\'{e}minaire Bourbaki, Vol. 9, Exp. No. 279, 41--55, D\'{e}cembre 1964.

Grothendieck, A. 1966. On the de Rham cohomology of algebraic varieties. Inst.
Hautes \'{E}tudes Sci. Publ. Math. No. 29 (1966) 95--103.

Grothendieck, A. 1968. Crystals and the de Rham cohomology of schemes. Dix
Expos\'{e}s sur la Cohomologie des Sch\'{e}mas pp. 306--358 North-Holland,
Amsterdam; Masson, Paris, 1968.

Grothendieck, A. 1969. Standard conjectures on algebraic cycles. Algebraic
Geometry (Internat. Colloq., Tata Inst. Fund. Res., Bombay, 1968) pp. 193--199
Oxford Univ. Press, London, 1969.

Hasse, H. 1936. Zur Theorie der abstrakten elliptischen Funktionenk\"{o}rper.
I, II, III, J. Reine Angew. Math. 175, (1936), 55--62; ibid. 69--88; ibid. pp.193--208.

Hodge, W. 1937. Note on the theory of the base for curves on an algebraic
surface. J. Lond. Math. Soc. 12, (1937), 58-63.

Honda, T. 1968. Isogeny classes of abelian varieties over finite fields. J.
Math. Soc. Japan 20 1968 83--95.

Hurwitz, A. 1887. Ueber algebraische Correspondenzen und das verallgemeinerte
Correspondenzprincip. Math. Ann. 28 (1887), no. 4, 561--585.

Igusa, Jun-ichi. 1949. On the theory of algebraic correspondances and its
application to the Riemann hypothesis in function-fields. J. Math. Soc. Japan
1, (1949). 147--197.

Igusa, Jun-ichi. 1955. On some problems in abstract algebraic geometry. Proc.
Nat. Acad. Sci. U. S. A. 41 (1955), 964--967.

Illusie, L. 1983. Finiteness, duality, and K\"{u}nneth theorems in the
cohomology of the de Rham-Witt complex. Algebraic geometry (Tokyo/Kyoto,
1982), 20--72, Lecture Notes in Math., 1016, Springer, Berlin, 1983.

Illusie, L. 2010. Reminiscences of Grothendieck and his school. Notices Amer.
Math. Soc. 57 (2010), no. 9, 1106--1115.

Illusie, L. 2014. Grothendieck et la cohomologie \'{e}tale. Alexandre
Grothendieck: a mathematical portrait, 175--192, Int. Press, Somerville, MA, 2014.

Jannsen, U. 1992. Motives, numerical equivalence, and semi-simplicity. Invent.
Math. 107 (1992), no. 3, 447--452.

Jannsen, U. 2007. On finite-dimensional motives and Murre's conjecture.
Algebraic cycles and motives. Vol. 2, 112--142, London Math. Soc. Lecture Note
Ser., 344, Cambridge Univ. Press, Cambridge, 2007.

Kani, E. 1984. On Castelnuovo's equivalence defect. J. Reine Angew. Math. 352
(1984), 24--70.

Katz, N. 1976. An overview of Deligne's proof of the Riemann hypothesis for
varieties over finite fields. Mathematical developments arising from Hilbert
problems (Proc. Sympos. Pure Math., Vol. XXVIII, Northern Illinois Univ., De
Kalb, Ill., 1974), pp. 275--305. Amer. Math. Soc., Providence, R.I., 1976.

Katz, N. 2001. $L$-functions and monodromy: four lectures on Weil II. Adv.
Math. 160 (2001), no. 1, 81--132.

Kedlaya, K. 2006. Fourier transforms and $p$-adic `Weil II'. Compos. Math. 142
(2006), no. 6, 1426--1450.

Kiehl, R.; Weissauer, R, 2001. Weil conjectures, perverse sheaves and l'adic
(sic) Fourier transform. Springer-Verlag, Berlin, 2001.

Kleiman, S. 1968. Algebraic cycles and the Weil conjectures. Dix expos\'{e}s
sur la cohomologie des sch\'{e}mas, pp. 359--386. North-Holland, Amsterdam;
Masson, Paris, 1968.

Kleiman, S. 1994. The standard conjectures. Motives (Seattle, WA, 1991),
3--20, Proc. Sympos. Pure Math., 55, Part 1, Amer. Math. Soc., Providence, RI, 1994.

Kleiman, S. 2005. The Picard scheme. Fundamental algebraic geometry, 235--321,
Math. Surveys Monogr., 123, Amer. Math. Soc., Providence, RI, 2005.

Lang, S. 1957. Divisors and endomorphisms on an abelian variety. Amer. J.
Math. 79 (1957) 761--777.

Lang, S. 1959. Abelian varieties. Interscience Tracts in Pure and Applied
Mathematics. No. 7 Interscience Publishers, Inc., New York; Interscience
Publishers Ltd., London 1959

Langlands, R. 1970. Problems in the theory of automorphic forms. Lectures in
modern analysis and applications, III, pp. 18--61. Lecture Notes in Math.,
Vol. 170, Springer, Berlin, 1970.

Langlands, R. 1979. Automorphic representations, Shimura varieties, and
motives. Ein M\"{a}rchen. Automorphic forms, representations and $L$-functions
(Proc. Sympos. Pure Math., Oregon State Univ., Corvallis, Ore., 1977), Part 2,
pp. 205--246, Proc. Sympos. Pure Math., XXXIII, Amer. Math. Soc., Providence,
R.I., 1979.

Langlands, R. 2007. Reflexions on receiving the Shaw Prize, September 2007.

Langlands, R. 2015, Funktorialit\"{a}t in der Theorie der automorphen Formen:
Ihre Endeckung und ihre Ziele. In: Emil Artin and Beyond --- Class Field
Theory and $L$-Functions (Dumbaugh, D., and Schwermer, J.) E.M.S., 2015.

Lefschetz, S. 1924. L'Analysis situs et la g\'{e}om\'{e}trie alg\'{e}brique,
Gauthier-Villars, Paris, 1924.

Matsusaka, T. 1957. The criteria for algebraic equivalence and the torsion
group. Amer. J. Math. 79 (1957), 53--66.

Mattuck, A.; Tate, J. 1958. On the inequality of Castelnuovo-Severi. Abh.
Math. Sem. Univ. Hamburg 22 (1958), 295--299.

Mazza, C.; Voevodsky, V.; Weibel, C. 2006. Lecture notes on motivic
cohomology. Clay Mathematics Monographs, 2. American Mathematical Society,
Providence, RI; Clay Mathematics Institute, Cambridge, MA, 2006.

Milne, J. 1986. Jacobian varieties. Arithmetic geometry (Storrs, Conn., 1984),
167--212, Springer, New York, 1986.

Milne, J. 2002. Polarizations and Grothendieck's standard conjectures. Ann. of
Math. (2) 155 (2002), no. 2, 599--610.

Milne, J. 2009a. Rational Tate classes. Mosc. Math. J. 9 (2009), no. 1, 111--141,

Milne, J. 2009b. Motives --- Grothendieck's Dream. Mathematical Advances in
Translation, Vol.28, No.3, 193-206, 2009 (Chinese). Reprinted (in English) in
Open problems and surveys of contemporary mathematics. Surveys of Modern
Mathematics, 6. International Press, Somerville, MA; 2013.

Murre, J.; Nagel, J.; Peters, C. 2013. Lectures on the theory of pure motives.
University Lecture Series, 61. American Mathematical Society, Providence, RI, 2013.

N\'{e}ron, A. 1964. Mod\`{e}les minimaux des vari\'{e}t\'{e}s ab\'{e}liennes
sur les corps locaux et globaux. Inst. Hautes \'{E}tudes Sci. Publ.Math. No.
21 1964 128 pp.

Oort, F. and Schappacher, N., Early history of the Riemann Hypothesis in
positive characteristic. This volume.

Quigley, F.D. 1953. On Schubert's Formula, thesis, University of Chicago.

Rankin, R. 1939. Contributions to the theory of Ramanujan's function $\tau(n)$
and similar arithmetical functions. I. The zeros of the function $\sum
_{n=1}^{\infty}\tau(n)/n^{s}$ on the line ${\Re}s=\frac{13}{2}$. II. The order
of the Fourier coefficients of integral modular forms. Proc. Cambridge Philos.
Soc. 35, (1939). 351--372.

Roquette, P. 1953. Arithmetischer Beweis der Riemannschen Vermutung in
Kongruenzfunktionenk\"{o}rpern beliebigen Geschlechts. J. Reine Angew. Math.
191, (1953). 199--252.

Rosen, M. 1986. Abelian varieties over $\mathbb{C}$. Arithmetic geometry
(Storrs, Conn., 1984), 79--101, Springer, New York, 1986.

Saavedra Rivano, N. 1972. Cat\'{e}gories Tannakiennes. Lecture Notes in
Mathematics, Vol. 265. Springer-Verlag, Berlin-New York, 1972.

Schappacher, N. 2006. Seventy years ago: The Bourbaki Congress at El Escorial
and other mathematical (non)events of 1936. The Mathematical Intelligencer,
Special issue International Congress of Mathematicians Madrid August 2006, 8-15.

Scholl, A. 1994. Classical motives. Motives (Seattle, WA, 1991), 163--187,
Proc. Sympos. Pure Math., 55, Part 1, Amer. Math. Soc., Providence, RI, 1994.

Schmidt, F.\ K. 1931. Analytische Zahlentheorie in K\"orpern der
Charakteristik $p$. Math.\ Zeitschr.\ 33 (1931), 1--32. (Habilitationsschrift)

Schmidt, W. 1973. Zur Methode von Stepanov. Acta Arith. 24 (1973), 347--367.

Segel, J. 2009. Recountings. Conversations with MIT mathematicians. Edited by
Joel Segel. A K Peters, Ltd., Wellesley, MA, 2009.

Segre, B. 1937. Intorno ad un teorema di Hodge sulla teoria della base per le
curve di una superficie algebrica. Ann. Mat. pura appl., Bologna, (4) 16,
157-163 (1937).

Serre, J-P. 1954. Cohomologie et g\'{e}om\'{e}trie alg\'{e}brique. Proceedings
of the International Congress of Mathematicians, 1954, Amsterdam, vol. III,
pp. 515--520. Erven P. Noordhoff N.V., Groningen; North-Holland Publishing
Co., Amsterdam, 1956.

Serre, J-P. 1958. Quelques propri\'{e}t\'{e}s des vari\'{e}t\'{e}s
ab\'{e}liennes en caract\'{e}ristique $p$. Amer. J. Math. 80 (1958) 715--739.

Serre, J-P. 1960. Analogues k\"{a}hl\'{e}riens de certaines conjectures de
Weil. Ann. of Math. (2) 71 (1960) 392--394.

Serre, J-P. 1968. Une interpr\'{e}tation des congruences relatives \`{a} la
fonction $\tau$ de Ramanujan. S\'{e}minaire Delange-Pisot-Poitou: 1967/68,
Th\'{e}orie des Nombres, Fasc. 1, Exp. 14 17 pp. Secr\'{e}tariat
math\'{e}matique, Paris

Serre, J-P. 1970. Facteurs locaux des fonctions z\^{e}ta des vari\'{e}t\'{e}s
alg\'{e}briques (d\'{e}finitions et conjectures). S\'{e}minaire
Delange-Pisot-Poitou, 1969/70, n$^{\circ}$19.

Serre, J-P. 2001. Expos\'{e}s de s\'{e}minaires (1950-1999). Documents
Math\'{e}matiques, 1. Soci\'{e}t\'{e} Math\'{e}matique de France, Paris, 2001.

Severi, F. 1903. Sulle corrispondence fra i punti di una curva algebrica e
sopra certi classi di superficie, Mem. R. Accad. Sci. Torino (2) 54 (1903).

Severi, F. 1926. Trattato di geometria algebrica. Vol. 1. N. Zanichelli.

Shafarevich, I. 1994. Basic algebraic geometry. 1. Varieties in projective
space. Second edition. Translated from the 1988 Russian edition and with notes
by Miles Reid. Springer-Verlag, Berlin, 1994.

Steenrod, N. 1957. The work and influence of Professor S. Lefschetz in
algebraic topology. Algebraic geometry and topology. A symposium in honor of
S. Lefschetz, pp. 24--43. Princeton University Press, Princeton, N. J., 1957.

Stepanov, S. 1969. The number of points of a hyperelliptic curve over a finite
prime field. (Russian) Izv. Akad. Nauk SSSR Ser. Mat. 33 (1969), 1171--1181.

Tate, J. 1964. Algebraic cohomology classes, 25 pages. In: Lecture notes
prepared in connection with seminars held at the Summer Institute on Algebraic
Geometry, Whitney Estate, Woods Hole, MA, July 6--July 31, 1964.

Tate, J. 1966a. On the conjectures of Birch and Swinnerton-Dyer and a
geometric analog. S\'{e}minaire Bourbaki, Vol. 9, Exp. No. 306, 415--440.

Tate, J. 1966b. Endomorphisms of abelian varieties over finite fields. Invent.
Math. 2 1966 134--144.

Tate, J. 1968. Classes d'isog\'{e}nie des vari\'{e}t\'{e}s ab\'{e}liennes sur
un corps fini (d'apr\`{e}s T. Honda) S\'{e}minaire Bourbaki 352, (1968/69).

Tate, J. 1994. Conjectures on algebraic cycles in $\ell$-adic cohomology.
Motives (Seattle, WA, 1991), 71--83, Proc. Sympos. Pure Math., 55, Part 1,
Amer. Math. Soc., Providence, RI.

Verdier, J-L. 1972. Ind\'{e}pendance par rapport \`{a} $\ell$ des
polyn\^{o}mes caract\'{e}ristiques des endomorphismes de Frobenius de la
cohomologie $\ell$-adique (d'apr\`{e}s P. Deligne). S\'{e}minaire Bourbaki,
25\`{e}me ann\'{e}e (1972/1973), Exp. No. 423, pp. 98--115.

Voevodsky, V. 1995. A nilpotence theorem for cycles algebraically equivalent
to zero. Internat. Math. Res. Notices 1995, no. 4, 187--198

Weil, A. 1940. Sur les fonctions alg\'{e}briques \`{a} corps de constantes
fini. C. R. Acad. Sci. Paris 210, (1940). 592--594.

Weil, A. 1941. On the Riemann hypothesis in function-fields. Proc. Nat. Acad.
Sci. U. S. A. 27, (1941). 345--347.

Weil, A. 1946. Foundations of Algebraic Geometry. American Mathematical
Society Colloquium Publications, vol. 29. American Mathematical Society, New
York, 1946.

Weil, A. 1948a. Sur les courbes alg\'{e}briques et les vari\'{e}t\'{e}s qui
s'en d\'{e}duisent. Actualit\'{e}s Sci. Ind., no. 1041 = Publ. Inst. Math.
Univ. Strasbourg 7 (1945). Hermann et Cie., Paris, 1948.

Weil, A. 1948b. Vari\'{e}t\'{e}s ab\'{e}liennes et courbes alg\'{e}briques.
Actualit\'{e}s Sci. Ind., no. 1064 = Publ. Inst. Math. Univ. Strasbourg 8
(1946). Hermann \& Cie., Paris, 1948.

Weil, A. 1948c. On some exponential sums. Proc. Nat. Acad. Sci. U. S. A. 34,
(1948). 204--207.

Weil, A. 1949. Numbers of solutions of equations in finite fields. Bull. Amer.
Math. Soc. 55, (1949). 497--508.

Weil, A. 1950. Vari\'{e}t\'{e}s ab\'{e}liennes. Alg\`{e}bre et th\'{e}orie des
nombres. Colloques Internationaux du Centre National de la Recherche
Scientifique, no. 24, pp. 125--127. Centre National de la Recherche
Scientifique, Paris, 1950.

Weil, A. 1952. Fibre spaces in algebraic geometry (Notes by A. Wallace).
University of Chicago, (mimeographed) 48 pp., 1952.

Weil, A. 1955. On algebraic groups of transformations. Amer. J. Math. 77
(1955), 355--391.

Weil, A. 1958. Introduction \`{a} l'\'{e}tude des vari\'{e}t\'{e}s
k\"{a}hl\'{e}riennes. Publications de l'Institut de Math\'{e}matique de
l'Universit\'{e} de Nancago, VI. Actualit\'{e}s Sci. Ind. no. 1267 Hermann,
Paris 1958.

Weil, A., \OE uvres Scientifiques, Collected papers. Springer-Verlag, New
York-Heidelberg, Corrected second printing 1980. (For the first printing
(1979), it may be necessary to subtract 4 from some page numbers.)

Zariski, O. 1952. Complete linear systems on normal varieties and a
generalization of a lemma of Enriques-Severi. Ann. of Math. (2) 55, (1952). 552--592.

\end{small}

\clearpage \pagestyle{empty} \vspace*{-1.8in}\hspace*{-1.3in}\noindent\includegraphics[scale=0.9]{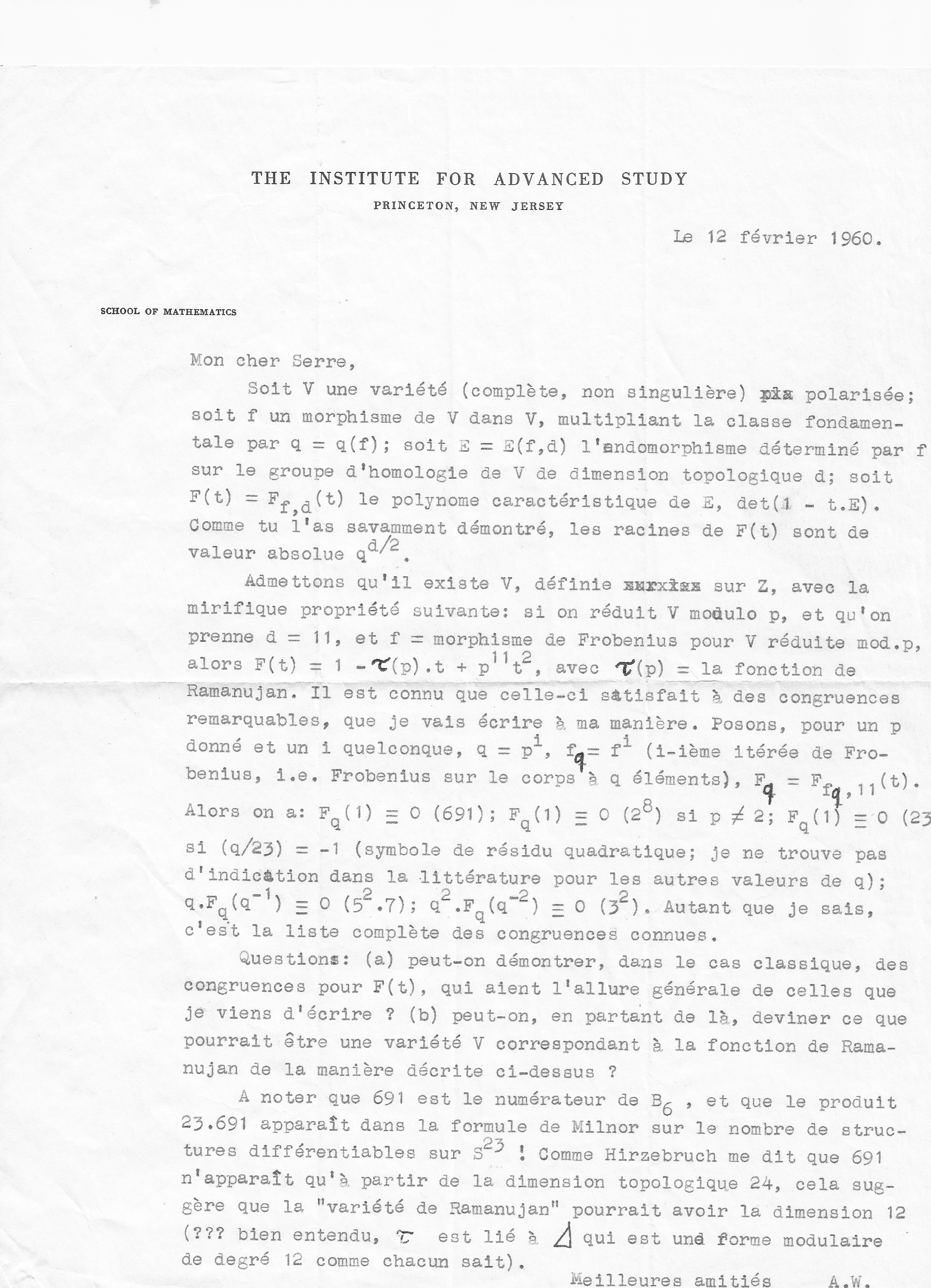}

Letter from Weil to Serre concerning the \textquotedblleft vari\'{e}t\'{e} de
Ramanujan\textquotedblright.

\end{document}

%% file: defa.tex
\contentsmargin{2.0em}
\dottedcontents{section}[2.3em]{}{1.8em}{1pc}
\dottedcontents{subsection}[6.0em]{}{1.8em}{1pc}

\usepackage{titlesec}                                                                            
\titleformat*{\section}{\LARGE\bfseries}
\titleformat*{\subsection}{\Large\itshape}
\titleformat*{\subsubsection}{\large\scshape}%
\titleformat*{\paragraph}{\bfseries\slshape}

\usepackage{enumitem}

\setitemize[1]{label=$\diamond$\hspace{0.07in}}

\setlist{noitemsep}

\renewcommand{\theenumi}{{\alph{enumi}}}

\usepackage{mathptmx}

\usepackage{array}

\usepackage{microtype}

\usepackage{tikz}
\usetikzlibrary{matrix,arrows,positioning,decorations.pathmorphing}
\usepackage{tikz-cd}
 
\usepackage{etex}

\usepackage[notcite,notref]{showkeys}

\usepackage{mathtools}

\usepackage[overload]{textcase}

\newcommand{\eb}[1]{{\itshape\bfseries#1}}
\renewcommand{\emph}{\eb}

\hypersetup{linkcolor=blue,anchorcolor=blue,citecolor=blue,filecolor=blue,urlcolor=blue}
\hypersetup{pdftitle={The Riemann Hypothesis over Finite Fields},pdfauthor={J.S. Milne}}
\hypersetup{pdfsubject={Weil conjectures},pdfkeywords={Standard conjectures}}

\usepackage[textwidth=5.5in,textheight=9.0in,centering]{geometry}

%% file: defs.tex
\newcommand{\da}{\dashrightarrow}

\newcommand{\bcomment}{}
\newcommand{\bfootnotesize}{\begin{footnotesize}}\newcommand\efootnotesize{\end{footnotesize}}
\newcommand{\bquote}{\begin{quote}}\newcommand\equote{\end{quote}}
\newcommand{\bsmall}{\begin{small}}\newcommand\esmall{\end{small}}
\newcommand{\btable}{\begin{table}}\newcommand{\etable}{\end{table}}
\newcommand{\bappendices}{\begin{appendices}}
\newcommand{\eappendices}{\end{appendices}}
\newcommand{\dstyle}{\displaystyle}
\newcommand{\edocument}{\end{document}}

\newcommand{\tstyle}{\textstyle}

\def\1{{1\mkern-7mu1}}

\DeclareMathOperator{\Corr}{Corr}

\DeclareMathOperator{\End}{End}

\DeclareMathOperator{\Gal}{Gal}

\DeclareMathOperator{\Hom}{Hom}
\DeclareMathOperator{\id}{id}

\DeclareMathOperator{\inv}{inv}

\DeclareMathOperator{\Jac}{Jac}
\DeclareMathOperator{\Ker}{Ker}

\DeclareMathOperator{\ord}{ord}

\DeclareMathOperator{\Pic}{\mathrm{Pic}}

\DeclareMathOperator{\Spec}{Spec}

\DeclareMathOperator{\SP}{Sp} 

\DeclareMathOperator{\Tr}{Tr}

